\documentclass{amsart}
\usepackage{amssymb,amsmath,amsthm}
\usepackage{units}
\usepackage{graphicx}
\usepackage{tikz}
\usetikzlibrary{arrows,calc,positioning}
\usepackage{relsize}
\usepackage[noadjust]{cite}
\usepackage{mhsetup}
\usepackage{mathtools}

\numberwithin{equation}{section}
\numberwithin{figure}{section}

\newtheorem{thm}{Theorem}[section]
\newtheorem{lem}[thm]{Lemma}
\newtheorem{prop}[thm]{Proposition}
\newtheorem{coro}[thm]{Corollary}
\newtheorem{coroalph}{Corollary}
  
\newtheorem*{mainthm}{Main Theorem}

\theoremstyle{definition}
\newtheorem{defn}[thm]{Definition}

\theoremstyle{remark}
\newtheorem{rmk}[thm]{Remark}
\newtheorem{eg}[thm]{Example}
\newtheorem{ceg}[thm]{Counterexample}
\newtheorem*{notation}{Notation}
\newtheorem*{note}{Note}
\newtheorem*{aside}{Aside}

\newcommand{\calQ}{\mathcal{Q}}
\newcommand{\calS}{\mathcal{S}}
\newcommand{\bbF}{\mathbb{F}}
\newcommand{\bbH}{\mathbb{H}}
\newcommand{\bbQ}{\mathbb{Q}}
\newcommand{\bbR}{\mathbb{R}}
\newcommand{\bbZ}{\mathbb{Z}}
\newcommand{\id}{\mathrm{id}}
\newcommand{\Emb}{\mathrm{Emb}}
\newcommand{\hocofib}{\mathrm{hocofib}}
\newcommand{\inj}{\mathrm{inj}}
\newcommand{\Cob}{\mathrm{Cob}}
\newcommand{\hconn}{\ensuremath{h\mathrm{conn}}}
\newcommand{\mbar}{\ensuremath{{\,\,\overline{\!\! M\!}\,}}}
\newcommand{\bbar}{\ensuremath{{\,\,\overline{\!\! B\!}\,}}}
\newcommand{\sbar}{\ensuremath{{\,\overline{\! S}}}}
\newcommand{\minus}[2]{#1\!\setminus\! #2}
\newcommand{\geomr}[1]{\lVert #1 \rVert}
\newcommand{\geomrp}[1]{\lVert #1 \rVert_{\star}}
\newcommand{\geomrsk}[2]{\lVert #1 \rVert^{#2}}
\newcommand{\geomrpsk}[2]{\lVert #1 \rVert_{\star}^{#2}}
\newcommand{\rightquotient}{\mathlarger{\mathlarger{/}}}
\newcommand{\tvect}[2]{\ensuremath{\bigl(\negthinspace\begin{smallmatrix}#1\\#2\end{smallmatrix}\bigr)}}
\newcommand{\pair}[2]{\ensuremath{\begin{smallmatrix}#1\\#2\end{smallmatrix}}}
\newcommand{\tikzmono}{%
\tikz[baseline=-0.6ex,>=stealth']{\draw[>->] (0,0)--(2em,0); \useasboundingbox (2.5em,0);}%
}
\newcommand{\incl}[3][right]%
{%
\draw[<-,>=#1 hook] #2 to ($ #2!0.5!#3 $);
\draw[->,>=stealth'] ($ #2!0.5!#3 $) to #3;%
}
\newcommand{\inclusion}[5][right]%
{%
\draw[<-,>=#1 hook] #4 to ($ #4!0.5!#5 $) node[#2,font=\small]{#3};
\draw[->,>=stealth'] ($ #4!0.5!#5 $) to #5;%
}

\begin{document}
\title{Homological stability for oriented configuration spaces}
\author{Martin Palmer}
\date{24\textsuperscript{th} November 2011}
\address{Mathematical Institute\\University of Oxford\\24--29 St Giles'\\Oxford\\OX1 3LB\\UK}
\email{palmer@maths.ox.ac.uk}

\begin{abstract}
\noindent In this paper we prove (integral) homological stability for the sequences of spaces $C_n^+ (M,X)$. These are the spaces of configurations of $n$ points in a connected manifold of dimension at least $2$ which `admits a boundary', with labels in a path-connected space $X$, and with an \emph{orientation} --- an ordering of the points up to even permutations.

They are double covers of the unordered configuration spaces $C_n (M,X)$, and indeed to prove our result we adapt methods from a paper of Randal-Williams, which proves homological stability in the unordered case. Interestingly the oriented configuration spaces stabilise more slowly than the unordered ones: the stability slope we obtain is $\frac13$, compared to $\frac12$ in the unordered case (and these are the best possible slopes in their respective cases).

This result can also be interpreted as homological stability for the unordered configuration spaces with certain twisted $\mathbb{Z} \oplus \mathbb{Z}$-coefficients.
\end{abstract}

\maketitle

%%%%%%%%%%%%%%%%%%%%%%%%%%%%%%%%%%%%%%%%%%%%%%%%%%%%%%%%%%%%%%%%%%%%%%%%%%%%%%%%%%
\section{Introduction}\label{sec:introduction}

\noindent For a manifold $M$ and space $X$, we define the unordered configuration space to be
\begin{equation*}
C_n (M,X)\; =\; \Emb ([n],M) \times_{\Sigma_n} X^n ,
\end{equation*}
where $[n]$ is the discrete space $\lbrace 1, ..., n \rbrace$. This is the space of configurations of $n$ distinct points (or `particles') in $M$, each carrying a label (or `parameter') in $X$. (When $X=pt$ we call $C_n (M,pt) = \Emb ([n],M) / \Sigma_n$ an \emph{unlabelled} configuration space.) The \emph{oriented} configuration space is defined to be the double cover
\begin{equation*}
C_n^+ (M,X)\; =\; \Emb ([n],M) \times_{A_n} X^n
\end{equation*}
of this space, so oriented configurations have an additional global parameter: an ordering of the $n$ points up to even permutations. If $M$ `admits a boundary' there is a natural map $s$ which adds a new point to the configuration near this boundary (see \S \ref{subsec:stabilisation:maps} for precise definitions).
\begin{mainthm}
If $M$ is the interior of a connected manifold-with-boundary of dimension at least $2$, and $X$ is any path-connected space, then
\begin{equation*}
s\colon C_n^+ (M,X) \longrightarrow C_{n+1}^+ (M,X)
\end{equation*}
is an isomorphism on homology up to degree $\frac{n-5}{3}$, and a surjection up to degree $\frac{n-2}{3}$.
\end{mainthm}
\begin{rmk}
If either $M$ or $X$ is not path-connected, then the number of path-components of $C_n^+ (M,X)$ grows unboundedly as $n\to\infty$, so homological stability fails even in degree zero. We also exclude the case of $1$-dimensional manifolds, where homological stability also fails in general: the space $C_n^+ (\bbR ,X)$ deformation retracts onto $X^n \sqcup X^n$.
\end{rmk}
\noindent When such a statement holds for a range of degrees $*\leq \alpha n + c$, we say that the \emph{stability slope} is $\alpha$; so in this case we have homological stability for oriented configuration spaces with a stability slope of $\frac13$.

The underlying method we use for the proof is that of taking `resolutions of moduli spaces', as introduced and studied by Randal-Williams in \cite{Randal:Williams2010}. This method involves considering a semi-simplicial space augmented by the space of interest, where in the `standard' strategy for proving homological stability one would consider a simplicial complex acted on by the group of interest. The method was applied in \cite{Randal:Williams2011} to prove the analogous theorem for unordered configuration spaces, which has a stability slope of $\frac12$. Our method is a modified version of that of \cite{Randal:Williams2011}; however some important complications arise in going from the unordered to the oriented case, which are outlined in \S \ref{sec:sketch:of:the:proof} below. In particular \S \ref{subsec:sketch:of:the:proof:oriented} explains why the stability slope goes from $\frac12$ to $\frac13$ when we apply the techniques of \cite{Randal:Williams2011} to the oriented case.
\begin{rmk}\label{rmk:best:possible:stability:slope}
We note that the stability slope of $\frac13$ is the best possible for oriented configuration spaces (for $\bbZ$-coefficients), as can be seen by the calculations in \cite{Hausmann1978} or \cite{Guest:et:al1996} (see \S \ref{subsec:counterexamples:to:injectivity}).
\end{rmk}

%%%%%%%%%%%%%%%%%%%%%%%%%%%%%%%%%%%%%%%%%%%%%%%%%%%
\subsection{Background}
%%%%%%%%%%%%%%%%%%%%%%%%%%%%%%%%%%%%%%%%%%%%%%%%%%%

A brief history of homology-stability theorems for unordered and oriented configuration spaces is as follows.

%%%%%%%%%%%%%%%%%%%%%%%%%%%%%%%%%
\subsubsection*{Unordered configuration spaces}
%%%%%%%%%%%%%%%%%%%%%%%%%%%%%%%%%

Two special cases which were proved early on are homology-stability for the sequences of symmetric and braid groups, corresponding to $M=\bbR^\infty , \bbR^2$ respectively (and $X=pt$, i.e.\ unlabelled). The result for the symmetric groups is due to Nakaoka \cite{Nakaoka1960}, and the result for the braid groups was proved later by Arnol'd \cite{Arnold1970}. The stability slope obtained in each case was $\frac12$. Using more indirect methods, Segal \cite{Segal1973} proved homology-stability for all Euclidean spaces $M=\bbR^d$ and arbitrary path-connected label-spaces $X$, but this time without an explicit range of stability (see also \cite[\S 3]{Lehrer:Segal2001}). Generalising in a different direction, in \cite{McDuff1975} McDuff proved homology-stability for arbitrary manifolds $M$ (assuming connectivity and that $M$ admits a boundary) but without labels ($X=pt$), also without an explicit stability range. Later, Segal \cite{Segal1979} showed by a different method that in this case we do in fact have a stability slope of $\frac12$, as with the symmetric and braid groups.

The most general result for unordered configuration spaces is due to Randal-Williams \cite[Theorem A]{Randal:Williams2011}\footnote{This is also recalled as Theorem \ref{thm:the:unordered:version} below.}, which allows arbitrary manifolds \emph{and} label-spaces: specifically, he proves homology-stability for $C_n (M,X)$, with a slope of $\frac12$, under the same assumptions on $M$ and $X$ as stated in the Main Theorem above.

A recent result of Church \cite{Church2011} concerning representation stability shows, as a corollary of his main theorem, that rational homology-stability holds (with slope $1$) for unordered, unlabelled configuration spaces where $M$ is allowed to be a \emph{closed} manifold. In this case $M$ does not admit a boundary, and there is no natural map $s$ adding a point to the configuration, but nevertheless stability still holds rationally. The isomorphism in this case is induced by a transfer map which \emph{removes} a point from the configuration. This result is also proved directly in Theorems B and C of \cite{Randal:Williams2011} (although here the increased stability slope of $1$ is only obtained when the manifold has dimension at least $3$).

%%%%%%%%%%%%%%%%%%%%%%%%%%%%%%%%%
\subsubsection*{Oriented configuration spaces}
%%%%%%%%%%%%%%%%%%%%%%%%%%%%%%%%%

Homology-stability for oriented configuration spaces $C_n^+ (M,X)$ has been proved in two special cases: For the alternating groups ($M=\bbR^\infty$, $X=pt$) it can be quickly deduced from a result of Hausmann \cite[page 130]{Hausmann1978}, with a stability slope of $\frac13$, which can be improved to $\frac12$ by taking $\bbZ [\frac13]$-coefficients. For $M$ a compact connected Riemann surface minus a non-empty finite set of points (and $X=pt$), Guest--Kozlowsky--Yamaguchi \cite{Guest:et:al1996} proved homology-stability with a slope of $\frac13$, which again is improved to $\frac12$ by taking $\bbZ [\frac13]$-coefficients. The proofs of \cite{Hausmann1978} and \cite{Guest:et:al1996} involve explicit calculations, using methods which are specific to their respective cases, so do not generalise naturally to all manifolds. The main result of the present paper answers a question in \cite{Guest:et:al1996}, which asks whether their result generalises to arbitrary open manifolds.

In general, for unlabelled oriented configuration spaces, \emph{rational} homology-stability follows from the result of Church mentioned above. It corresponds to stability for the multiplicities of the trivial and alternating representations of $\Sigma_n$ in the rational cohomology of the ordered configuration space $\widetilde{C}_n(M)$. Representation stability for $\widetilde{C}_n(M)$ \cite[Theorem 1]{Church2011} includes multiplicity stability for the trivial representation, and indirectly shows that the multiplicity of the alternating representation is eventually zero (c.f.\ discussion after the statement of Theorem 1 in \cite{Church2011}).

%%%%%%%%%%%%%%%%%%%%%%%%%%%%%%%%%%%%%%%%%%%%%%%%%%%
\subsection{Remarks}
%%%%%%%%%%%%%%%%%%%%%%%%%%%%%%%%%%%%%%%%%%%%%%%%%%%

\begin{rmk}
The Serre spectral sequence for the fibration $\bbZ_2 \to C_n^+ (M,X) \to C_n (M,X)$ implies that
\begin{equation*}
H_* (C_n^+ (M,X);\bbZ ) \; \cong \; H_* (C_n (M,X);\underline{\bbZ \oplus \bbZ}),
\end{equation*}
where the $\underline{\bbZ \oplus \bbZ}$-coefficients on the right are twisted by the action of $\pi_1 C_n (M,X)$ on $\bbZ \oplus \bbZ$ by first projecting to $\bbZ_2$ (corresponding to the index-$2$ subgroup $\pi_1 C_n^+ (M,X)$) and then letting the generator of $\bbZ_2$ act by swapping the two $\bbZ$-summands. So the Main Theorem above is also twisted homological stability for unordered configuration spaces with this sequence of $\pi_1 C_n (M,X)$-modules. We note that in the $M = \bbR^\infty$, $X = pt$ case this sequence of $\Sigma_n$-modules does \emph{not} extend to a (functorial) coefficient system in the sense of \cite{Betley2002}.
\end{rmk}
\begin{rmk}\label{rmk:global:data}
The orientation of a configuration in $C_n^+ (M,X)$ is an example of a \emph{global} parameter on configuration spaces (the labels in $X$ are local parameters); in a sense it is the simplest possible one. It is interesting that homological stability still holds for these spaces, since the `scanning' method of Segal and McDuff does not work in this case: In this method one first uses a `transfer-type' argument to show that, on homology of any degree, the adding-a-point maps $s$ are inclusions of direct summands; then one shows that the colimit of this sequence of maps is finitely generated (c.f.\ proof of Theorem 4.5 in \cite{McDuff1975}). However, for oriented configuration spaces the maps $s$ are \emph{not} always injective on homology (see \S \ref{subsec:counterexamples:to:injectivity} for counterexamples). Arguably, it is the existence of global data in $C_n^+ (M,X)$ which causes this injectivity-on-homology to fail.
\end{rmk}
\begin{rmk}
A nice orientability property of oriented configuration spaces is the following: if $M$ and $X$ are both orientable manifolds, then $C_n^+ (M,X)$ is again an orientable manifold. This is simpler than in the unordered case, where $C_n (M,X)$ is \emph{nonorientable} (exactly) if either
\begin{itemize}
\item $\mathrm{dim} (M)\geq 2$ and $\mathrm{dim} (M) + \mathrm{dim} (X)$ is odd, or
\item $M = S^1$ and $\mathrm{dim} (X)$ and $n$ are even
\end{itemize}
(c.f.\ remark following Proposition A.1 in \cite{Segal1979}).
\end{rmk}

%%%%%%%%%%%%%%%%%%%%%%%%%%%%%%%%%%%%%%%%%%%%%%%%%%%
\subsection{Corollaries}\label{subsec:corollaries}
%%%%%%%%%%%%%%%%%%%%%%%%%%%%%%%%%%%%%%%%%%%%%%%%%%%

The Main Theorem has corollaries for homological stability of certain sequences of groups:
\begin{coroalph}\label{coro:corollary:A}
If $G$ is any discrete group and $S$ is the interior of a connected surface-with-boundary \sbar, then the natural maps
\begin{equation*}
G\wr A\beta_n^S \longrightarrow G\wr A\beta_{n+1}^S \quad \text{and} \quad G\wr A_n \longrightarrow G\wr A_{n+1}
\end{equation*}
are isomorphisms on homology up to degree $\frac{n-5}{3}$ and surjections up to degree $\frac{n-2}{3}$.
\end{coroalph}
\noindent Here $\beta_n^S$ is the braid groups on $n$ strands on the surface $S$, and $A\beta_n^S$ is its alternating subgroup, consisting of those braids whose induced permutation is even. Of course, these corollaries exactly parallel those of the unordered version of the Main Theorem, which concern $G\wr \Sigma_n$ and $G\wr \beta_n^S$. Homological stability for $A_n$ and for $A\beta_n^S$ with \sbar\ compact and orientable were known previously by \cite[Proposition A]{Hausmann1978} (via the relative Hurewicz theorem) and \cite{Guest:et:al1996} respectively. The above corollaries are new (as far as the author is aware) for $G$ nontrivial or for \sbar\ nonorientable or noncompact.

Via the Universal Coefficient Theorem and the Atiyah--Hirzebruch spectral sequence, homological stability for (trivial) $\bbZ$-coefficients implies homological stability for any connective homology theory:
\begin{coroalph}\label{coro:corollary:B}
Under the hypotheses of the Main Theorem, if $h_*$ is a connective homology theory with connectivity $c$, the map
\begin{equation*}
s\colon C_n^+ (M,X) \longrightarrow C_{n+1}^+ (M,X)
\end{equation*}
is an isomorphism on $h_*$ for $*\leq \frac{n-5}{3} + c$ and surjective on $h_*$ for $*\leq \frac{n-2}{3} + c$.
\end{coroalph}

%%%%%%%%%%%%%%%%%%%%%%%%%%%%%%%%%%%%%%%%%%%%%%%%%%%
\subsection*{Layout of the paper}
%%%%%%%%%%%%%%%%%%%%%%%%%%%%%%%%%%%%%%%%%%%%%%%%%%%

In section \ref{sec:definitions:and:set:up} we define all the spaces, semi-simplicial spaces, and maps which will be used later. Section \ref{sec:sketch:of:the:proof} contains an outline of the proof, and explains the differences between the method in the unordered and the oriented cases. The proof itself is contained in sections \ref{sec:two:spectral:sequences}, \ref{sec:the:connectivity:of:the:unpuncturing:map}, and \ref{sec:proof:of:the:main:theorem}; section \ref{sec:two:spectral:sequences} produces some spectral sequences and proves some facts about them, section \ref{sec:the:connectivity:of:the:unpuncturing:map} uses excision to relate the connectivity of two different maps between configuration spaces, and section \ref{sec:proof:of:the:main:theorem} brings this together to prove the Main Theorem. Section \ref{sec:corollaries} establishes the corollaries stated above, and section \ref{sec:failure:of:injectivity} contains a note on the (failure of) injectivity of stabilisation maps on homology.

Some technical constructions have been deferred to the appendices, to avoid lengthy digressions during the proof of the Main Theorem. \ref{appendix:proof:of:the:factorisation:lemma} constructs a factorisation on homology for maps between mapping cones, under fairly general conditions, and \ref{appendix:spectral:sequences:from:dspaces} recalls the details of the construction of various spectral sequences arising from semi-simplicial spaces.

%%%%%%%%%%%%%%%%%%%%%%%%%%%%%%%%%%%%%%%%%%%%%%%%%%%
\subsection*{Acknowledgements}
%%%%%%%%%%%%%%%%%%%%%%%%%%%%%%%%%%%%%%%%%%%%%%%%%%%

I would firstly like to thank my supervisor, Ulrike Tillmann, for her invaluable advice and support during my studies so far. I would also like to thank Oscar Randal-Williams for numerous very helpful discussions and suggestions, in particular the idea leading to the `factorisation lemma' used in the second half of the proof. I am also grateful to Alejandro Adem for pointing out the reference \cite{Hausmann1978}, and to Graeme Segal and Christopher Douglas for constructive comments on an earlier version of this paper, which formed my `Transfer of Status' dissertation \cite{TransferReport} (part of my PhD thesis).
%%%%%%%%%%%%%%%%%%%%%%%%%%%%%%%%%%%%%%%%%%%%%%%%%%%%%%%%%%%%%%%%%%%%%%%%%%%%%%%%%%

%%%%%%%%%%%%%%%%%%%%%%%%%%%%%%%%%%%%%%%%%%%%%%%%%%%%%%%%%%%%%%%%%%%%%%%%%%%%%%%%%%
\section{Definitions and set-up}\label{sec:definitions:and:set:up}

Firstly we mention two general notational conventions: A connected manifold $M$ with $k$ points removed will be denoted $M_k$ (since it is connected, its homeomorphism type is independent of \emph{which} $k$ points are removed). The symbol \tikzmono will be reserved for the canonical inclusion of the codomain of a map into its mapping cone:
\begin{center}
\begin{tikzpicture}
[x=4em,>=stealth']
\node (Y) at (0,0) {$Y$};
\node (Z) at (1,0) {$Z$};
\node (Cf) at (2,0) {$Cf.$};
\draw[->] (Y) to node[above=-2pt,font=\footnotesize]{$f$} (Z);
\draw[>-] (Z) to (1.5,0);
\draw[->] (1.5,0) to (Cf);
\end{tikzpicture}
\end{center}

%%%%%%%%%%%%%%%%%%%%%%%%%%%%%%%%%%%%%%%%%%%%%%%%%%%
\subsection{Configuration spaces}\label{subsec:configuration:spaces}
%%%%%%%%%%%%%%%%%%%%%%%%%%%%%%%%%%%%%%%%%%%%%%%%%%%

\begin{defn}
For a manifold $M$ and space $X$, we define the \emph{ordered configuration space} to be
\begin{equation*}
\widetilde{C}_n (M,X) \coloneqq \Emb ([n],M) \times X^n ,
\end{equation*}
where $[n]$ is the $n$-point discrete space. This is the space of ordered, distinct points (`particles') in $M$, each carrying a label (or parameter) in $X$. The symmetric group acts diagonally on this space, permuting the points along with their labels, and we define the \emph{unordered configuration space} to be the quotient
\begin{equation*}
C_n (M,X) \coloneqq \widetilde{C}_n (M,X) / \Sigma_n .
\end{equation*}
If instead we just quotient out by the action of the alternating group, we obtain the \emph{oriented configuration space}
\begin{equation*}
C_n^+ (M,X) \coloneqq \widetilde{C}_n (M,X) / A_n .
\end{equation*}
\end{defn}
\begin{notation}
We will denote elements of ordered, oriented, unordered configuration spaces respectively by $\left( \pair{p_1}{x_1} \cdots \pair{p_n}{x_n} \right)$, $\left[ \pair{p_1}{x_1} \cdots \pair{p_n}{x_n} \right]$, $\left\lbrace \pair{p_1}{x_1} \cdots \pair{p_n}{x_n} \right\rbrace$, where $p_i \in M$ and $x_i \in X$. So square brackets denote the equivalence class under even permutations of the columns. The orientation-reversing automorphism
\begin{equation*}
\left[ \pair{p_1}{x_1} \cdots \pair{p_n}{x_n} \right] \quad \mapsto \quad \left[ \pair{p_1}{x_1} \cdots \pair{p_{n-2}}{x_{n-2}} \pair{p_n}{x_n} \pair{p_{n-1}}{x_{n-1}} \right]
\end{equation*}
of $C_n^+ (M,X)$ will be denoted by $\nu$. We will often abbreviate these spaces to $\widetilde{C}_n (M)$, $C_n^+ (M)$, and $C_n (M)$ when the space of labels is clear, to avoid cluttering our notation.
\end{notation}

%%%%%%%%%%%%%%%%%%%%%%%%%%%%%%%%%%%%%%%%%%%%%%%%%%%
\subsection{Adding a point to a configuration space}\label{subsec:stabilisation:maps}
%%%%%%%%%%%%%%%%%%%%%%%%%%%%%%%%%%%%%%%%%%%%%%%%%%%

To add a point to a configuration on $M$, there needs to be somewhere `at infinity' from which to push in this new configuration point. An appropriate condition is to `admit a boundary':
\begin{defn}
We say that $M$ \emph{admits a boundary} if it is the interior of some manifold-with-boundary $\mbar$. Note that we do not require $\mbar$ to be compact.
\end{defn}
\noindent When $M$ admits a boundary, there is a natural adding-a-point map, as follows:
\begin{defn}
Suppose that $M = \mathrm{int}(\mbar)$, where $\mbar$ is a manifold-with-boundary, and choose a point $b_0 \in \partial \mbar$. Let $B_0 = \partial_0 \mbar$ be the boundary-component containing $b_0$. Choose also a basepoint $x_0 \in X$. We initially define the adding-a-point map at the level of \emph{ordered} configuration spaces to be
\begin{equation*}
\left( \pair{p_1}{x_1} \cdots \pair{p_n}{x_n} \right) \quad \mapsto \quad \left( \pair{p_1}{x_1} \cdots \pair{p_n}{x_n} \pair{b_0}{x_0} \right) .
\end{equation*}
This is a map $\widetilde{C}_n (M,X) \to \widetilde{C}_{n+1} (M^\prime ,X)$, where $M^\prime$ is $\mbar$ with an open collar attached at $B_0$:
\begin{equation*}
M^\prime = \mbar \cup_{B_0} \bigl( B_0 \times [0,1) \bigr) .
\end{equation*}
Choosing a canonical homeomorphism $\phi\colon M^\prime \cong M$ (with support contained in a small neighbourhood of $B_0$), which pushes this collar back into $M$, we obtain a map
\begin{equation*}
s\colon \widetilde{C}_n (M,X) \longrightarrow \widetilde{C}_{n+1} (M,X).
\end{equation*}
This process is illustrated in Figure \ref{fig:adding:a:point:map}. The map $s$ is equivariant w.r.t.\ the standard inclusion $\Sigma_n \hookrightarrow \Sigma_{n+1}$ (and hence also w.r.t.\ $A_n \hookrightarrow A_{n+1}$), so it descends to maps
\begin{align*}
&& s\colon C_n (M,X) &\longrightarrow C_{n+1} (M,X) && \\
&\text{and}& s\colon C_n^+ (M,X) &\longrightarrow C_{n+1}^+ (M,X). &&
\end{align*}
\end{defn}
\begin{figure}[ht]
\centering
\begin{tikzpicture}
[x=1mm,y=1mm,>=stealth']
\node at (-4,0) [color=gray,font=\small] {$\cdots$};
\fill[black!15] (15,4) arc (90:270:2 and 4) -- (0,-4) arc (270:90:2 and 4) -- cycle;
\draw (0,4)--(15,4);
\draw (0,-4)--(15,-4);
\draw (15,0) ellipse (2 and 4);
\draw[dashed] (15,4)--(20,4);
\draw[dashed] (15,-4)--(20,-4);
\draw[dashed] (20,0) ellipse (2 and 4);
\node[circle,fill,inner sep=1pt] at (13,0) {};
\node at (11,0) [font=\footnotesize] {$b_0$};
\draw[->,thick] (26,0)--(33,0);
\begin{scope}[xshift=43mm]
\node at (-4,0) [color=gray,font=\small] {$\cdots$};
\fill[black!15] (15,4) arc (90:270:2 and 4) -- (0,-4) arc (270:90:2 and 4) -- cycle;
\draw (0,4)--(15,4);
\draw (0,-4)--(15,-4);
\draw (15,0) ellipse (2 and 4);
\draw (15,4)--(20,4);
\draw (15,-4)--(20,-4);
\draw (20,0) ellipse (2 and 4);
\node[circle,fill,inner sep=1pt] at (13,0) {};
\draw[->] (17.5,2)--(9,2);
\draw[->] (17.5,-2)--(9,-2);
\end{scope}
\draw[->,thick] (69,0)--(76,0);
\begin{scope}[xshift=86mm]
\node at (-4,0) [color=gray,font=\small] {$\cdots$};
\fill[black!15] (10,4) arc (90:270:2 and 4) -- (0,-4) arc (270:90:2 and 4) -- cycle;
\draw (0,4)--(15,4);
\draw (0,-4)--(15,-4);
\draw (15,0) ellipse (2 and 4);
\draw[dashed] (10,0) ellipse (2 and 4);
\node[circle,fill,inner sep=1pt] at (8,0) {};
\end{scope}
\end{tikzpicture}
\caption{The adding-a-point map $s$. The original configuration is contained in the interior of the shaded region in each picture.}\label{fig:adding:a:point:map}
\end{figure}
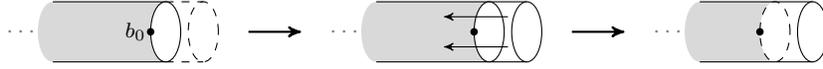
\begin{notation}
We will generally refer to the adding-a-point maps $s$ as \emph{stabilisation maps}, since these are the maps with respect to which the unordered and oriented configuration spaces stabilise. When it is necessary to keep track of the number of points in a configuration, we write $s=s_n$ for the map which adds the $(n+1)$-st point to a configuration. In the oriented case, we define $-s \coloneqq \nu\circ s$ (and $+s \coloneqq s$). So $-s$ just takes the opposite orientation convention in its definition, sending $\left[ \pair{p_1}{x_1} \cdots \pair{p_n}{x_n} \right]$ to $\left[ \pair{p_1}{x_1} \cdots \pair{p_{n-1}}{x_{n-1}} \pair{b_0}{x_0} \pair{p_n}{x_n} \right]$ instead of $\left[ \pair{p_1}{x_1} \cdots \pair{p_n}{x_n} \pair{b_0}{x_0} \right]$.
\end{notation}
\begin{rmk}
Up to homotopy, the stabilisation map $s$ depends only on the choice of boundary-component $B_0$, and the choice of path-component of $X$ containing $x_0$. We will later only consider the case when $X$ is path-connected, so $s$ will only depend on which `end' of the manifold the new configuration point is pushed in from.
\end{rmk}
\begin{rmk}\label{rmk:s:and:minus:s:difference}
Since $\pm s$ only differ by an automorphism of their common codomain, they have exactly the same properties w.r.t.\ injectivity- and surjectivity-on-homology, so they are interchangeable for the purposes of homology-stability.
\end{rmk}

%%%%%%%%%%%%%%%%%%%%%%%%%%%%%%%%%%%%%%%%%%%%%%%%%%%
\subsection{Semi-simplicial spaces}
%%%%%%%%%%%%%%%%%%%%%%%%%%%%%%%%%%%%%%%%%%%%%%%%%%%

In general, a semi-simplicial space (which we will call a $\Delta$-space) is a diagram of the form
\begin{center}
\begin{tikzpicture}
[x=1mm,y=1mm,>=stealth']
\node at (0,0) {$\cdots$};
\node at (20,0) {$Y_1$};
\node at (40,0) {$Y_0$};
\draw[->] (5,2)--(15,2);
\draw[->] (5,0)--(15,0);
\draw[->] (5,-2)--(15,-2);
\draw[->] (25,1)--(35,1);
\draw[->] (25,-1)--(35,-1);
\end{tikzpicture}
\end{center}
where the `face maps' $d_i \colon Y_k \to Y_{k-1}$ ($1\leq i\leq k+1$) satisfy the simplicial identities $d_i d_j = d_{j-1} d_i$ whenever $i<j$. The $\Delta$-space as a whole is denoted $Y_\bullet$. An \emph{augmented} $\Delta$-space is a diagram of the form
\begin{center}
\begin{tikzpicture}
[x=1mm,y=1mm,>=stealth']
\node at (0,0) {$\cdots$};
\node at (20,0) {$Y_1$};
\node at (40,0) {$Y_0$};
\node at (60,0) {$Y_{-1}$};
\draw[->] (5,2)--(15,2);
\draw[->] (5,0)--(15,0);
\draw[->] (5,-2)--(15,-2);
\draw[->] (25,1)--(35,1);
\draw[->] (25,-1)--(35,-1);
\draw[->] (45,0)--(55,0);
\end{tikzpicture}
\end{center}
where again the face maps satisfy the simplicial identities. In other words this is a $\Delta$-space together with an `augmentation map' $Y_0 \to Y_{-1}$ which equalises the two face maps $d_1, d_2 \colon Y_1 \rightrightarrows Y_0$. A map of (augmented) $\Delta$-spaces is a collection of maps, one for each level $k$, which commutes with $d_i$ for each $i$.

The (thick) geometric realisation of a $\Delta$-space $Y_\bullet$ is $\geomr{Y_\bullet} = \bigl( \coprod_{k\geq 0} Y_k \times \Delta^k \bigr) \rightquotient\! \sim$, where $\sim$ is the equivalence relation generated by the face relations $(d_i (y),z) \sim (y, \delta_i (z))$, where $\delta_i$ is the inclusion of the $i$th face of $\Delta^{k+1}$. If $Y_\bullet$ is an augmented $\Delta$-space, there is a unique composition of face maps $Y_k \to Y_{-1}$ for each $k$. These fit together to give an induced map $\geomr{Y_\bullet} \to Y_{-1}$, where $\geomr{Y_\bullet}$ is the geometric realisation of the \emph{non-augmented} part of $Y_\bullet$.
\begin{defn}
A $\Delta$-space $Y_\bullet$ with an augmentation to $Y_{-1}$ such that the induced map $\geomr{Y_\bullet} \to Y_{-1}$ is $n$-connected is called an `$n$-resolution' of $Y_{-1}$ in the terminology of \cite{Randal:Williams2010} and \cite{Randal:Williams2011}.
\end{defn}

%%%%%%%%%%%%%%%%%%%%%%%%%%%%%%%%%
\subsubsection{A configuration $\Delta$-space.}
%%%%%%%%%%%%%%%%%%%%%%%%%%%%%%%%%

We now extend the oriented configuration space $C_n^+ (M,X)$ so that it is the $(-1)$st level of an augmented $\Delta$-space.
\begin{defn}
The augmented $\Delta$-space $C_n^+ (M,X)^\bullet$ is defined as follows: The elements of the space of $i$-simplices $C_n^+ (M,X)^i$ are configurations $\left[ \pair{p_1}{x_1} \cdots \pair{p_n}{x_n} \right]$, together with an \emph{ordering} of $i+1$ of the pairs $\left( \pair{p_i}{x_i} \right)$. In particular $C_n^+ (M,X)^{-1}$ is just $C_n^+ (M,X)$, and $C_n^+ (M,X)^0$ consists of (oriented, labelled) configurations with one of the points marked out as `special'. The face map $d_j$ is given by forgetting the $j$th element of the $(i+1)$-ordering; in particular the augmentation map is the map $C_n^+ (M,X)^0 \to C_n^+ (M,X)$ which forgets which point is `special'.
\end{defn}
\begin{rmk}\label{rmk:n:minus:1:resolution}
We will show later (see Corollary \ref{coro:n:minus:1:resolution}) that $C_n^+ (M,X)^\bullet$ is an $(n-1)$-resolution of $C_n^+ (M,X)$.
\end{rmk}
\begin{note}
The definition of $\pm s$ clearly extends to each level $C_n^+ (M,X)^i$ and commutes with the face maps, so we have maps of augmented $\Delta$-spaces:
\begin{equation*}
C_n^+ (M,X)^\bullet \xrightarrow{\pm s^\bullet} C_{n+1}^+ (M,X)^\bullet .
\end{equation*}
\end{note}
\noindent As before, we will often abbreviate the augmented $\Delta$-space $C_n^+ (M,X)^\bullet$ to $C_n^+ (M)^\bullet$.

%%%%%%%%%%%%%%%%%%%%%%%%%%%%%%%%%%%%%%%%%%%%%%%%%%%
\subsection{Maps between configuration spaces}
%%%%%%%%%%%%%%%%%%%%%%%%%%%%%%%%%%%%%%%%%%%%%%%%%%%

We will make use of the following maps between configuration spaces in the proof of the Main Theorem.

%%%%%%%%%%%%%%%%%%%%%%%%%%%%%%%%%
\subsubsection{$\varepsilon_n$ and $a_n$.}
%%%%%%%%%%%%%%%%%%%%%%%%%%%%%%%%%

These come automatically from the structure of the augmented $\Delta$-space $C_n^+ (M,X)^\bullet$: $a_n$ denotes the augmentation map
\begin{equation*}
a_n \colon C_n^+ (M,X)^0 \longrightarrow C_n^+ (M,X),
\end{equation*}
which forgets which point is the `special' point, and $\varepsilon_n$ is the induced map
\begin{equation*}
\varepsilon_n \colon \geomr{C_n^+ (M,X)^\bullet} \longrightarrow C_n^+ (M,X)
\end{equation*}
from the geometric realisation of the unaugmented part of the $\Delta$-space to the augmentation.

\begin{aside}[\emph{Puncturing $M$}]
Recall from the beginning of the section that $M_1$, $M_k$ denote the connected manifold $M$ with any point, or more generally any $k$ points, removed. Since $M$ is connected, the manifolds resulting from removing different choices of a set of $k$ points can all be (non-canonically) identified. So where necessary we may assume that $M_1$ means $M$ with a point \emph{near $B_0$} (in the sense of the definition of the stabilisation map) removed. It is also implicitly assumed that we remember the inclusion $M_k \hookrightarrow M$, as well as the abstract manifold $M_k$.
\end{aside}

%%%%%%%%%%%%%%%%%%%%%%%%%%%%%%%%%
\subsubsection{$p_n$ and $u_n$.}\label{subsubsec:pn:and:un}
%%%%%%%%%%%%%%%%%%%%%%%%%%%%%%%%%

The maps
\begin{equation*}
C_n^+ (M) \xrightarrow{p_n} C_n^+ (M_1) \xrightarrow{u_n} C_n^+ (M)
\end{equation*}
`puncture' and `unpuncture' the manifold $M$ respectively. The second map is easiest to describe: it is just induced by the inclusion $M_1 \hookrightarrow M$. The puncturing map $p_n$ is defined similarly to the stabilisation map. Let $M^{(1)}$ be \mbar\ with an open collar attached at $B_0$, and then punctured at $b_0$:
\begin{equation*}
M^{(1)} = \mbar \cup_{B_0} \bigl( B_0 \times [0,1) \bigr) \setminus \lbrace b_0 \rbrace .
\end{equation*}
Then $p_n$ is induced by the inclusion $M \hookrightarrow M^{(1)}$ and the canonical homeomorphism $\phi|_{M^{(1)}} \colon M^{(1)} \cong M_1$ (from the definition of the stabilisation map) which pushes the collar back into $M$.
\begin{rmk}\label{rmk:up:homotopic:to:id:1}
The composition $u_n \circ p_n$ is homotopic to the identity, since it just pushes the configuration away from $B_0$ slightly.
\end{rmk}

%%%%%%%%%%%%%%%%%%%%%%%%%%%%%%%%%
\subsubsection{$\pi_{n,i}$ and $j_{n,i}$.}
%%%%%%%%%%%%%%%%%%%%%%%%%%%%%%%%%

The projection
\begin{equation*}
\pi_{n,i} \colon C_n^+ (M)^i \longrightarrow \widetilde{C}_{i+1} (M)
\end{equation*}
forgets all but the $(i+1)$-ordered points of the configuration in $C_n^+ (M)^i$. It clearly commutes with the stabilisation map: $\pi_{n+1,i} \circ s_n = \pi_{n,i}$.

This is a fibre bundle, with fibre homeomorphic to $C_{n-i-1}^+ (M_{i+1})$ when ${i\leq n-3}$. This is closely analogous to the fibre bundle constructed by Fadell--Neuwirth in \cite[Theorem 3]{Fadell:Neuwirth1962}, and the fact that \emph{this} is a fibre bundle is proved in detail as Lemma 1.26 in \cite[page 26]{Kassel:Turaev2008}, so we refer to this for a detailed exposition. To find a trivialising neighbourhood for $\left( \pair{p_1}{x_1} \cdots \pair{p_{i+1}}{x_{i+1}} \right) \in \widetilde{C}_{i+1} (M)$, one just needs to choose pairwise disjoint open neighbourhoods for the points $p_1, ..., p_{i+1} \in M$. The condition $i\leq n-3$ is to ensure that the fibre is path-connected; in the cases $i=n-2$ and $i=n-1$, the fibre is $M_{n-1} \times X \times [2]$ and $[2]$ respectively.

Pick a basepoint $\left( \pair{m_1}{x_0} \cdots \pair{m_{i+1}}{x_0} \right) \in \widetilde{C}_{i+1} (M)$. We define $j_{n,i}$ to be the inclusion of the fibre:
\begin{equation*}
j_{n,i} \colon C_{n-i-1}^+ (M_{i+1}) = C_{n-i-1}^+ (M \setminus \lbrace m_1, ..., m_{i+1} \rbrace) \;\hookrightarrow\; C_n^+ (M)^i .
\end{equation*}
In identifying the fibre abstractly with $C_{n-i-1}^+ (M_{i+1})$, we have implicitly chosen a convention for combining the orientation of $\left[ \pair{p_1}{x_1} \cdots \pair{p_{n-i-1}}{x_{n-i-1}} \right] \in C_{n-i-1}^+ (M_{i+1})$ with the ordering $\left( \pair{m_1}{x_0} \cdots \pair{m_{i+1}}{x_0} \right)$ to induce an orientation of all $n$ points. We declare this convention to be $\left[ \pair{m_1}{x_0} \cdots \pair{m_{i+1}}{x_0} \pair{p_1}{x_1} \cdots \pair{p_{n-i-1}}{x_{n-i-1}} \right]$, which completes the definition of $j_{n,i}$.

So abstractly $j_{n,i}$ is a map which replaces $i+1$ punctures with $i+1$ new configuration points, which are additionally given an $(i+1)$-ordering. The new points are all labelled by $x_0 \in X$, and the orientation of the new, larger configuration is given by the convention stated above.

\begin{rmk}
Due to our choice of orientation convention for $j_{n,i}$, these maps commute with stabilisation maps, and we have a map of fibre bundles:
\begin{equation}\label{eqn:j:commuting:with:s}
\centering
\begin{split}
\begin{tikzpicture}
[x=1.5mm,y=1.2mm,>=stealth']
\node (tl) at (0,10) {$C_{n-i}^+ (M_{i+1})$};
\node (tm) at (25,10) {$C_{n+1}^+ (M)^i$};
\node (bl) at (0,0) {$C_{n-i-1}^+ (M_{i+1})$};
\node (bm) at (25,0) {$C_n^+ (M)^i$};
\node (r) at (50,5) {$\widetilde{C}_{i+1} (M)$};
\draw [<-,>=right hook] (tl) to ($ (tl.east)!0.5!(tm.west) $) node[above,font=\small]{$j_{n+1,i}$};
\draw [->] ($ (tl.east)!0.5!(tm.west) $) to (tm);
\draw [<-,>=right hook] (bl) to ($ (bl.east)!0.5!(bm.west) $) node[below,font=\small]{$j_{n,i}$};
\draw [->] ($ (bl.east)!0.5!(bm.west) $) to (bm);
\draw [->] (bl) to node[left,font=\small]{$s$} (tl);
\draw [->] (bm) to node[right,font=\small]{$s^i$} (tm);
\draw [->>] (tm) to node[above,font=\footnotesize]{$\pi_{n+1,i}$} (r);
\draw [->>] (bm) to node[below,font=\footnotesize]{$\pi_{n,i}$} (r);
\end{tikzpicture}
\end{split}
\end{equation}
\end{rmk}
\begin{rmk}\label{rmk:factorisation:of:sn}
The composition
\begin{equation*}
C_n^+ (M) \xrightarrow{p_n} C_n^+ (M_1) \xrightarrow{j_{n+1,0}} C_{n+1}^+ (M)^0 \xrightarrow{a_{n+1}} C_{n+1}^+ (M)
\end{equation*}
sends $\left[ \pair{p_1}{x_1} \cdots \pair{p_n}{x_n} \right]$ to $\left[ \pair{b_0}{x_0} \pair{\bar{p}_1}{x_1} \cdots \pair{\bar{p}_n}{x_n} \right]$, where $\bar{p}_i = \phi (p_i)$ is $p_i$ pushed slightly away from $B_0$ if it is near $B_0$. Hence this is a factorisation of $(-1)^n s_n$.

This factorisation will be key to the proof of the Main Theorem, and the appearance of $(-1)^n$ here is in a sense where the extra complication (compared to the unordered case) comes from --- and why we only obtain a stability slope of $\frac13$.
\end{rmk}

%%%%%%%%%%%%%%%%%%%%%%%%%%%%%%%%%%%%%%%%%%%%%%%%%%%
\subsection{Relative configuration spaces}\label{subsec:relative:config:spaces}
%%%%%%%%%%%%%%%%%%%%%%%%%%%%%%%%%%%%%%%%%%%%%%%%%%%

\begin{defn}
We define the \emph{relative configuration space} to be the mapping cone of the (positive) stabilisation map:
\begin{equation*}
R_n^+ (M,X) \coloneqq \hocofib \bigl( C_n^+ (M,X) \xrightarrow{s_n} C_{n+1}^+ (M,X) \bigr) .
\end{equation*}
Similarly, $R_n^+ (M,X)^i$ is defined to be the mapping cone of the stabilisation map $s_n^i$ between the $i$th levels of the corresponding $\Delta$-spaces. Since the face maps commute exactly with the stabilisation maps, they induce relative face maps which give $\lbrace R_n^+ (M,X)^i \rbrace_{i\geq -1}$ the structure of an augmented $\Delta$-space $R_n^+ (M,X)^\bullet$. Again, we will usually abbreviate the notation to $R_n^+ (M)$ and $R_n^+ (M)^\bullet$ when $X$ is understood.
\end{defn}

%%%%%%%%%%%%%%%%%%%%%%%%%%%%%%%%%%%%%%%%%%%%%%%%%%%
\subsection{Maps between relative configuration spaces}
%%%%%%%%%%%%%%%%%%%%%%%%%%%%%%%%%%%%%%%%%%%%%%%%%%%

All our maps $\widetilde{f} \colon R_n^+ (M)^i \longrightarrow R_{n^\prime}^+ (M^\prime)^{i^\prime}$ between relative configuration spaces will be induced by maps defined on the non-relative configuration spaces:
\begin{center}
\begin{tikzpicture}
[x=1.5mm,y=1.5mm,>=stealth']
\node (tl) at (0,10) {$C_n^+ (M)^i$};
\node (tm) at (20,10) {$C_{n+1}^+ (M)^i$};
\node (tr) at (40,10) {$R_n^+ (M)^i$};
\node (bl) at (0,0) {$C_{n^\prime}^+ (M^\prime)^{i^\prime}$};
\node (bm) at (20,0) {$C_{n^\prime +1}^+ (M^\prime)^{i^\prime}$};
\node (br) at (40,0) {$R_{n^\prime}^+ (M^\prime)^{i^\prime}$};
\draw [->] (tl) to node[above,font=\small]{$s$} (tm);
\draw [->] (tm) to (tr);
\draw [->] (bl) to node[below,font=\small]{$s$} (bm);
\draw [->] (bm) to (br);
\draw [->] (tl) to node[left,font=\small]{$f$} (bl);
\draw [->] (tm) to node[right,font=\small]{$f$} (bm);
\draw [->] (tr) to node[right,font=\small]{$\widetilde{f}$} (br);
\node at (10,5) [font=\small] {$H_f$};
\end{tikzpicture}
\end{center}
Note that $\widetilde{f}$ (even up to homotopy) depends on the non-relative maps $f$, \emph{and} the homotopy $H_f$ chosen to fill the square.

%%%%%%%%%%%%%%%%%%%%%%%%%%%%%%%%%
\subsubsection{$\widetilde{\jmath}_{n,i}$, $\widetilde{a}_n$, and $\widetilde{u}_n$.}\label{subsubsec:relative:maps:1}
%%%%%%%%%%%%%%%%%%%%%%%%%%%%%%%%%

We define relative versions of the inclusion-of-the-fibre, augmentation, and unpuncturing maps as follows:

\begin{defn}
The maps $j_{n,i}$, $a_n$, and $u_n$ commute exactly with stabilisation maps, so we may define
\begin{align*}
&& \widetilde{\jmath}_{n,i} &\colon R_{n-i-1}^+ (M_{i+1}) \longrightarrow R_n^+ (M)^i , && \\
&& \widetilde{a}_n &\colon R_n^+ (M)^0 \longrightarrow R_n^+ (M) , && \\
&\text{and}& \widetilde{u}_n &\colon R_n^+ (M_1) \longrightarrow R_n^+ (M) &&
\end{align*}
as explained above, taking the homotopy $H_f$ to be the \emph{constant} homotopy in each case.
\end{defn}

%%%%%%%%%%%%%%%%%%%%%%%%%%%%%%%%%
\subsubsection{$\widetilde{p}_n$ and relative stabilisation maps.}\label{subsubsec:relative:stabilisation:maps}
%%%%%%%%%%%%%%%%%%%%%%%%%%%%%%%%%

We now define relative versions of the puncturing map $p_n$, the stabilisation map $s\colon C_n(M) \to C_{n+1}(M)$ and the (negative) iterated stabilisation map $-s^2 = \nu\circ s\circ s \colon C_n^+(M) \to C_{n+2}^+(M)$.
\begin{notation}
To differentiate clearly between the unordered and oriented cases, we will temporarily (for Definition \ref{defn:relative:stab:maps} and Remark \ref{rmk:relative:stab:maps} below) use the following notation when we want to emphasise which case we are dealing with: $\breve{s}$ denotes the stabilisation map between \emph{u}nordered configuration spaces, and $\mathring{s}$ denotes the stabilisation map between \emph{o}riented configuration spaces. So we want to define relative versions of $p_n$, $\breve{s}$, and $-\mathring{s}^2$.
\end{notation}
\begin{defn}\label{defn:relative:stab:maps}
Embed $\bbH = \lbrace (x,y)\in\bbR^2 | y\geq 0 \rbrace$ in $\mbar$, in a neighbourhood of the boundary-component $B_0$, so that $b_0$ is identified with $(0,0)$, and such that the homeomorphism $\phi\colon M^\prime \cong M$ from the definition of the stabilisation map restricts to $(x,y) \mapsto (x,y+1)$ on $\bbH$. So on $\bbH$, the stabilisation map pushes points up by $1$ and adds a new point at $(0,1)$.

Define the self-homotopies $(12)\colon \breve{s}^2 \simeq \breve{s}^2$ and $(123), (132)\colon -\mathring{s}^3 \simeq -\mathring{s}^3$ to fix the original configuration, and move the new configuration points around on $\bbH$ as illustrated below:
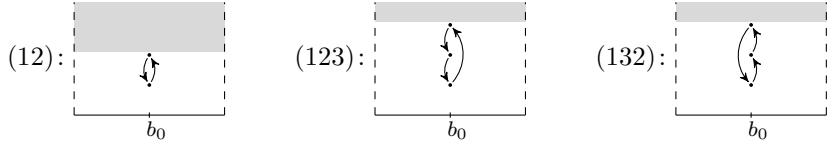
\begin{center}
\begin{tikzpicture}
[x=1mm,y=1mm,>=stealth']
\fill[black!15] (-10,8mm+1pt) rectangle (10,15);
\draw (-10,0)--(10,0);
\draw[dashed] (-10,0)--(-10,15);
\draw[dashed] (10,0)--(10,15);
\node[circle,fill,inner sep=0.5pt] at (0,4) {};
\node[circle,fill,inner sep=0.5pt] at (0,8) {};
\draw [->] ($ (0,4) + (60:0.5) $) to[out=60,in=300] ($ (0,8) + (300:0.5) $);
\draw [->] ($ (0,8) + (240:0.5) $) to[out=240,in=120] ($ (0,4) + (120:0.5) $);
\draw (0,-1pt)--(0,1pt);
\node at (1,-2) [font=\footnotesize] {$b_0$};
\node at (-16,7.5) {$\phantom{3}(12)\! :$};
\begin{scope}[xshift=40mm]
\fill[black!15] (-10,12mm+1pt) rectangle (10,15);
\draw (-10,0)--(10,0);
\draw[dashed] (-10,0)--(-10,15);
\draw[dashed] (10,0)--(10,15);
\node[circle,fill,inner sep=0.5pt] at (0,4) {};
\node[circle,fill,inner sep=0.5pt] at (0,8) {};
\node[circle,fill,inner sep=0.5pt] at (0,12) {};
\draw [->] ($ (0,12) + (240:0.5) $) to[out=240,in=120] ($ (0,8) + (120:0.5) $);
\draw [->] ($ (0,8) + (240:0.5) $) to[out=240,in=120] ($ (0,4) + (120:0.5) $);
\draw [->] ($ (0,4) + (50:0.5) $) to[out=50,in=310] ($ (0,12) + (310:0.5) $);
\draw (0,-1pt)--(0,1pt);
\node at (1,-2) [font=\footnotesize] {$b_0$};
\node at (-16,7.5) {$(123)\! :$};
\end{scope}
\begin{scope}[xshift=80mm]
\fill[black!15] (-10,12mm+1pt) rectangle (10,15);
\draw (-10,0)--(10,0);
\draw[dashed] (-10,0)--(-10,15);
\draw[dashed] (10,0)--(10,15);
\node[circle,fill,inner sep=0.5pt] at (0,4) {};
\node[circle,fill,inner sep=0.5pt] at (0,8) {};
\node[circle,fill,inner sep=0.5pt] at (0,12) {};
\draw [->] ($ (0,12) + (230:0.5) $) to[out=230,in=130] ($ (0,4) + (130:0.5) $);
\draw [->] ($ (0,4) + (60:0.5) $) to[out=60,in=300] ($ (0,8) + (300:0.5) $);
\draw [->] ($ (0,8) + (60:0.5) $) to[out=60,in=300] ($ (0,12) + (300:0.5) $);
\draw (0,-1pt)--(0,1pt);
\node at (1,-2) [font=\footnotesize] {$b_0$};
\node at (-16,7.5) {$(132)\! :$};
\end{scope}
\end{tikzpicture}
\end{center}
(The original configuration is contained in the interior of the shaded region in each case.) Now, the left square below
\begin{equation}\label{eqn:squares:of:stabilisation:maps}
\centering
\begin{split}
\begin{tikzpicture}
[x=1mm,y=1mm,>=stealth',font=\small]
\node (a) at (0,10) {$\cdot$};
\node (b) at (20,10) {$\cdot$};
\node (c) at (0,0) {$\cdot$};
\node (d) at (20,0) {$\cdot$};
\draw [->] (a) to node[above]{$\breve{s}$} (b);
\draw [->] (c) to node[below]{$\breve{s}$} (d);
\draw [->] (a) to node[left]{$\breve{s}$} (c);
\draw [->] (b) to node[right]{$\breve{s}$} (d);
\node at (10,5) {\rotatebox{225}{$\Rightarrow$}};
\begin{scope}[xshift=50mm]
\node (a) at (0,10) {$\cdot$};
\node (b) at (20,10) {$\cdot$};
\node (c) at (0,0) {$\cdot$};
\node (d) at (20,0) {$\cdot$};
\draw [->] (a) to node[above]{$\mathring{s}$} (b);
\draw [->] (c) to node[below]{$\mathring{s}$} (d);
\draw [->] (a) to node[left]{$-\mathring{s}^2$} (c);
\draw [->] (b) to node[right]{$-\mathring{s}^2$} (d);
\node at (10,5) {\rotatebox{225}{$\Rightarrow$}};
\end{scope}
\end{tikzpicture}
\end{split}
\end{equation}
admits the identity homotopy $1$ and the homotopy $(12)$. These induce \emph{relative stabilisation maps}
\begin{equation*}
\widetilde{s}_1, \widetilde{s}_{(12)} \colon R_n (M) \longrightarrow R_{n+1} (M)
\end{equation*}
on relative unordered configuration spaces. Similarly the right square admits $1$, $(123)$ and $(132)$, which induce \emph{relative double stabilisation maps}
\begin{equation*}
\widetilde{s}^2_1, \widetilde{s}^2_{(123)}, \widetilde{s}^2_{(132)} \colon R_n^+ (M) \longrightarrow R_{n+2}^+ (M)
\end{equation*}
on relative oriented configuration spaces.
\end{defn}
\begin{rmk}\label{rmk:relative:stab:maps}
The natural self-homotopies $\breve{s}^2 \simeq \breve{s}^2$ come from the different ways of moving the two new configuration points around in the collar neighbourhood ${B_0 \times [0,1)}$, so they are parametrised by $\pi_1 C_2 (B_0 \times [0,1))$. We are only considering those which are supported in a coordinate neighbourhood near $b_0$, which are parametrised by $\pi_1 C_2 (\bbR^d)$. This is either $\Sigma_2$ ($d\geq 3$) or $\beta_2$ ($d=2$); the homotopy $(12)$ defined above corresponds respectively to the only nontrivial element, or a generator.

The analogous statement holds for self-homotopies $-\mathring{s}^3 \simeq -\mathring{s}^3$, replacing `$C_2$' by `$C_3^+$'. In this case $\pi_1 C_3^+ (\bbR^d)$ is either $A_3$ ($d\geq 3$) or $A\beta_3$ ($d=2$); again the homotopies $(123)$, $(132)$ defined above correspond respectively to the only nontrivial elements, or a generating pair.
\end{rmk}
\begin{defn}
To define the relative puncturing map $\widetilde{p}_n \colon R_n^+ (M) \longrightarrow R_n^+ (M_1)$, we need to choose a homotopy $sp_n \simeq p_n s$. Similarly to the definition of the relative stabilisation maps, we define this to fix the original configuration, and swap the puncture
\tikz[baseline=-0.6ex]{%
\node[circle,draw,inner sep=1pt] at (0,0) {};
}
and the new configuration point
\tikz[baseline=-0.6ex]{%
\node[circle,fill,inner sep=1pt] at (0,0) {};
}
on $\bbH$ as illustrated below:
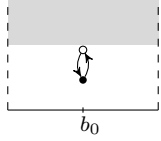
\begin{center}
\begin{tikzpicture}
[x=1mm,y=1mm,>=stealth']
\fill[black!15] (-10,8mm+2pt) rectangle (10,15);
\draw (-10,0)--(10,0);
\draw[dashed] (-10,0)--(-10,15);
\draw[dashed] (10,0)--(10,15);
\node[circle,fill,inner sep=1pt] at (0,4) {};
\node[circle,draw,inner sep=1pt] at (0,8) {};
\draw [->] ($ (0,4) + (60:0.5) $) to[out=60,in=300] ($ (0,8) + (300:0.5) $);
\draw [->] ($ (0,8) + (240:0.5) $) to[out=240,in=120] ($ (0,4) + (120:0.5) $);
\draw (0,-1pt)--(0,1pt);
\node at (1,-2) [font=\footnotesize] {$b_0$};
\end{tikzpicture}
\end{center}
This homotopy will be called $(12)_p$, to distinguish it from the homotopy $(12)$ fitting in to the left square of \eqref{eqn:squares:of:stabilisation:maps}.
\end{defn}
\begin{rmk}\label{rmk:factorisation:of:relative:sn}
By Remark \ref{rmk:factorisation:of:sn}, and the definitions above, the composition
\begin{equation*}
\widetilde{a}_{n+2} \circ \widetilde{\jmath}_{n+2,0} \circ \widetilde{p}_{n+1} \circ \widetilde{a}_{n+1} \circ \widetilde{\jmath}_{n+1,0} \circ \widetilde{p}_n \colon R_n^+ (M) \longrightarrow R_{n+2}^+ (M)
\end{equation*}
is a factorisation of $\widetilde{s}^2_H$, where $H$ is the composite homotopy
\begin{center}
\begin{tikzpicture}
[x=0.7mm,y=0.7mm,font=\footnotesize]
\draw (0,0) rectangle (20,10);
\draw (0,5) -- (20,5);
\node at (10,2.5) {$(12)$};
\node at (10,7.5) {$(12)$};
\end{tikzpicture}
\end{center}
We are reading this in the direction \rotatebox[origin=c]{45}{$\Leftarrow$}, so this is $(132)$. We note that the homotopy $(123)$ \emph{also} factorises into two copies of $(12)$, but pasted together differently:
\begin{center}
\begin{tikzpicture}
[x=0.7mm,y=0.7mm,>=stealth',font=\footnotesize]
\draw (0,0) rectangle (20,10);
\draw [->] (18,9) -- (2,1);
\node at (5,7) {$(12)$};
\node at (15,3) {$(12)$};
\end{tikzpicture}
\end{center}
The diagonal map here is the (positive) stabilisation map $s\colon C_{n+1}^+ (M) \longrightarrow C_{n+2}^+ (M)$.
\end{rmk}
\begin{rmk}\label{rmk:up:homotopic:to:id:2}
We note that the composition $\widetilde{u}_n \circ \widetilde{p}_n$ is homotopic to the identity (c.f.\ Remark \ref{rmk:up:homotopic:to:id:1}). Indeed, composing the diagrams defining $\widetilde{u}_n$ and $\widetilde{p}_n$ results in
\begin{center}
\begin{tikzpicture}
[>=stealth']
\begin{scope}[x=1.5mm,y=1.2mm,font=\footnotesize]
\node (tl) at (0,10) {$\cdot$};
\node (tm) at (20,10) {$\cdot$};
\node (tr) at (40,10) {$\cdot$};
\node (bl) at (0,0) {$\cdot$};
\node (bm) at (20,0) {$\cdot$};
\node (br) at (40,0) {$\cdot$};
\draw [->] (tl) to node[above]{$s$} (tm);
\draw [->] (tm) to (tr);
\draw [->] (bl) to node[below]{$s$} (bm);
\draw [->] (bm) to (br);
\draw [->] (tl) to node[left]{$u_n \circ p_n$} (bl);
\draw [->] (tm) to node[right]{$u_{n+1} \circ p_{n+1}$} (bm);
\draw [->] (tr) to node[right]{$\widetilde{u}_n \circ \widetilde{p}_n$} (br);
\node at (10,5) {$\text{\rotatebox[origin=c]{45}{$\Leftarrow$}}_H$};
\end{scope}
\draw[black!50,dashed] (75mm,-5mm)--(75mm,15mm);
\begin{scope}[xshift=95mm,x=1mm,y=1mm]
\node at (-15,5) [font=\footnotesize] {$H=$};
\fill[black!15] (-10,8mm+2pt) rectangle (10,12);
\draw (-10,0)--(10,0);
\draw[dashed] (-10,0)--(-10,12);
\draw[dashed] (10,0)--(10,12);
\node[circle,fill,inner sep=0.5pt] at (0,8) {};
\draw [->] ($ (0,8) + (240:0.5) $) to[out=240,in=120] (0,4);
\draw (0,-1pt)--(0,1pt);
\node at (1,-2) [font=\footnotesize] {$b_0$};
\end{scope}
\end{tikzpicture}
\end{center}
and a little thought shows that the maps $u_n \circ p_n$, $u_{n+1} \circ p_{n+1}$ and the homotopy $H$ can be simultaneously homotoped to identities, which induces a homotopy from $\widetilde{u}_n \circ \widetilde{p}_n$ to the identity.
\end{rmk}
%%%%%%%%%%%%%%%%%%%%%%%%%%%%%%%%%%%%%%%%%%%%%%%%%%%%%%%%%%%%%%%%%%%%%%%%%%%%%%%%%%

%%%%%%%%%%%%%%%%%%%%%%%%%%%%%%%%%%%%%%%%%%%%%%%%%%%%%%%%%%%%%%%%%%%%%%%%%%%%%%%%%%
\section{Sketch of the proof}\label{sec:sketch:of:the:proof}

The aim of this section is to explain some of the ideas in the proof of the Main Theorem, and especially how the proof differs from the proof of the \emph{unordered} version of this theorem. The proof itself is contained in sections \ref{sec:two:spectral:sequences}, \ref{sec:the:connectivity:of:the:unpuncturing:map}, and \ref{sec:proof:of:the:main:theorem} below, and does not depend on the contents of the present section, which is purely an overview.

%%%%%%%%%%%%%%%%%%%%%%%%%%%%%%%%%%%%%%%%%%%%%%%%%%%
\subsection{The unordered case}
%%%%%%%%%%%%%%%%%%%%%%%%%%%%%%%%%%%%%%%%%%%%%%%%%%%

We first outline the proof of the \emph{unordered} version of the Main Theorem, due to Randal-Williams:
\begin{thm}[\cite{Randal:Williams2011}] \label{thm:the:unordered:version}
If $M$ is the interior of a connected manifold-with-boundary of dimension at least $2$, and $X$ is any path-connected space, then the stabilisation map
\begin{equation*}
C_n (M,X) \xrightarrow{s} C_{n+1} (M,X)
\end{equation*}
is an isomorphism on homology up to degree $\frac{n-2}{2}$ and a surjection up to degree $\frac{n}{2}$. Equivalently,
\begin{equation}\label{eqn:IH:unordered}
\widetilde{H}_* R_n (M,X) = 0 \text{ for } *\leq \tfrac{n}{2}.
\end{equation}
\end{thm}
\begin{proof}[Sketch of proof]\renewcommand{\qedsymbol}{}
Since $M$ and $X$ are path-connected and $\mathrm{dim} (M)\geq 2$, all the configuration and relative configuration spaces are also path-connected, so $\widetilde{H}_0 R_n (M,X)=0$ for all $n$. This proves the $n=0,1$ cases of \eqref{eqn:IH:unordered}; the general result will be proved by induction on $n$.

The strategy will be to construct some map with target $R_n (M,X) = R_n (M)$, and then prove that it is both zero and surjective on homology up to degree $\frac{n}{2}$. Two possible maps into $R_n (M)$ are the relative stabilisation maps $\widetilde{s}_1$ and $\widetilde{s}_{(12)}$ defined in \S \ref{subsubsec:relative:stabilisation:maps}, which are induced by putting the homotopies $1$, $(12)$ into the left-hand square of \eqref{eqn:squares:of:stabilisation:maps}. By the unordered version of Remark \ref{rmk:factorisation:of:sn}, the vertical maps $s$ in this square factorise into $a\circ j\circ p$ (corresponding to puncturing the manifold, then replacing the puncture by a new configuration point which is marked as special, and then forgetting which point is special). Now, the unordered versions of the maps $\widetilde{p}$, $\widetilde{\jmath}$, $\widetilde{a}$ on relative configuration spaces are defined similarly to the oriented ones: $\widetilde{p}$ is induced by a square containing the homotopy $(12)_p$ and $\widetilde{\jmath}$, $\widetilde{a}$ are induced by squares containing the identity homotopy. Hence the homotopy $(12)$ respects the factorisation $s=a\circ j\circ p$ (i.e.\ it factorises into
\tikz[baseline=0.4ex,x=1mm,y=1mm,scale=0.3,>=stealth']{%
\draw (0,0) rectangle (10,10);
\draw [->,very thin] (1,3.3) -- (9,3.3);
\draw [->,very thin] (1,6.6) -- (9,6.6);
}),
whereas the identity homotopy $1$ does not. So $\widetilde{s}_{(12)}$ has an induced factorisation $\widetilde{s}_{(12)} = \widetilde{a}\circ \widetilde{\jmath}\circ \widetilde{p}$. On the other hand $1$ trivially factorises into triangles
\tikz[baseline=0.4ex,x=1mm,y=1mm,scale=0.3,>=stealth']{%
\draw (0,0) rectangle (10,10);
\draw [->] (9,9) -- (1,1);
}
(here, the diagonal map and both homotopies are just identities), but $(12)$ does not.

Now, by some intricate arguments (this is where the bulk of the proof lies, and is contained in \S\S \ref{sec:two:spectral:sequences}, \ref{sec:the:connectivity:of:the:unpuncturing:map}, \ref{sec:proof:of:the:main:theorem} for the oriented case) the induced factorisation of $\widetilde{s}_{(12)}$ into $\widetilde{a}\circ \widetilde{\jmath}\circ \widetilde{p}$ allows us, using the inductive hypothesis, to prove that it is \emph{surjective} on homology up to the required degree. On the other hand a factorisation into triangles
\tikz[baseline=0.4ex,x=1mm,y=1mm,scale=0.3,>=stealth']{%
\draw (0,0) rectangle (10,10);
\draw [->] (9,9) -- (1,1);
}
automatically gives a nullhomotopy of the induced map on mapping cones; hence $\widetilde{s}_1$ is \emph{zero} on homology (in all degrees). But neither map factorises both ways, so this doesn't yet finish the inductive step. Instead, in the unordered case, the following trick suffices to finish it off:

We have a map of long exact sequences
\begin{center}
\begin{tikzpicture}
[x=1mm,y=1mm,scale=1.5,>=stealth',font=\small]
\node (t1) at (0,10) {$H_* C_n (M)$};
\node (t2) at (20,10) {$H_* R_{n-1} (M)$};
\node (t3) at (40,10) {$H_{*-1} C_{n-1} (M)$};
\node (t4) at (60,10) {$H_{*-1} C_n (M)$};
\node (b1) at (0,0) {$H_* C_{n+1} (M)$};
\node (b2) at (20,0) {$H_* R_n (M)$};
\node (b3) at (40,0) {$H_{*-1} C_n (M)$};
\node (b4) at (60,0) {$H_{*-1} C_{n+1} (M)$};
\node (tl) at (-10,10) {};
\node (tr) at (70,10) {};
\node (bl) at (-10,0) {};
\node (br) at (70,0) {};
\draw[->] (tl) to (t1);
\draw[->] (t1) to (t2);
\draw[->] (t2) to (t3);
\draw[->] (t3) to node[above,font=\small]{$s_*$} (t4);
\draw[->] (t4) to (tr);
\draw[->] (bl) to (b1);
\draw[->] (b1) to (b2);
\draw[->] (b2) to (b3);
\draw[->] (b3) to (b4);
\draw[->] (b4) to (br);
\draw[->] (t1) to (b1);
\draw[->>] (t2) to node[right,font=\footnotesize]{$(\widetilde{s}_{(12)})_*$} (b2);
\draw[->] (t3) to (b3);
\draw[->] (t4) to (b4);
\draw[->,rounded corners=5pt] (4,8) to (4,2) to node[above,font=\small]{$0$} (16,2);
\end{tikzpicture}
\end{center}
where the indicated composition is zero since it is induced by a cofibration sequence. In the range of degrees under consideration we know that $(\widetilde{s}_{(12)})_*$ is surjective, so it is sufficient to prove surjectivity of the map $H_* C_n (M) \longrightarrow H_* R_{n-1} (M)$. By exactness, this is equivalent to injectivity of $s_* \colon H_{*-1} C_{n-1} (M) \longrightarrow H_{*-1} C_n (M)$. The inductive hypothesis only gives us this in the range $*\leq \frac{n-1}{2}$, which is not quite enough. However, in the unordered case one can show, by a completely different argument, that $s_*$ is split-injective in \emph{every} degree (see \S \ref{subsec:inj:of:s}), so this finishes the proof.
\end{proof}

%%%%%%%%%%%%%%%%%%%%%%%%%%%%%%%%%%%%%%%%%%%%%%%%%%%
\subsection{The oriented case}\label{subsec:sketch:of:the:proof:oriented}
%%%%%%%%%%%%%%%%%%%%%%%%%%%%%%%%%%%%%%%%%%%%%%%%%%%

The oriented version of this theorem is the main theorem of this paper:
\begin{mainthm}
If $M$ is the interior of a connected manifold-with-boundary of dimension at least $2$, and $X$ is any path-connected space, then the stabilisation map
\begin{equation*}
C_n^+ (M,X) \xrightarrow{s} C_{n+1}^+ (M,X)
\end{equation*}
is an isomorphism on homology up to degree $\frac{n-5}{3}$ and a surjection up to degree $\frac{n-2}{3}$. Equivalently,
\begin{equation}\label{eqn:IH:oriented}
\widetilde{H}_* R_n^+ (M,X) = 0 \text{ for } *\leq \tfrac{n-2}{3}.
\end{equation}
\end{mainthm}
\begin{proof}[Sketch of proof]\renewcommand{\qedsymbol}{}
The basic strategy for the inductive step in the oriented case is the same: find a map with target $R_n^+ (M,X) = R_n^+ (M)$ which is both zero and surjective on homology up to degree $\frac{n-2}{3}$. By analogy with the unordered case, the first thing one might try is the relative stabilisation maps induced by
\begin{center}
\begin{tikzpicture}
[x=1mm,y=1mm,>=stealth',font=\small]
\node (a) at (0,10) {$\cdot$};
\node (b) at (20,10) {$\cdot$};
\node (c) at (0,0) {$\cdot$};
\node (d) at (20,0) {$\cdot$};
\draw [->] (a) to node[above,font=\footnotesize]{$s$} (b);
\draw [->] (c) to node[below,font=\footnotesize]{$s$} (d);
\draw [->] (a) to node[left,font=\footnotesize]{$\pm s$} (c);
\draw [->] (b) to node[right,font=\footnotesize]{$\pm s$} (d);
\node at (10,5) [font=\footnotesize] {$H$};
\draw[black!50,dashed] (30,-1) -- (30,11);
\node [anchor=west] at (35,5) {with $H= \left\lbrace \begin{smallmatrix} 1 & \text{\itshape if the vertical maps have the same sign} \\ (12) & \text{\itshape if the vertical maps have opposite signs} \end{smallmatrix} \right\rbrace .$};
\end{tikzpicture}
\end{center}
Similarly to before, we would like the homotopy $H$ to factorise like
\tikz[baseline=0.4ex,x=1mm,y=1mm,scale=0.3,>=stealth']{%
\draw (0,0) rectangle (10,10);
\draw [->,very thin] (1,3.3) -- (9,3.3);
\draw [->,very thin] (1,6.6) -- (9,6.6);
},
so we need to choose the case where the vertical maps have opposite signs and $H=(12)$. This gives an induced factorisation of the relative stabilisation map into $\widetilde{a}\circ\widetilde{\jmath}\circ\widetilde{p}$, which allows us to prove that it is surjective on homology, by the same kind of arguments as in the unordered case.

However, $(12)$ does \emph{not} factorise into triangles
\tikz[baseline=0.4ex,x=1mm,y=1mm,scale=0.3,>=stealth']{%
\draw (0,0) rectangle (10,10);
\draw [->] (9,9) -- (1,1);
},
so we cannot deduce that it is also zero on homology. So far this is just as in the unordered case, but this time the `ladder trick' which finished off the inductive step in the unordered case does \emph{not} work: It depends on knowing injectivity of $s_*$ in all degrees, in advance, by a separate argument, but in the oriented case $s_*$ is \emph{not} always injective (see \S \ref{sec:failure:of:injectivity}).

So to solve this we will instead construct a \emph{different} factorisation of the relative stabilisation map on homology, and then use this factorisation (and naturality of the factorisation w.r.t.\ stabilisation maps) to show that it factors through the zero map in the required range of degrees. This new factorisation is actually just a general construction for homotopy-commutative squares: the map on mapping cones induced by choosing any particular homotopy to fill the square has a certain factorisation on homology --- as long as the square admits \emph{some} homotopy which factorises into triangles
\tikz[baseline=0.4ex,x=1mm,y=1mm,scale=0.3,>=stealth']{%
\draw (0,0) rectangle (10,10);
\draw [->] (9,9) -- (1,1);
}.
However, we do not currently have such a split homotopy. To remedy this, we can stack two copies of our square on top of each other; this produces the right-hand square of diagram \eqref{eqn:squares:of:stabilisation:maps}, filled by the homotopy $(132)$. So we have extended our map into $R_n^+ (M)$ further back, to $\widetilde{s}^2_{(132)} \colon R_{n-2}^+ (M) \to R_n^+ (M)$.

Now we also have the homotopy $(123)$ filling the same square, and as noted in Remark \ref{rmk:factorisation:of:relative:sn}, this factorises into triangles
\tikz[baseline=0.4ex,x=1mm,y=1mm,scale=0.3,>=stealth']{%
\draw (0,0) rectangle (10,10);
\draw [->] (9,9) -- (1,1);
}
(as does the identity homotopy, in fact). This allows us to construct the aforementioned factorisation of $\widetilde{s}^2_{(132)}$ on homology, which is
\begin{center}
\begin{tikzpicture}
[x=1.5mm,y=1mm,scale=1,>=stealth']
\node (a) at (0,0) {$R_{n-2}^+ (M)$};
\node (b) at (20,0) {$\Sigma C_{n-2}^+ (M)$};
\node (c) at (40,0) {$C_{n+1}^+ (M)$};
\node (d) at (60,0) {$R_n^+ (M),$};
\draw [->] (a) to (b);
\draw [->,dashed] (b) to (c);
\draw [->] (c) to (d);
\end{tikzpicture}
\end{center}
where a dotted arrow indicates a map defined only on homology.

In this case one can also check that the middle part of the factorisation commutes with stabilisation maps in the following way:
\begin{center}
\begin{tikzpicture}
[x=1.5mm,y=1.2mm,scale=1,>=stealth']
\node (a) at (0,0) {$R_{n-2}^+ (M)$};
\node (b) at (20,0) {$\Sigma C_{n-2}^+ (M)$};
\node (c) at (40,0) {$C_{n+1}^+ (M)$};
\node (d) at (60,0) {$R_n^+ (M)$};
\node (b2) at (20,-10) {$\Sigma C_{n-3}^+ (M)$};
\node (c2) at (40,-10) {$C_n^+ (M)$};
\draw [->] (a) to (b);
\draw [->,dashed] (b) to (c);
\draw [->] (c) to (d);
\draw [->,dashed] (b2) to (c2);
\draw [->] (b2) to node[left,font=\small]{$\Sigma (-s_{n-3})$} (b);
\draw [->] (c2) to node[right,font=\small]{$s_n$} (c);
\node at (30,-5) [font=\small] {$\circlearrowleft$};
\end{tikzpicture}
\end{center}
Now we can show that the top row ($\widetilde{s}^2_{(132)}$ on homology) is zero in the desired range:

The inductive hypothesis implies that $\Sigma (-s_{n-3})$ is surjective on homology in this range, so we can factor the top row along the bottom of the diagram like
\tikz[x=2mm,y=2mm,>=stealth']{%
\draw [->] (0,1)--(1,1)--(1,0)--(2,0)--(2,1)--(3.5,1);
}.
In particular it factors through $C_n^+ (M) \xrightarrow{s_n} C_{n+1}^+ (M) \to R_n^+ (M)$, which is zero on homology since it is induced by a cofibration sequence.

This finishes the inductive step, since \emph{surjectivity}-on-homology can be proved as before, using the factorisation $\widetilde{s}^2_{(132)} = \widetilde{a}\circ\widetilde{\jmath}\circ\widetilde{p}\circ\widetilde{a}\circ\widetilde{\jmath}\circ\widetilde{p}$. However, note that we are now using the inductive hypothesis from further back (to prove surjectivity for the ``older'' copies of $\widetilde{a}_*$, $\widetilde{\jmath}_*$, $\widetilde{p}_*$), which results in a \emph{smaller} improvement in the range of stability during each inductive step --- and hence the slower rate of stabilisation in the oriented case.
\end{proof}
\begin{rmk}
This narrative outlines a fairly direct link from the existence of a global parameter on configuration spaces to the reduced stability slope: Firstly it means that injectivity of $s_*$ fails (see \S \ref{sec:failure:of:injectivity} for more on this), cutting off one line of attack, and secondly it makes the other line of attack weaker: The global parameter is an obstruction to the existence of certain self-homotopies of iterated stabilisation maps, which are needed to do the \emph{zero}-on-homology half of the proof in this line of attack. Hence we need to extend our map into $R_n^+ (M)$ further back to obtain such self-homotopies. This means we need to use the inductive hypothesis from further back, to prove \emph{surjectivity}-on-homology for the ``older'' parts of this map, and so this only goes through for a \emph{smaller} range of degrees --- hence we get a smaller increase in the stability range with each inductive step, and hence the rate of stabilisation is slower.
\end{rmk}
%%%%%%%%%%%%%%%%%%%%%%%%%%%%%%%%%%%%%%%%%%%%%%%%%%%%%%%%%%%%%%%%%%%%%%%%%%%%%%%%%%

%%%%%%%%%%%%%%%%%%%%%%%%%%%%%%%%%%%%%%%%%%%%%%%%%%%%%%%%%%%%%%%%%%%%%%%%%%%%%%%%%%
\section{Two spectral sequences}\label{sec:two:spectral:sequences}

In this section we first establish the two spectral sequences to be used in the proof of the Main Theorem, and then show that (as mentioned in Remark \ref{rmk:n:minus:1:resolution}) the augmented $\Delta$-space $C_n^+ (M,X)^\bullet$ is an $(n-1)$-resolution, implying that one of our spectral sequences converges to zero in a range of degrees.

%%%%%%%%%%%%%%%%%%%%%%%%%%%%%%%%%%%%%%%%%%%%%%%%%%%
\subsection{General constructions}
%%%%%%%%%%%%%%%%%%%%%%%%%%%%%%%%%%%%%%%%%%%%%%%%%%%

The first spectral sequence we will make use of is a relative version of the Serre spectral sequence. We denote the mapping cone of a map $g$ by $Cg$.
\begin{prop}\label{prop:rsss}
Suppose $f$ is a map of fibrations over a path-connected space $B$
\begin{equation}\label{eqn:map:of:fibrations}
\centering
\begin{split}
\begin{tikzpicture}
[x=1mm,y=1mm,>=stealth']
\node (e0) at (-10,10) {$E_0$};
\node (e1) at (10,10) {$E_1$};
\node (b) at (0,0) {$B$};
\draw [->] (e0) to (b);
\draw [->] (e1) to (b);
\draw [->] (e0) to node[above,font=\small]{$f$} (e1);
\end{tikzpicture}
\end{split}
\end{equation}
Let $F_0, F_1$ be the fibres over a point $b\in B$, and denote the restriction of $f$ to $F_0 \to F_1$ by $f_b$. Then there is a first quadrant spectral sequence
\begin{equation*}
E_{s,t}^2 \; \cong \; H_s (B; \widetilde{H}_t (Cf_b)) \quad \Rightarrow \quad \widetilde{H}_* (Cf)
\end{equation*}
in which the $r$th differential has bidegree $(-r,r-1)$. The edge homomorphism
\begin{equation*}
\widetilde{H}_t (Cf_b) \cong E^2_{0,t} \;\twoheadrightarrow\; E^{\infty}_{0,t} \;\hookrightarrow\; \widetilde{H}_t (Cf)
\end{equation*}
is the map on $\widetilde{H}_t$ induced by the inclusion $Cf_b \hookrightarrow Cf$.
\end{prop}
\noindent This is mentioned as Remark 2 on p.\ 351 of \cite{Switzer1975}. (There it is assumed that $f$ is an inclusion, but this can be ensured by replacing \eqref{eqn:map:of:fibrations} by a homotopy-equivalent diagram.) We will show how to derive this from the usual (absolute) Serre spectral sequence:
\begin{proof}
Let $C^{\mathrm{fib}}E_0$ be the \emph{fibrewise} cone on $E_0$, i.e.\ $E_0 \times [0,1]$ with $F_b \times \lbrace 1 \rbrace$ collapsed to a point separately for each fibre $F_b$, and let
\begin{equation*}
C^{\mathrm{fib}}f = E_1 \cup_f C^{\mathrm{fib}}E_0
\end{equation*}
(compare $Cf = E_1 \cup_f CE_0$). There is an induced fibration $p\colon C^{\mathrm{fib}}f \to B$, whose fibre is $Cf_b$, and which has a section $s\colon B \to C^{\mathrm{fib}}f$ taking $b^\prime \in B$ to the tip of the cone in the fibre over $b^\prime$. Collapsing this section gives a map $c\colon C^{\mathrm{fib}}f \to Cf$. These maps fit into the diagram
\begin{center}
\begin{tikzpicture}
[x=1mm,y=1.2mm,>=stealth']
\node (TL) at (0,20) {$\lbrace b \rbrace$};
\node (TM) at (20,20) {$Cf_b$};
\node (TR) at (40,20) {$Cf_b$};
\node (ML) at (0,10) {$B$};
\node (MM) at (20,10) {$C^{\mathrm{fib}}f$};
\node (MR) at (40,10) {$Cf$};
\node (BL) at (0,0) {$B$};
\node (BM) at (20,0) {$B$};
\draw [->] (TL) to (TM);
\draw [->] (TM) to node[above,font=\footnotesize]{$1$} (TR);
\draw [->] (ML) to node[above,font=\footnotesize]{$s$} (MM);
\draw [->] (MM) to node[above,font=\footnotesize]{$c$} (MR);
\draw [->] (BL) to node[below,font=\footnotesize]{$1$} (BM);
\draw [->] (TL) to (ML);
\draw [->] (ML) to node[left,font=\footnotesize]{$1$} (BL);
\incl{(TM)}{(MM)}
\draw [->] (MM) to node[right,font=\footnotesize]{$p$} (BM);
\incl{(TR)}{(MR)}
\end{tikzpicture}
\end{center}
where each \emph{vertical} sequence is a fibration sequence and each \emph{horizontal} sequence is a split cofibration sequence.

The required spectral sequence will be a direct summand of the Serre spectral sequence associated to the middle fibration --- this can be seen as follows: The map of fibrations in the diagram above induces a map of Serre spectral sequences, and the fact that the horizontal sequences are split cofibrations means that we can identify this map of spectral sequences, on each page $E^r$ and in the limit, as an inclusion of a direct summand. In particular,
\begin{align*}
&\text{on the $E^2$ page:}& H_s (B; H_t (pt)) &\;\hookrightarrow\; H_s (B; H_t (pt)) \oplus H_s (B; \widetilde{H}_t (Cf_b)) ; \\
&\text{in the limit:}& H_* (B) &\;\hookrightarrow\; H_* (B) \oplus \widetilde{H}_* (Cf).
\end{align*}
Passing to the \emph{other} direct summand now gives the required spectral sequence.

The claim about edge homomorphisms follows from the analogous fact about edge homomorphisms for the Serre spectral sequence associated to $Cf_b \hookrightarrow C^{\mathrm{fib}}f \twoheadrightarrow B$, of which our spectral sequence is a direct summand.
\end{proof}
To state the next construction of a spectral sequence, we first define the notion of a `double mapping cone':
\begin{defn}
Given a square of maps which commutes up to homotopy, and a chosen homotopy to fill this square, one can apply the mapping cone construction either vertically then horizontally, or horizontally then vertically. It is not hard to check that the resulting $2$-by-$2$ grid
\tikz[baseline=0.2ex,scale=0.8]{%
\draw [step=2mm] (0,0) grid (4mm,4mm);
}
of spaces and maps is the same up to homeomorphism whichever way around this is done. In particular the mapping cone (taken horizontally) of the induced map-on-mapping cones (taken vertically) is homeomorphic to the mapping cone (taken vertically) of the induced map-on-mapping-cones (taken horizontally). We define this to be the \emph{double mapping cone} of the original square-with-homotopy.
\end{defn}
The second spectral sequence we will need is constructed from a map of augmented $\Delta$-spaces. There are versions of this construction for $\Delta$-spaces and for augmented $\Delta$-spaces, which can be either basepointed or non-basepointed, and maps of any of the above. The version we will use is:
\begin{prop}\label{prop:dss}
Given a map of augmented $\Delta$-spaces $Y_\bullet \to Z_\bullet$, there is an induced square of maps
\begin{equation}\label{eqn:map:of:aug:dspaces}
\centering
\begin{split}
\begin{tikzpicture}
[x=1mm,y=1.2mm,>=stealth']
\node (TL) at (0,10) {$\geomr{Y_\bullet}$};
\node (TR) at (20,10) {$\geomr{Z_\bullet}$};
\node (BL) at (0,0) {$Y_{-1}$};
\node (BR) at (20,0) {$Z_{-1}$};
\draw [->] (TL) to (TR);
\draw [->] (BL) to (BR);
\draw [->] (TL) to (BL);
\draw [->] (TR) to (BR);
\end{tikzpicture}
\end{split}
\end{equation}
Denote the double mapping cone of this square by $C^2 (Y_\bullet \to Z_\bullet)$, and as before denote the mapping cone of $Y_s \to Z_s$ by $C(Y_s \to Z_s)$. Then there is a spectral sequence in $\lbrace s\geq -1, t\geq 0 \rbrace$,
\begin{equation*}
E_{s,t}^1 \; \cong \; \widetilde{H}_t (C(Y_s \to Z_s)) \quad \Rightarrow \quad \widetilde{H}_{*+1} (C^2 (Y_\bullet \to Z_\bullet)),
\end{equation*}
where the first differential is the alternating sum of the maps on homology induced by the relative face maps; in particular $E_{-1,t}^1 \longleftarrow E_{0,t}^1$ is $\widetilde{H}_t$ of the relative augmentation map.
\end{prop}
\begin{proof}
The construction is given in \ref{appendix:spectral:sequences:from:dspaces}.
\end{proof}

%%%%%%%%%%%%%%%%%%%%%%%%%%%%%%%%%%%%%%%%%%%%%%%%%%%
\subsection{The spectral sequences to be used in the proof of the Main Theorem}
%%%%%%%%%%%%%%%%%%%%%%%%%%%%%%%%%%%%%%%%%%%%%%%%%%%

\begin{prop}\label{prop:the:two:spectral:sequences}
We have the following spectral sequences\textup{:}
\begin{align}
E_{s,t}^2 \; &\cong \; H_s \bigl( \widetilde{C}_{i+1} (M); \widetilde{H}_t (R_{n-i-1}^+ (M_{i+1})) \bigr) \quad &\Rightarrow \quad &\widetilde{H}_* (R_n^+ (M)^i) \label{eqn:RSSS}\tag{$\mathrm{RSSS}_i$}\\
E_{s,t}^1 \; &\cong \; \widetilde{H}_t (R_n^+ (M)^s) \quad &\Rightarrow \quad &\widetilde{H}_{*+1} C\widetilde{\varepsilon}_n \label{eqn:DSS}\tag{$\Delta$SS}
\end{align}
for $0\leq i\leq n-3$, where $C\widetilde{\varepsilon}_n$ is as follows\textup{:}
\begin{equation}\label{eqn:double:mapping:cone}
\centering
\begin{split}
\begin{tikzpicture}
[x=1.5mm,y=1.2mm,>=stealth']
\node (TL) at (0,20) {$\geomr{C_n^+ (M)^\bullet}$};
\node (TM) at (20,20) {$\geomr{C_{n+1}^+ (M)^\bullet}$};
\node (TR) at (40,20) {$\geomr{R_n^+ (M)^\bullet}$};
\node (ML) at (0,10) {$C_n^+ (M)$};
\node (MM) at (20,10) {$C_{n+1}^+ (M)$};
\node (MR) at (40,10) {$R_n^+ (M)$};
\node (BL) at (0,0) {$C\varepsilon_n$};
\node (BM) at (20,0) {$C\varepsilon_{n+1}$};
\node (BR) at (40,0) {$C\widetilde{\varepsilon}_n$};
\draw [->] (TL) to (TM);
\draw [>->] (TM) to (TR);
\draw [->] (ML) to (MM);
\draw [>->] (MM) to (MR);
\draw [->] (BL) to (BM);
\draw [>->] (BM) to (BR);
\draw [->] (TL) to node[left,font=\footnotesize]{$\varepsilon_n$} (ML);
\draw [>->] (ML) to (BL);
\draw [->] (TM) to node[right,font=\footnotesize]{$\varepsilon_{n+1}$} (MM);
\draw [>->] (MM) to (BM);
\draw [->] (TR) to node[right,font=\footnotesize]{$\widetilde{\varepsilon}_n$} (MR);
\draw [>->] (MR) to (BR);
%s
\end{tikzpicture}
\end{split}
\end{equation}
The edge homomorphisms on the vertical axis of \eqref{eqn:RSSS} are the maps on $\widetilde{H}_t$ induced by $\widetilde{\jmath}_{n,i}$, and the leftmost $d^1$-differentials of \eqref{eqn:DSS} are the maps on $\widetilde{H}_t$ induced by $\widetilde{a}_n$.
\end{prop}
\begin{proof}
This follows immediately by applying Proposition \ref{prop:rsss} to the map of fibre bundles \eqref{eqn:j:commuting:with:s}, and applying Proposition \ref{prop:dss} to the map of augmented $\Delta$-spaces
\begin{equation*}
s_n^\bullet \colon C_n^+ (M,X)^\bullet \longrightarrow C_{n+1}^+ (M,X)^\bullet .\qedhere
\end{equation*}
\end{proof}

%%%%%%%%%%%%%%%%%%%%%%%%%%%%%%%%%%%%%%%%%%%%%%%%%%%
\subsection{$C_n^+ (M,X)^\bullet$ is an $(n-1)$-resolution}
%%%%%%%%%%%%%%%%%%%%%%%%%%%%%%%%%%%%%%%%%%%%%%%%%%%

In the remainder of this section, we prove that the spectral sequence \eqref{eqn:DSS} converges to zero up to total degree $n-1$, which will follow from the fact that $C_n^+ (M,X)^\bullet$ is an $(n-1)$-resolution. First we define a certain semi-simplicial set:
\begin{defn}
Let $\inj ([i+1],[n])$ be the discrete space of all injections ${[i+1] \hookrightarrow [n]}$; these combine to form a $\Delta$-space $\inj ([\bullet +1],[n])$, with face maps induced by all strictly increasing functions $[i] \to [i+1]$.
\end{defn}
\noindent This appears as the fibre of the map $\varepsilon_n$, which we next prove is a fibre bundle:
\begin{lem}
The map $\varepsilon_n \colon \geomr{C_n^+ (M)^\bullet} \to C_n^+ (M)$ is a fibre bundle, with fibre homeomorphic to $\geomr{\inj ([\bullet +1],[n])}$.
\end{lem}
\begin{proof}
For each level $i\geq 0$, the (unique) composition of face maps $f_i \colon C_n^+ (M)^i \to C_n^+ (M)$ is a finite-sheeted covering map, so in particular it is a fibre bundle. Moreover, this collection can be \emph{simultaneously} locally trivialised: each point $c\in C_n^+ (M)$ has an open neighbourhood $U_c$ over which $f_i$ is a trivial bundle for all $i$. Explicitly, we may take $U_c$ to be the following: Choose pairwise disjoint open balls around the $n$ points of $c$, and associate to these open balls the orientation and $X$-labelling inherited from $c$. Then let $U_c$ be all configurations in $C_n^+ (M)$ which have one point in each open ball, and whose orientation and $X$-labelling matches that of the open balls.

Over $U_c$, the trivialisation itself can be described as follows: Choose an arbitrary, fixed ordering of the $n$ open balls, $(B_1, ..., B_n)$. Given $a\in f_i^{-1} (U_c)$, the $(i+1)$-ordering of $a$ induces an injection $[i+1] \to \lbrace B_1, ..., B_n \rbrace$, and hence an element $\mathrm{ord}(a) \in \inj ([i+1],[n])$. Define the trivialisation $f_i^{-1} (U_c) \cong U_c \times \inj ([i+1],[n])$ to be $a \mapsto (f_i (a), \mathrm{ord}(a))$.

Since we have a simultaneous local trivialisation for $\lbrace f_i \rbrace$, we get a local trivialisation for the map $\coprod_i C_n^+ (M)^i \times \Delta^i \to C_n^+ (M)$, which identifies the preimage of $U_c$ with $U_c \times \bigl( \coprod_i \inj ([i+1],[n]) \times \Delta^i \bigr)$. Under this identification, the face relations for $C_n^+ (M)^\bullet$ correspond exactly to the face relations for $\inj ([\bullet +1],[n])$, since
\begin{center}
\begin{tikzpicture}
[x=1.7mm,y=1.2mm,>=stealth']
\node (a) at (0,10) {$U_c \times \inj ([i+1],[n])$};
\node (b) at (20,10) {$f_i^{-1} (U_c)$};
\node (c) at (0,0) {$U_c \times \inj ([i],[n])$};
\node (d) at (20,0) {$f_{i-1}^{-1} (U_c)$};
\draw [->] (b) to node[right,font=\small]{$d_j$} (d);
\draw [->] (a) to node[left,font=\small]{$1 \!\times\! d_j$} (c);
\node at ($ (a.east) !0.5! (b.west) $) {$\cong$};
\node at ($ (c.east) !0.5! (d.west) $) {$\cong$};
\node at (10,5) {$\circlearrowleft$};
\end{tikzpicture}
\end{center}
Hence we have an induced local trivialisation of the quotient map
\begin{equation*}
\varepsilon_n \colon \geomr{C_n^+ (M)^\bullet} = \Bigl( \coprod_i C_n^+ (M)^i \times \Delta^i \Bigr) \rightquotient \sim \; \longrightarrow \; C_n^+ (M),
\end{equation*}
which identifies the preimage of $U_c$ with $U_c \times \geomr{\inj ([\bullet +1],[n])}$. In particular the fibre over a point is identified with $\geomr{\inj ([\bullet +1],[n])}$.
\end{proof}
\noindent The homotopy type of $\geomr{\inj ([\bullet +1],[n])}$ was identified by Randal-Williams in \cite{Randal:Williams2011}:\footnote{As noted there, this fact has been proved before in the literature, where $\inj ([\bullet +1],[n])$ is known as the ``complex of injective words''.}
\begin{prop}[Proposition 3.2 of \cite{Randal:Williams2011}]
The geometric realisation of the $\Delta$-space ${\inj ([\bullet +1],[n])}$ is a wedge of $(n-1)$-spheres\textup{:}
\begin{equation*}
\geomr{\inj ([\bullet +1],[n])} \simeq \bigvee S^{n-1}.
\end{equation*}
\end{prop}
\noindent Putting this together, we immediately get:
\begin{coro}\label{coro:n:minus:1:resolution}
The map $\varepsilon_n \colon \geomr{C_n^+ (M)^\bullet} \to C_n^+ (M)$ is $(n-1)$-connected; in other words $C_n^+ (M)^\bullet$ is an $(n-1)$-resolution of $C_n^+ (M)$.
\end{coro}
\noindent By the relative Hurewicz theorem and a diagram chase in \eqref{eqn:double:mapping:cone}, this in turn immediately implies that $\widetilde{H}_* C\widetilde{\varepsilon}_n = 0$ for $*\leq n$, and hence
\begin{coro}\label{coro:converging:to:zero}
The spectral sequence \eqref{eqn:DSS} converges to zero in total degree $\leq n-1$.
\end{coro}
%%%%%%%%%%%%%%%%%%%%%%%%%%%%%%%%%%%%%%%%%%%%%%%%%%%%%%%%%%%%%%%%%%%%%%%%%%%%%%%%%%

%%%%%%%%%%%%%%%%%%%%%%%%%%%%%%%%%%%%%%%%%%%%%%%%%%%%%%%%%%%%%%%%%%%%%%%%%%%%%%%%%%
\section{The connectivity of the unpuncturing map}\label{sec:the:connectivity:of:the:unpuncturing:map}

In this section we relate the homology-connectivity of the relative unpuncturing map
\begin{align*}
\widetilde{u}_n &\colon R_n^+ (M_1) \longrightarrow R_n^+ (M) \\
\intertext{(which was defined in \S \ref{subsubsec:relative:maps:1}) to the homology-connectivity of the stabilisation map}
s_{n-1} &\colon C_{n-1}^+ (M) \longrightarrow C_n^+ (M).
\end{align*}
First, we define precisely what we mean by `homology-connectivity':
\begin{defn}
For a map $f\colon Y \to Z$, the \emph{homology-connectivity} of $f$ is
\begin{equation*}
\hconn (f) \coloneqq \mathrm{max} \Bigl\lbrace * \Bigl\lvert \begin{smallmatrix} f \text{\itshape\ is surjective on homology up to degree } * \\ f \text{\itshape\ is injective on homology up to degree } *-1 \end{smallmatrix} \Bigr. \Bigr\rbrace .
\end{equation*}
Equivalently, this is the degree up to which the reduced homology of the mapping cone $Cf$ is zero.
\end{defn}
\begin{prop}\label{prop:excision}
For $n\geq 3$,
\begin{equation*}
\hconn (\widetilde{u}_n) \geq \hconn (s_{n-1}) + \mathrm{dim} (M).
\end{equation*}
\end{prop}
\noindent To prove this we will first construct an excisive square: Let $d = \mathrm{dim} (M)$, and let $D \subset M$ be an open, $d$-dimensional disc embedded in the interior of $M$, far away from the boundary-component $B_0$ of \mbar. We identify $D$ with the standard $d$-dimensional disc with its metric. Let $U_n^+ (M) \subseteq C_n^+ (M)$ be the subspace of configurations which have a unique closest point in $D$ to $0\in D$. (In particular configurations in $U_n^+ (M)$ are required to \emph{have} a point in $D$.) The pair $\lbrace U_n^+ (M), C_n^+ (\minus{M}{0}) \rbrace$ is an open cover of $C_n^+ (M)$, so the square
\begin{equation}\label{eq:excisive:square}
\centering
\begin{split}
\begin{tikzpicture}
[x=1mm,y=1mm,>=stealth',scale=2]
\node (TL) at (0,0) {$C_n^+ (\minus{M}{0})$};
\node (TR) at (20,0) {$C_n^+ (M)$};
\node (BL) at (10,-5) {$U_n^+ (\minus{M}{0})$};
\node (BR) at (30,-5) {$U_n^+ (M)$};
\incl{(BL)}{(BR)}
\incl[left]{(BL)}{(TL)}
\incl[left]{(BR)}{(TR)}
\inclusion{above}{$u_n$}{(TL)}{(TR)}
\end{tikzpicture}
\end{split}
\end{equation}
is excisive.

Now, $U_n^+ (M)$ may be decomposed as follows:
\begin{lem}\label{lem:decomposition:of:Un}
For $n\geq 3$, $U_n^+ (M) \cong C_{n-1}^+ (\minus{M}{0}) \times D \times X$.
\end{lem}
\begin{proof}
First, choose a family of homeomorphisms $\psi_r \colon \minus{M}{\bbar_r (0)} \cong \minus{M}{0}$, with support contained in $D$, depending continuously on the parameter $r\in [0,1)$. Here, $\bbar_r (0)$ means the closed ball in $D$, of radius $r$ centred at $0\in D$.

Given $\left[ \pair{p_1}{x_1} \cdots \pair{p_n}{x_n} \right] \in U_n^+ (M)$, we may assume by applying an even permutation (since $n\geq 3$) that the unique closest point in $D$ to $0$ for this configuration is $p_n$. Sending this to
\begin{equation*}
\left( \left[ \pair{\psi_{\lvert p_n \rvert} (p_1)}{x_1} \;\cdots\; \pair{\psi_{\lvert p_n \rvert} (p_{n-1})}{x_{n-1}} \right], p_n, x_n \right) \in C_n^+ (\minus{M}{0}) \times D \times X
\end{equation*}
defines the required homeomorphism.
\end{proof}
This identification restricts to $U_n^+ (\minus{M}{0}) \cong C_{n-1}^+ (\minus{M}{0}) \times (\minus{D}{0}) \times X$, and under the identification,
\begin{itemize}
\item the inclusion at the bottom of \eqref{eq:excisive:square} is the identity on the first and third factors, and the inclusion $\minus{D}{0} \hookrightarrow D$ on the middle factor;
\item restricting the stabilisation map $s_n \colon C_n^+ (M) \to C_{n+1}^+ (M)$ to $U_n^+ (M) \to U_{n+1}^+ (M)$ yields
\begin{equation*}
s_{n-1} \times 1 \times 1 \colon C_{n-1}^+ (\minus{M}{0}) \times D \times X \longrightarrow C_n^+ (\minus{M}{0}) \times D \times X,
\end{equation*}
and similarly for $s_n \colon C_n^+ (\minus{M}{0}) \to C_{n+1}^+ (\minus{M}{0})$. In other words the identification commutes with stabilisation maps; this is because we embedded $D$ far away from the boundary-component $B_0$. (More precisely, it is ensured by embedding $D$ sufficiently far away from $B_0$ so that the homeomorphism $\phi\colon M^\prime \cong M$ from the definition of the stabilisation map has support disjoint from $D$.)
\end{itemize}
Having done this set-up, we can now prove the main result of this section:
\begin{proof}[Proof of Proposition \ref{prop:excision}]
Apply stabilisation maps vertically to the square \eqref{eq:excisive:square}, to get a commuting cube of maps, and then take mapping cones horizontally and vertically, to produce a commutative lattice of maps of the form
\raisebox{0pt}[0pt][0pt]{%
\tikz[x=1mm,y=1mm,scale=0.8,baseline=5mm]{%
\draw [step=5] (0,0) grid (10,10);
\begin{scope}[xshift=2mm,yshift=-2mm] \draw [step=5] (0,0) grid (10,10); \end{scope}
\draw (0,0) -- +(2,-2);
\draw (5,0) -- +(2,-2);
\draw (10,0) -- +(2,-2);
\draw (0,5) -- +(2,-2);
\draw (5,5) -- +(2,-2);
\draw (10,5) -- +(2,-2);
\draw (0,10) -- +(2,-2);
\draw (5,10) -- +(2,-2);
\draw (10,10) -- +(2,-2);
\useasboundingbox (13,0);
}}.
The back face of this can be identified as:
\begin{center}
\begin{tikzpicture}
[x=1.5mm,y=1.2mm,scale=1,>=stealth']
\node (TL) at (0,20) {$R_n^+ (\minus{M}{0})$};
\node (TM) at (20,20) {$R_n^+ (M)$};
\node (TR) at (40,20) {$C\widetilde{u}_n$};
\node (ML) at (0,10) {$C_{n+1}^+ (\minus{M}{0})$};
\node (MM) at (20,10) {$C_{n+1}^+ (M)$};
\node (MR) at (40,10) {$Cu_{n+1}$};
\node (BL) at (0,0) {$C_n^+ (\minus{M}{0})$};
\node (BM) at (20,0) {$C_n^+ (M)$};
\node (BR) at (40,0) {$Cu_n$};
\draw [->] (TL) to node[above,font=\footnotesize]{$\widetilde{u}_n$} (TM);
\draw [>->] (TM) to (TR);
\draw [->] (ML) to node[above,font=\footnotesize]{$u_{n+1}$} (MM);
\draw [>->] (MM) to (MR);
\draw [->] (BL) to node[below,font=\footnotesize]{$u_n$} (BM);
\draw [>->] (BM) to (BR);
\draw [->] (BL) to node[left,font=\footnotesize]{$s_n$} (ML);
\draw [>->] (ML) to (TL);
\draw [->] (BM) to node[right,font=\footnotesize]{$s_n$} (MM);
\draw [>->] (MM) to (TM);
\draw [->] (BR) to (MR);
\draw [>->] (MR) to (TR);
\end{tikzpicture}
\end{center}
Using Lemma \ref{lem:decomposition:of:Un} and the fact that the mapping cone of $A \!\times\! Y \xrightarrow{1 \!\times\! f} A \!\times\! Z$ is $C(1\!\times\! f) \cong (A_+)\wedge Cf$, the front face can be identified as:
\begin{center}
\begin{tikzpicture}
[x=2.3mm,y=1.2mm,scale=1,>=stealth']
\node (TL) at (0,25) {$C(s_{n-1} \!\times\! 1 \!\times\! 1)$};
\node (TM) at (20,25) {$C(s_{n-1} \!\times\! 1 \!\times\! 1)$};
\node (TR) at (40,25) {$C\Sigma^d (s_{n-1} \!\times\! 1)_+$};
\node (ML) at (0,15) {$C_n^+ (\minus{M}{0}) \!\times\! (\minus{D}{0}) \!\times\! X$};
\node (MM) at (20,15) {$C_n^+ (\minus{M}{0}) \!\times\! D \!\times\! X$};
\node (MR) at (40,15) {$\Sigma^d (C_n^+ (\minus{M}{0}) \!\times\! X)_+$};
\node (BL) at (0,0) {$C_{n-1}^+ (\minus{M}{0}) \!\times\! (\minus{D}{0}) \!\times\! X$};
\node (BM) at (20,0) {$C_{n-1}^+ (\minus{M}{0}) \!\times\! D \!\times\! X$};
\node (BR) at (40,0) {$\Sigma^d (C_{n-1}^+ (\minus{M}{0}) \!\times\! X)_+$};
\draw [->] (TL) to (TM);
\draw [>->] (TM) to (TR);
\incl{(ML)}{(MM)}
\draw [>->] (MM) to (MR);
\incl{(BL)}{(BM)}
\draw [>->] (BM) to (BR);
\draw [->] (BL) to node[left,font=\footnotesize]{$s_{n-1}\!\times\! 1 \!\times\! 1$} (ML);
\draw [>->] (ML) to (TL);
\draw [->] (BM) to node[right,font=\footnotesize]{$s_{n-1}\!\times\! 1 \!\times\! 1$} (MM);
\draw [>->] (MM) to (TM);
\draw [->] (BR) to node[right,font=\footnotesize]{$\Sigma^d (s_{n-1}\!\times\! 1)_+$} (MR);
\draw [>->] (MR) to (TR);
\end{tikzpicture}
\end{center}
Now, one way of stating the excision theorem is that the map-on-mapping-cones induced by an excisive square is a homology-equivalence. Hence the homology of the right-hand columns of the two diagrams above is the same; in particular, $\widetilde{H}_* C\widetilde{u}_n \cong \widetilde{H}_* C\Sigma^d (s_{n-1} \!\times\! 1)_+$. So:
\begin{align*}
&&\hconn (\widetilde{u}_n) &= \hconn \bigl( \Sigma^d (s_{n-1} \!\times\! 1)_+ \bigr) && \\
&&&= d+ \hconn (s_{n-1} \!\times\! 1)_+ &&\text{\small\itshape by the suspension isomorphism} \\
&&&= d+ \hconn (s_{n-1} \!\times\! 1) && \\
&&&\geq d+ \hconn (s_{n-1}) &&\text{\small\itshape by the K\"{u}nneth theorem.} \qedhere
\end{align*}
\end{proof}
%%%%%%%%%%%%%%%%%%%%%%%%%%%%%%%%%%%%%%%%%%%%%%%%%%%%%%%%%%%%%%%%%%%%%%%%%%%%%%%%%%

%%%%%%%%%%%%%%%%%%%%%%%%%%%%%%%%%%%%%%%%%%%%%%%%%%%%%%%%%%%%%%%%%%%%%%%%%%%%%%%%%%
\section{Proof of the main theorem}\label{sec:proof:of:the:main:theorem}

We now apply the constructions and results of the previous two sections to prove the Main Theorem. This can be rephrased in terms of relative configuration spaces (as defined in \S \ref{subsec:relative:config:spaces}):
\begin{mainthm}
If $M$ is the interior of a connected manifold-with-boundary of dimension at least $2$, and $X$ is a path-connected space, then
\begin{equation}\label{eqn:inductive:hypothesis}
\widetilde{H}_* R_n^+ (M,X) = 0 \quad \text{for} \quad *\leq \tfrac{n-2}{3}.
\end{equation}
\end{mainthm}
%

%%%%%%%%%%%%%%%%%%%%%%%%%%%%%%%%%%%%%%%%%%%%%%%%%%%
\subsection{Strategy of the proof}
%%%%%%%%%%%%%%%%%%%%%%%%%%%%%%%%%%%%%%%%%%%%%%%%%%%

We defined in \S \ref{subsubsec:relative:stabilisation:maps} the `relative double stabilisation map'
\begin{equation*}
\widetilde{s}^2_{(132)} \colon R_{n-2}^+ (M,X) \longrightarrow R_n^+ (M,X).
\end{equation*}
The proof will be by induction on $n$, and the idea is to show, using the inductive hypothesis, that this map is both \emph{surjective} and the \emph{zero-map} on homology, up to the required degree. We will use completely different factorisations of $\widetilde{s}^2_{(132)}$ for each of these. The first will allow us to prove surjectivity-on-homology \emph{piece by piece}, using different methods for the different pieces of the factorisation, and the second (which only exists on homology) will turn out to factor through the zero map in the required range of degrees.
\begin{proof}[Proof of the Main Theorem, by induction on $n$]
Since $M$ and $X$ are path-connected and $\mathrm{dim} (M)\geq 2$, $C_n^+ (M,X)$ is path-connected for all $n$, and hence so is $R_n^+ (M,X)$. So the theorem is true for $n\leq 4$ --- this is the base case.

Now assume $n\geq 5$. By Lemmas \ref{lem:surjectivity:on:homology} and \ref{lem:zero:on:homology} below, the map
\begin{equation*}
(\widetilde{s}^2_{(132)})_* \colon \widetilde{H}_* R_{n-2}^+ (M,X) \longrightarrow \widetilde{H}_* R_n^+ (M,X)
\end{equation*}
is surjective and zero for $*\leq \frac{n-2}{3}$. Hence $\widetilde{H}_* R_n^+ (M,X) = 0$ in this range.
\end{proof}
\noindent Of course the main content of the proof is contained in the proofs of Lemmas \ref{lem:surjectivity:on:homology} and \ref{lem:zero:on:homology} below. We begin with the one asserting \emph{surjectivity} of $(\widetilde{s}^2_{(132)})_*$ for $*\leq \frac{n-2}{3}$.

%%%%%%%%%%%%%%%%%%%%%%%%%%%%%%%%%%%%%%%%%%%%%%%%%%%
\subsection{Surjectivity on homology}
%%%%%%%%%%%%%%%%%%%%%%%%%%%%%%%%%%%%%%%%%%%%%%%%%%%

As noted in Remark \ref{rmk:factorisation:of:relative:sn}, $\widetilde{s}^2_{(132)}$ factorises into:
\begin{equation}\label{eqn:6:maps}
\centering
\begin{split}
\begin{tikzpicture}
[x=1mm,y=1mm,>=stealth']
\node (n1) at (0,15) {$R_{n-2}^+ (M)$};
\node (n2) at (30,15) {$R_{n-2}^+ (M_1)$};
\node (n3) at (60,15) {$R_{n-1}^+ (M)^0$};
\node (n4) at (90,15) {$R_{n-1}^+ (M)$};
\node (s1) at (0,0) {$R_{n-1}^+ (M)$};
\node (s2) at (30,0) {$R_{n-1}^+ (M_1)$};
\node (s3) at (60,0) {$R_n^+ (M)^0$};
\node (s4) at (90,0) {$R_n^+ (M)$};
\draw [->] (n1) to node[above,font=\small]{$\widetilde{p}_{n-2}$} (n2);
\draw [->] (n2) to node[above,font=\small]{$\widetilde{\jmath}_{n-1,0}$} (n3);
\draw [->] (n3) to node[above,font=\small]{$\widetilde{a}_{n-1}$} (n4);
\draw [->] (n4) to node[font=\small,fill=white]{$=$} (s1);
\draw [->] (s1) to node[below,font=\small]{$\widetilde{p}_{n-1}$} (s2);
\draw [->] (s2) to node[below,font=\small]{$\widetilde{\jmath}_{n,0}$} (s3);
\draw [->] (s3) to node[below,font=\small]{$\widetilde{a}_n$} (s4);
\end{tikzpicture}
\end{split}
\end{equation}
which is the mapping cone construction applied to
\begin{equation}\label{eqn:6:squares}
\centering
\begin{split}
\begin{tikzpicture}
[x=1.5mm,y=1mm,>=stealth']
\node (n1) at (0,10) [color=gray] {$\cdot$};
\node (n2) at (10,10) [color=gray] {$\cdot$};
\node (n3) at (20,10) [color=gray] {$\cdot$};
\node (n4) at (30,10) [color=gray] {$\cdot$};
\node (n5) at (40,10) [color=gray] {$\cdot$};
\node (n6) at (50,10) [color=gray] {$\cdot$};
\node (n7) at (60,10) [color=gray] {$\cdot$};
\node (s1) at (0,0) [color=gray] {$\cdot$};
\node (s2) at (10,0) [color=gray] {$\cdot$};
\node (s3) at (20,0) [color=gray] {$\cdot$};
\node (s4) at (30,0) [color=gray] {$\cdot$};
\node (s5) at (40,0) [color=gray] {$\cdot$};
\node (s6) at (50,0) [color=gray] {$\cdot$};
\node (s7) at (60,0) [color=gray] {$\cdot$};
\draw [->] (n1) to node[above,font=\footnotesize]{$p_{n-1}$} (n2);
\draw [->] (n2) to node[above,font=\footnotesize]{$j_{n,0}$} (n3);
\draw [->] (n3) to node[above,font=\footnotesize]{$a_n$} (n4);
\draw [->] (n4) to node[above,font=\footnotesize]{$p_n$} (n5);
\draw [->] (n5) to node[above,font=\footnotesize]{$j_{n+1,0}$} (n6);
\draw [->] (n6) to node[above,font=\footnotesize]{$a_{n+1}$} (n7);
\draw [->] (s1) to node[below,font=\footnotesize]{$p_{n-2}$} (s2);
\draw [->] (s2) to node[below,font=\footnotesize]{$j_{n-1,0}$} (s3);
\draw [->] (s3) to node[below,font=\footnotesize]{$a_{n-1}$} (s4);
\draw [->] (s4) to node[below,font=\footnotesize]{$p_{n-1}$} (s5);
\draw [->] (s5) to node[below,font=\footnotesize]{$j_{n,0}$} (s6);
\draw [->] (s6) to node[below,font=\footnotesize]{$a_n$} (s7);
\draw [->] (s1) to (n1);
\draw [->] (s2) to (n2);
\draw [->] (s3) to (n3);
\draw [->] (s4) to (n4);
\draw [->] (s5) to (n5);
\draw [->] (s6) to (n6);
\draw [->] (s7) to (n7);
\node at (5,5) [font=\small] {$(12)_p$};
\node at (15,5) [font=\small] {$1$};
\node at (25,5) [font=\small] {$1$};
\node at (35,5) [font=\small] {$(12)_p$};
\node at (45,5) [font=\small] {$1$};
\node at (55,5) [font=\small] {$1$};
\end{tikzpicture}
\end{split}
\end{equation}
where vertical maps are stabilisation maps. Recall that $p$ punctures the manifold, $j$ replaces the puncture by a new configuration point which is marked as special, and $a$ forgets which point is special. This is the factorisation we will use to show surjectivity-on-homology.
\begin{lem}\label{lem:surjectivity:on:homology}
Let $n\geq 5$, and assume as inductive hypothesis that \eqref{eqn:inductive:hypothesis} holds for smaller values of $n$. Then $\widetilde{s}^2_{(132)}$ is surjective on homology up to degree $\frac{n-2}{3}$.
\end{lem}
\begin{proof}
We will show that the six maps in \eqref{eqn:6:maps} are each surjective on homology up to this degree.

%%%%%%%%%%%%%%%%%%%%%%%%%%
\subsubsection*{The relative puncturing maps $\widetilde{p}_{n-1}$ and $\widetilde{p}_{n-2}$}
%%%%%%%%%%%%%%%%%%%%%%%%%%

Recall from \S \ref{sec:the:connectivity:of:the:unpuncturing:map} that
\begin{equation*}
\hconn (f) \coloneqq \mathrm{max} \Bigl\lbrace * \Bigl\lvert \begin{smallmatrix} f \text{\itshape\ is surjective on homology up to degree } * \\ f \text{\itshape\ is injective on homology up to degree } *-1 \end{smallmatrix} \Bigr. \Bigr\rbrace .
\end{equation*}
In this notation the inductive hypothesis is
\begin{equation*}
\hconn (s_{n^\prime}) \geq \tfrac{n^\prime -2}{3}, \; \forall n^\prime <n.
\end{equation*}
As noted in Remark \ref{rmk:up:homotopic:to:id:2}, $\widetilde{u}_r \circ \widetilde{p}_r$ is homotopic to the identity, so $(\widetilde{u}_r)_* \circ (\widetilde{p}_r)_* = \id$. Hence $(\widetilde{u}_r)_*$ is injective up to the same degree which $(\widetilde{p}_r)_*$ is surjective up to, so $\hconn (\widetilde{p}_r) = \hconn (\widetilde{u}_r) -1$. Combining this with Proposition \ref{prop:excision} we have
\begin{equation*}
\hconn (\widetilde{p}_r) \geq \hconn (s_{r-1}) + \mathrm{dim} (M) -1,
\end{equation*}
for $r\geq 3$. Using the inductive hypothesis and the fact that $\mathrm{dim} (M)\geq 2$ we get:
\begin{align*}
\hconn (\widetilde{p}_{n-1}) &\geq \tfrac{n-1}{3} &\text{and}&& \hconn (\widetilde{p}_{n-2}) &\geq \tfrac{n-2}{3}.
\end{align*}

%%%%%%%%%%%%%%%%%%%%%%%%%%
\subsubsection*{The relative inclusion-of-the-fibre maps $\widetilde{\jmath}_{n,0}$ and $\widetilde{\jmath}_{n-1,0}$.}
%%%%%%%%%%%%%%%%%%%%%%%%%%

Recall the spectral sequence
\begin{equation}
E_{s,t}^2 \; \cong \; H_s \bigl( \widetilde{C}_{i+1} (M); \widetilde{H}_t (R_{n-i-1}^+ (M_{i+1})) \bigr) \quad \Rightarrow \quad \widetilde{H}_* (R_n^+ (M)^i) \tag{\ref{eqn:RSSS}}
\end{equation}
from Proposition \ref{prop:the:two:spectral:sequences}. The edge homomorphism
\begin{equation*}
\widetilde{H}_t (R_{n-i-1}^+ (M_{i+1})) \cong E^2_{0,t} \twoheadrightarrow E^{\infty}_{0,t} \hookrightarrow \widetilde{H}_t (R_n^+ (M)^i)
\end{equation*}
is the map on $\widetilde{H}_t$ induced by $\widetilde{\jmath}_{n,i}$.

Now, the inductive hypothesis implies that $E^2_{s,t} =0$ for $t\leq \frac{n-i-3}{3}$, so the $E^2$ page is as illustrated in Fig.\ (\ref{fig:3:ss:diagrams}a). Hence in degrees $t\leq\frac{n-i-3}{3}$ the map $\widetilde{\jmath}_{n,i}$ induces $0\to 0$ on $\widetilde{H}_t$, which is trivially surjective. Moreover in the larger range $t\leq\frac{n-i}{3}$ we can see from Fig.\ (\ref{fig:3:ss:diagrams}a) that the inclusion $E^{\infty}_{0,t} \hookrightarrow \widetilde{H}_t (R_n^+ (M)^i)$ is an isomorphism, so $\widetilde{\jmath}_{n,i}$ still induces a surjection on $\widetilde{H}_t$. Setting $i=0$, this proves that $\widetilde{\jmath}_{n,0}$ is surjective on homology up to degree $\frac{n}{3}$. The argument goes through identically when $n$ is replaced by $n-1$, and proves that $\widetilde{\jmath}_{n-1,0}$ is surjective on homology up to degree $\frac{n-1}{3}$.

%%%%%%%%%%%%%%%%%%%%%%%%%%
\subsubsection*{The relative augmentation maps $\widetilde{a}_n$ and $\widetilde{a}_{n-1}$.}
%%%%%%%%%%%%%%%%%%%%%%%%%%

Recall the spectral sequence
\begin{equation}
E_{s,t}^1 \; \cong \; \widetilde{H}_t (R_n^+ (M)^s) \quad \Rightarrow \quad \widetilde{H}_{*+1}C\widetilde{\varepsilon}_n \tag{\ref{eqn:DSS}}
\end{equation}
from Proposition \ref{prop:the:two:spectral:sequences}. The differential $E^1_{-1,t} \longleftarrow E^1_{0,t}$ is the map on $\widetilde{H}_t$ induced by $\widetilde{a}_n$.

Now, as noted above, the spectral sequence \eqref{eqn:RSSS} has $E^2$ page as illustrated in Fig.\ (\ref{fig:3:ss:diagrams}a) --- hence it converges to zero in total degree up to $\frac{n-i-3}{3}$. The limit of \eqref{eqn:RSSS} is the $i$th column of the $E^1$ page of \eqref{eqn:DSS}, so we have a column of zeros on the $E^1$ page of \eqref{eqn:DSS} as shown in Fig.\ (\ref{fig:3:ss:diagrams}b). There is a spectral sequence \eqref{eqn:RSSS} for each $0\leq i\leq n-3$, so there is a triangle of zeros on the $E^1$ page of \eqref{eqn:DSS} as shown in Fig.\ (\ref{fig:3:ss:diagrams}c).

Now assume that $t-1 \leq \frac{n-4}{3}$. Looking at Fig.\ (\ref{fig:3:ss:diagrams}c) we see that the \emph{first} differential is the only possible nontrivial differential hitting $E^1_{-1,t}$. Also, by Corollary \ref{coro:converging:to:zero}, the spectral sequence \eqref{eqn:DSS} converges to zero in total degree $\leq \frac{n-4}{3} \leq n-1$, so we have $E^{\infty}_{-1,t} =0$. Hence the first differential $E^1_{-1,t} \longleftarrow E^1_{0,t}$ must be surjective.

So $\widetilde{a}_n$ induces surjections on $\widetilde{H}_t$ for $t-1\leq\frac{n-4}{3}$, i.e.\ for $t\leq\frac{n-1}{3}$. The argument goes through identically when $n$ is replaced by $n-1$, and proves that $\widetilde{a}_{n-1}$ induces surjections on $\widetilde{H}_t$ for $t\leq\frac{n-2}{3}$.
\end{proof}
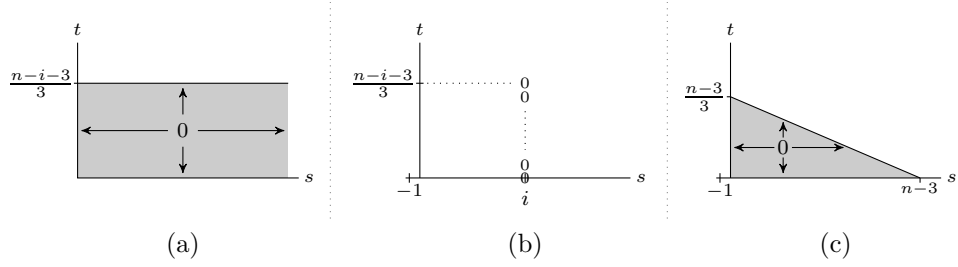
\begin{figure}[ht]
\centering
\begin{tikzpicture}
[x=1mm,y=1mm,yscale=1.8,xscale=1.4,font=\footnotesize,>=stealth']
\newcommand{\xoffseta}{32.5mm}
\newcommand{\xoffsetb}{62mm}
\node at (10,-5) [font=\normalsize] {(a)};
\fill [black!20] (0,0) rectangle (20,7);
\draw (0,10)--(0,0)--(21,0);
\draw ($ (0,7) + (-1pt,0) $) -- (20,7);
\node at (0,11) [font=\scriptsize] {$t$};
\node at (22,0) [font=\scriptsize] {$s$};
\node at (-3.5,7) {$\frac{n-i-3}{3}$};
\node (zero) at (10,3.5) {$0$};
\draw [->] (zero) to ($ (10,7) + (0,-1pt) $);
\draw [->] (zero) to ($ (10,0) + (0,1pt) $);
\draw [->] (zero) to ($ (0,3.5) + (1pt,0) $);
\draw [->] (zero) to ($ (20,3.5) + (-1pt,0) $);
\draw [black!50,dotted] (24,-3)--(24,13);
\begin{scope}[xshift=\xoffseta]
\node at (10,-5) [font=\normalsize] {(b)};
\draw ($ (-1,0) + (-1pt,0) $) -- (20,0);
\draw (0,0)--(0,10);
\draw (-1mm,1pt)--(-1mm,-1pt);
\node at (-1,-1) [font=\scriptsize] {$-1$};
\node at (21,0) [font=\scriptsize] {$s$};
\node at (0,11) [font=\scriptsize] {$t$};
\draw (-1pt,7mm)--(1pt,7mm);
\node at (-3.5,7) {$\frac{n-i-3}{3}$};
\draw (10mm,-1pt)--(10mm,1pt);
\node at (10,-1.5) {$i$};
\node (topzero) at (10,7) [font=\tiny] {$0$};
\node (2ndzero) at (10,6) [font=\tiny] {$0$};
\node (3rdzero) at (10,1) [font=\tiny] {$0$};
\node at (10,0) [font=\tiny] {$0$};
\draw [dotted] (2pt,7mm) to (topzero);
\draw [dotted] (2ndzero) to (3rdzero);
\end{scope}
\draw [black!50,dotted] (56,-3)--(56,13);
\begin{scope}[xshift=\xoffsetb]
\node at (10,-5) [font=\normalsize] {(c)};
\fill [black!20] (0,0) -- (0,6) -- (18,0) -- cycle;
\draw ($ (-1,0) + (-1pt,0) $) -- (20,0);
\draw (0,0)--(0,10);
\draw (-1mm,1pt)--(-1mm,-1pt);
\node at (-1,-1) [font=\scriptsize] {$-1$};
\node at (21,0) [font=\scriptsize] {$s$};
\node at (0,11) [font=\scriptsize] {$t$};
\draw (0,6) -- (18,0);
\draw (-1pt,6mm)--(0,6mm);
\node at (-2.5,6) {$\frac{n-3}{3}$};
\draw (18mm,0)--(18mm,-1pt);
\node at (18,-1) {$\phantom{ }^{n-3}$};
\node (zero) at (5,2.25) [inner sep=1pt] {$0$};
\draw [->] (zero) to (1pt,2.25mm);
\draw [->] (zero) to (5mm,1pt);
\draw [->,shorten >=2pt] (zero) to (intersection of 0,6--18,0 and 0,2.25--1,2.25);
\draw [->,shorten >=1pt] (zero) to (intersection of 0,6--18,0 and 5,0--5,1);
\end{scope}
\end{tikzpicture}
\caption{The two spectral sequences from the proof of surjectivity: (a) is the $E^2$ page of \eqref{eqn:RSSS}; (b) and (c) are the $E^1$ page of \eqref{eqn:DSS}.}\label{fig:3:ss:diagrams}
\end{figure}

%%%%%%%%%%%%%%%%%%%%%%%%%%%%%%%%%%%%%%%%%%%%%%%%%%%
\subsection{Zero on homology}
%%%%%%%%%%%%%%%%%%%%%%%%%%%%%%%%%%%%%%%%%%%%%%%%%%%

The factorisation of $\widetilde{s}^2_{(132)}$ (on homology) we will use for this part comes from a more general factorisation lemma, so we begin by stating this.

%%%%%%%%%%%%%%%%%%%%%%%%%%%%%%%%%%%%%%%
\subsubsection{A general factorisation lemma.}
%%%%%%%%%%%%%%%%%%%%%%%%%%%%%%%%%%%%%%%

As before, we let $Cf$ denote the mapping cone of a map $f$. Suppose we have a homotopy-commutative square of maps:
\begin{equation}\label{eqn:square:S}\tag{$\calS$}
\centering
\begin{split}
\begin{tikzpicture}
[x=1.2mm,y=1.2mm,>=stealth']
\node (tl) at (0,10) {$A$};
\node (tr) at (20,10) {$X$};
\node (bl) at (0,0) {$B$};
\node (br) at (20,0) {$Y$};
\draw [->] (tl) to node[above,font=\small]{$i$} (tr);
\draw [->] (bl) to node[below,font=\small]{$j$} (br);
\draw [->] (tl) to node[left,font=\small]{$f$} (bl);
\draw [->] (tr) to node[right,font=\small]{$g$} (br);
\end{tikzpicture}
\end{split}
\end{equation}
Choosing any particular homotopy $H\colon g\circ i \simeq j\circ f$ to fill this square induces a map $CH\colon Ci \to Cj$, and completes an exact `ladder' on homology:
\begin{equation}\label{eqn:exact:ladder}
\centering
\begin{split}
\begin{tikzpicture}
[x=1.2mm,y=1.2mm,>=stealth']
\node (t1) at (0,10) {$\cdots$};
\node (t2) at (20,10) {$\widetilde{H}_* X$};
\node (t3) at (40,10) {$\widetilde{H}_* Ci$};
\node (t4) at (60,10) {$\widetilde{H}_{*-1} A$};
\node (t5) at (80,10) {$\cdots$};
\node (b1) at (0,0) {$\cdots$};
\node (b2) at (20,0) {$\widetilde{H}_* Y$};
\node (b3) at (40,0) {$\widetilde{H}_* Cj$};
\node (b4) at (60,0) {$\widetilde{H}_{*-1} B$};
\node (b5) at (80,0) {$\cdots$};
\draw [->] (t1) to (t2);
\draw [->] (t2) to (t3);
\draw [->] (t3) to (t4);
\draw [->] (t4) to (t5);
\draw [->] (b1) to (b2);
\draw [->] (b2) to (b3);
\draw [->] (b3) to (b4);
\draw [->] (b4) to (b5);
\draw [->] (t2) to node[right,font=\small]{$g_*$} (b2);
\draw [->] (t3) to node[right,font=\small]{$CH_*$} (b3);
\draw [->] (t4) to node[right,font=\small]{$f_*$} (b4);
\end{tikzpicture}
\end{split}
\end{equation}
We say that \eqref{eqn:square:S} \emph{splits into triangles} if there exists a map $d\colon X \to B$, together with homotopies $F_1 \colon d\circ i \simeq f$, $F_2 \colon g \simeq j\circ d$. In other words the square can be filled in as:
\begin{equation}\label{eqn:splitting:into:triangles}
\centering
\begin{split}
\begin{tikzpicture}
[x=1.2mm,y=1.2mm,>=stealth']
\node (tl) at (0,10) {$A$};
\node (tr) at (20,10) {$X$};
\node (bl) at (0,0) {$B$};
\node (br) at (20,0) {$Y$};
\draw [->] (tl) to node[above,font=\small]{$i$} (tr);
\draw [->] (bl) to node[below,font=\small]{$j$} (br);
\draw [->] (tl) to node[left,font=\small]{$f$} (bl);
\draw [->] (tr) to node[right,font=\small]{$g$} (br);
\draw [->] (tr) to (bl);
\node at (5,7) {$\simeq$};
\node at (15,3) {$\simeq$};
\end{tikzpicture}
\end{split}
\end{equation}
\begin{lem}[``Factorisation lemma'']\label{lem:factorisation}
If the square \eqref{eqn:square:S} splits into triangles, and $H$ is any homotopy filling this square, then $CH_*$ factors through a map $z_H \colon \widetilde{H}_{*-1} A \to \widetilde{H}_* Y$ in diagram \eqref{eqn:exact:ladder}. Hence in particular the composition $\widetilde{H}_* X \to \widetilde{H}_* Cj$ in \eqref{eqn:exact:ladder} is zero.

Moreover, $z_H$ itself factorises as follows\textup{:}
\begin{equation*}
\widetilde{H}_{*-1} A \hookrightarrow \widetilde{H}_* (S^1 \times A) \xrightarrow{\gamma_*} \widetilde{H}_* Y,
\end{equation*}
where the first map is the inclusion of a direct summand in the K\"{u}nneth splitting $\widetilde{H}_* (S^1 \times A) \cong \widetilde{H}_{*-1} (A) \oplus H_{*-1} (pt) \oplus \widetilde{H}_* (A)$, and the second map is induced by the self-homotopy $\gamma\colon S^1 \times A \to Y$ built out of $H$ and the two homotopies $F_1$ and $F_2$ occurring in \eqref{eqn:splitting:into:triangles}.
\end{lem}
\begin{proof}\renewcommand{\qedsymbol}{}
See \ref{appendix:proof:of:the:factorisation:lemma}.
\end{proof}

%%%%%%%%%%%%%%%%%%%%%%%%%%%%%%%%%%%%%%%
\subsubsection{Applying the factorisation lemma.}
%%%%%%%%%%%%%%%%%%%%%%%%%%%%%%%%%%%%%%%

In particular we may take \eqref{eqn:square:S} to be the square
\begin{center}
\begin{tikzpicture}
[x=1.5mm,y=1.2mm,>=stealth']
\node (tl) at (0,10) {$C_{n-2}^+ (M)$};
\node (tr) at (20,10) {$C_{n-1}^+ (M)$};
\node (bl) at (0,0) {$C_n^+ (M)$};
\node (br) at (20,0) {$C_{n+1}^+ (M)$};
\draw [->] (tl) to node[above,font=\small]{$s$} (tr);
\draw [->] (bl) to node[below,font=\small]{$s$} (br);
\draw [->] (tl) to node[left,font=\small]{$-s^2$} (bl);
\draw [->] (tr) to node[right,font=\small]{$-s^2$} (br);
\end{tikzpicture}
\end{center}
(for $n\geq 3$). This is the right-hand square from \eqref{eqn:squares:of:stabilisation:maps}. It splits into triangles, since we may for example take the diagonal map to be $-s\colon C_{n-1}^+ (M) \to C_n^+ (M)$, and the two homotopies to be constant. (See also Remark \ref{rmk:factorisation:of:relative:sn}.) Taking $H$ to be the homotopy $(132)$, as defined in \S \ref{subsubsec:relative:stabilisation:maps}, Lemma \ref{lem:factorisation} implies the following factorisation of $(\widetilde{s}^2_{(132)})_*$:
\begin{coro}\label{coro:factorisation}
The map $(\widetilde{s}^2_{(132)})_* \colon \widetilde{H}_* R_{n-2}^+ (M) \to \widetilde{H}_* R_n^+ (M)$ factorises as follows\textup{:}
\begin{equation*}
\widetilde{H}_* R_{n-2}^+ (M) \to \widetilde{H}_{*-1} C_{n-2}^+ (M) \hookrightarrow \widetilde{H}_* (S^1 \!\times\! C_{n-2}^+ (M)) \xrightarrow{\gamma_*} \widetilde{H}_* C_{n+1}^+ (M) \to \widetilde{H}_* R_n^+ (M).
\end{equation*}
The first and last maps come from the long exact sequences for $C_{n-2}^+ (M) \to C_{n-1}^+ (M)$ and $C_n^+ (M) \to C_{n+1}^+ (M)$ respectively, the second map comes from the K\"{u}nneth splitting of $\widetilde{H}_* \bigl( S^1 \!\times\! C_{n-2}^+ (M) \bigr)$, and
\begin{equation*}
\gamma\colon S^1 \times C_{n-2}^+ (M) \to C_{n+1}^+ (M)
\end{equation*}
is the self-homotopy $(132)$.
\end{coro}
\begin{proof}
This is immediate from Lemma \ref{lem:factorisation}, once we note that in this case we can take the split homotopy \eqref{eqn:splitting:into:triangles} to be the \emph{constant} homotopy, so that $\gamma$ is just $H=(132)$.
\end{proof}
\noindent Rephrasing the definition of the homotopy $(132)$ in \S \ref{subsubsec:relative:stabilisation:maps}, we may describe $\gamma$, as just a \emph{map} $S^1 \!\times\! C_{n-2}^+ (M) \to C_{n+1}^+ (M)$, concretely as follows:
\begin{center}
\begin{tikzpicture}
[x=1mm,y=1mm]
\node at (-15,0) {$(t,c) \quad \mapsto$};
\fill[black!15] (10,4) arc (90:270:2 and 4) -- (0,-4) arc (270:90:2 and 4) -- cycle;
\draw[dashed] (0,4)--(5,4);
\draw (5,4)--(20,4);
\draw[dashed] (0,-4)--(5,-4);
\draw (5,-4)--(20,-4);
\draw (20,0) ellipse (2 and 4);
\draw (10,4) arc (90:270:2 and 4);
\node at (23.5,-3) [font=\footnotesize] {$B_0$};
\node at (3,0) {$c$};
\draw (13,0) circle (2);
\node[circle,fill,inner sep=1pt] at ($ (13,0) + (0:2) $) {};
\node[circle,fill,inner sep=1pt] at ($ (13,0) + (120:2) $) {};
\node[circle,fill,inner sep=1pt] at ($ (13,0) + (240:2) $) {};
\end{tikzpicture}
\end{center}
The configuration $c$ is pushed away from the chosen boundary-component $B_0$, and three new points are added on a small embedded circle near $B_0$, at the positions $\lbrace t^{\nicefrac13}, \omega t^{\nicefrac13}, \omega^2 t^{\nicefrac13} \rbrace$ where $\omega = \mathrm{exp}(\frac23 \pi i)$. Fix an orientation of the circle: this gives the three new points a cyclic ordering $[p_1,p_2,p_3]$, and we use the orientation convention $[c,p_1,p_2,p_3]$.

We can use this description to check that $\gamma$ is natural w.r.t.\ stabilisation maps:
\begin{lem}\label{lem:naturality:of:gamma}
The following square is commutative up to homotopy\textup{:}
\begin{center}
\begin{tikzpicture}
[x=2mm,y=1.5mm,>=stealth']
\node (tl) at (0,10) {$S^1 \times C_{n-2}^+ (M)$};
\node (tr) at (20,10) {$C_{n+1}^+ (M)$};
\node (bl) at (0,0) {$S^1 \times C_{n-3}^+ (M)$};
\node (br) at (20,0) {$C_n^+ (M)$};
\draw [->] (tl) to node[above,font=\small]{$\gamma$} (tr);
\draw [->] (bl) to node[below,font=\small]{$\gamma$} (br);
\draw [->] (bl) to node[left,font=\small]{$1 \times (-s)$} (tl);
\draw [->] (br) to node[right,font=\small]{$s$} (tr);
\end{tikzpicture}
\end{center}
\end{lem}
\begin{proof}
The two ways around this square are both of the form
\begin{equation*}
S^1 \times C_{n-3}^+ (M) \xrightarrow{\text{?} \times 1} C_4^+ (\bbR^d) \times C_{n-3}^+ (M) \longrightarrow C_{n+1}^+ (M),
\end{equation*}
where the second map is
\begin{center}
\begin{tikzpicture}
[x=1mm,y=1mm]
\node at (-15,0) {$(c_0,c) \quad \mapsto$};
\fill[black!15] (10,4) arc (90:270:2 and 4) -- (0,-4) arc (270:90:2 and 4) -- cycle;
\draw[dashed] (0,4)--(5,4);
\draw (5,4)--(20,4);
\draw[dashed] (0,-4)--(5,-4);
\draw (5,-4)--(20,-4);
\draw (20,0) ellipse (2 and 4);
\draw (10,4) arc (90:270:2 and 4);
\node at (23.5,-3) [font=\footnotesize] {$B_0$};
\node at (3,0) {$c$};
\draw[densely dotted] (10,-2) rectangle (16,2);
\node at (13,0) [font=\small] {$c_0$};
\end{tikzpicture}
\end{center}
Here, the configuration $c$ is pushed away from $B_0$, and the configuration $c_0$ is inserted into a coordinate neighbourhood near $B_0$ (and we use the orientation convention $[c,c_0]$).

The map $\text{`?'}\colon S^1 \longrightarrow C_4^+ (\bbR^d)$ is either
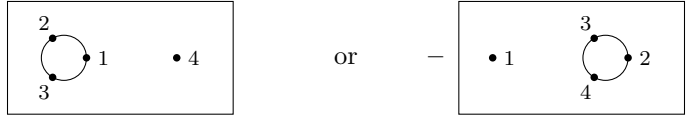
\begin{center}
\begin{tikzpicture}
[x=1mm,y=1mm,scale=1.5]
\draw (0,-5) rectangle (20,5);
\draw (5,0) circle (2);
\node[circle,fill,inner sep=1pt] at ($ (5,0) + (0:2) $) {};
\node[circle,fill,inner sep=1pt] at ($ (5,0) + (120:2) $) {};
\node[circle,fill,inner sep=1pt] at ($ (5,0) + (240:2) $) {};
\node[circle,fill,inner sep=1pt] at (15,0) {};
\node at ($ (5,0) + (0:3.5) $) [font=\footnotesize] {1};
\node at ($ (5,0) + (120:3.5) $) [font=\footnotesize] {2};
\node at ($ (5,0) + (240:3.5) $) [font=\footnotesize] {3};
\node at (16.5,0) [font=\footnotesize] {4};
\begin{scope}[xshift=40mm]
\draw (0,-5) rectangle (20,5);
\draw (13,0) circle (2);
\node[circle,fill,inner sep=1pt] at ($ (13,0) + (0:2) $) {};
\node[circle,fill,inner sep=1pt] at ($ (13,0) + (120:2) $) {};
\node[circle,fill,inner sep=1pt] at ($ (13,0) + (240:2) $) {};
\node[circle,fill,inner sep=1pt] at (3,0) {};
\node at ($ (13,0) + (0:3.5) $) [font=\footnotesize] {2};
\node at ($ (13,0) + (120:3.5) $) [font=\footnotesize] {3};
\node at ($ (13,0) + (240:3.5) $) [font=\footnotesize] {4};
\node at (4.5,0) [font=\footnotesize] {1};
\end{scope}
\node at (38,0) {\textbf{--}};
\node at (30,0) {or};
\end{tikzpicture}
\end{center}
(the numberings represent orientations of the configurations; the `\textbf{--}' on the right diagram indicates that the orientation should in fact be the \emph{opposite} of that illustrated). So it is enough to find a homotopy $h\colon S^1 \times I \longrightarrow C_4^+ (\bbR^d)$ connecting these two maps. Such a homotopy clearly does exist: for example define $h(t,u)$ to be
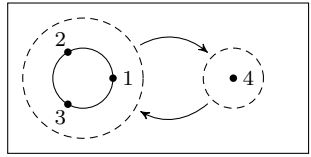
\begin{center}
\begin{tikzpicture}
[x=1mm,y=1mm,scale=2]

\draw (0,-5) rectangle (20,5);
\draw (5,0) circle (2);
\node[circle,fill,inner sep=1pt] at ($ (5,0) + (0:2) $) {};
\node[circle,fill,inner sep=1pt] at ($ (5,0) + (120:2) $) {};
\node[circle,fill,inner sep=1pt] at ($ (5,0) + (240:2) $) {};
\node[circle,fill,inner sep=1pt] at (15,0) {};
\node at ($ (5,0) + (0:3) $) [font=\footnotesize] {1};
\node at ($ (5,0) + (120:3) $) [font=\footnotesize] {2};
\node at ($ (5,0) + (240:3) $) [font=\footnotesize] {3};
\node at (16,0) [font=\footnotesize] {4};

\draw[densely dashed] (5,0) circle (4);
\draw[densely dashed] (15,0) circle (2);
\draw[->,>=stealth'] ($ (5,0) + (30:4.4) $) to [out=30,in=140] ($ (15,0) + (135:2.4) $);
\draw[->,>=stealth'] ($ (15,0) + (225:2.4) $) to [out=220,in=330] ($ (5,0) + (330:4.4) $);

\end{tikzpicture}
\end{center}
where $t\in S^1$ determines the positions of the $3$ points on the circle, and $u\in I$ determines how far along the arrows to move the dotted regions.
\end{proof}

%%%%%%%%%%%%%%%%%%%%%%%%%%%%%%%%%%%%%%%
\subsubsection{Zero on homology.}
%%%%%%%%%%%%%%%%%%%%%%%%%%%%%%%%%%%%%%%

Finally, we may apply our new factorisation of $(\widetilde{s}^2_{(132)})_*$ to deduce that it is zero in the required range:
\begin{lem}\label{lem:zero:on:homology}
Let $n\geq 4$, and assume as inductive hypothesis that \eqref{eqn:inductive:hypothesis} holds for smaller values of $n$. Then $\widetilde{s}^2_{(132)}$ is the zero map on (reduced) homology up to degree $\frac{n-2}{3}$.
\end{lem}
\begin{proof}
By Corollary \ref{coro:factorisation}, Lemma \ref{lem:naturality:of:gamma}, and the naturality of the K\"{u}nneth splitting we have a commutative diagram
\begin{center}
\begin{tikzpicture}
[x=1.5mm,y=1.2mm,>=stealth']
\node (t1) at (0,10) {$\widetilde{H}_* R_{n-2}^+ (M)$};
\node (t2) at (20,10) {$\widetilde{H}_{*-1} C_{n-2}^+ (M)$};
\node (t3) at (40,10) {$\widetilde{H}_* C_{n+1}^+ (M)$};
\node (t4) at (60,10) {$\widetilde{H}_* R_n^+ (M)$};
\node (b2) at (20,0) {$\widetilde{H}_{*-1} C_{n-3}^+ (M)$};
\node (b3) at (40,0) {$\widetilde{H}_* C_n^+ (M)$};
\draw [->] (t1) to (t2);
\draw [->] (t2) to (t3);
\draw [->] (t3) to (t4);
\draw [->] (b2) to (b3);
\draw [->] (b2) to node[left,font=\small]{$(-s)_*$} (t2);
\draw [->] (b3) to node[right,font=\small]{$s_*$} (t3);
\draw [->] (b3.east) to node[below right=-1pt and -1pt,font=\small]{0} (t4);
\end{tikzpicture}
\end{center}
where the composition along the top row is $(\widetilde{s}^2_{(132)})_*$. The composition on the right is zero since it is induced by a cofibration sequence. By definition, the maps $\pm s$ differ only by an automorphism of their common codomain, so (as noted in Remark \ref{rmk:s:and:minus:s:difference}) they have the same surjectivity-on-homology properties. Hence by the inductive hypothesis $(-s)_*$ is surjective for $*-1\leq \frac{n-5}{3}$, i.e.\ for $*\leq \frac{n-2}{3}$. So in this range $(\widetilde{s}^2_{(132)})_*$ factors through the zero map, and hence is itself zero.
\end{proof}
%%%%%%%%%%%%%%%%%%%%%%%%%%%%%%%%%%%%%%%%%%%%%%%%%%%%%%%%%%%%%%%%%%%%%%%%%%%%%%%%%%

%%%%%%%%%%%%%%%%%%%%%%%%%%%%%%%%%%%%%%%%%%%%%%%%%%%%%%%%%%%%%%%%%%%%%%%%%%%%%%%%%%
\section{Corollaries}\label{sec:corollaries}

%%%%%%%%%%%%%%%%%%%%%%%%%%%%%%%%%%%%%%%%%%%%%%%%%%%
\subsection{Stability for generalised homology theories}
%%%%%%%%%%%%%%%%%%%%%%%%%%%%%%%%%%%%%%%%%%%%%%%%%%%

First we will prove Corollary \ref{coro:corollary:B} (stated in \S \ref{subsec:corollaries}). This follows directly from the Main Theorem and the following lemma:
\begin{lem}\label{lem:conn:gen:hom:th}
If $h_*$ is a connective generalised homology theory with connectivity $c$ \textup{(}i.e.\ its associated spectrum has connectivity $c$\textup{)}, and if the map $f\colon X \to Y$ is an isomorphism on $H_* (-;\bbZ)$ up to degree $k-1$ and surjective up to degree $k$, then $f$ is an isomorphism on $h_*$ up to degree $k-1+c$ and surjective on $h_*$ up to degree $k+c$.
\end{lem}
\begin{proof}
By the long exact sequence for cofibration sequences, this is equivalent to the claim that
\begin{equation*}
\widetilde{H}_* (Cf;\bbZ) = 0 \; \forall *\leq k \quad \Rightarrow \quad \widetilde{h}_* (Cf) = 0 \; \forall *\leq k+c.
\end{equation*}
If $E$ is the spectrum associated to $h_*$, then we have the Atiyah--Hirzebruch spectral sequence (see \cite[Theorem 11.6]{McCleary1985}):
\begin{equation*}
E_{s,t}^2 \cong H_s (Cf; \pi_t (E)) \quad \Rightarrow \quad h_* (Cf).
\end{equation*}
Removing an $H_s (pt; \pi_t (E))$ summand from the $E^2$ page, and correspondingly an $h_* (pt) = \pi_* (E)$ summand from the limit, gives the reduced version:
\begin{equation*}
E_{s,t}^2 \cong \widetilde{H}_s (Cf; \pi_t (E)) \quad \Rightarrow \quad \widetilde{h}_* (Cf).
\end{equation*}
By the Universal Coefficient Theorem, and since $E$ is $c$-connected, the $E^2$ page is zero for $s\leq k$ or $t\leq c$. Therefore the limit is zero for total degrees $*\leq k+c$.
\end{proof}
\begin{rmk}
Alternatively, one could consider the map of (non-reduced) Atiyah--Hirzebruch spectral sequences induced by $f$, and apply the Zeeman comparison theorem \cite{Zeeman1957}.
\end{rmk}
\noindent We now revert to talking only about ordinary homology again, but of course the corollaries for sequences of groups below also have similar generalised homology versions.

%%%%%%%%%%%%%%%%%%%%%%%%%%%%%%%%%%%%%%%%%%%%%%%%%%%
\subsection{Wreath products with alternating braid groups}
%%%%%%%%%%%%%%%%%%%%%%%%%%%%%%%%%%%%%%%%%%%%%%%%%%%

Let $S$ be the interior of a connected surface-with-boundary \sbar, and $G$ be any discrete group.
\begin{defn}
The \emph{braid group on $n$ strands on $S$} is $\beta_n^S \coloneqq \pi_1 C_n (S,pt)$. When $S=\bbR^2$ this recovers the definition of the Artin braid group $\beta_n$ (by \cite{Fox:Neuwirth1962}). A based loop in $C_n (S,pt)$ induces a permutation of the basepoint configuration, so there is a natural projection $\beta_n^S \twoheadrightarrow \Sigma_n$. The \emph{alternating braid group on $n$ strands on $S$}, $A\beta_n^S$, is defined to be the index-$2$ subgroup of braids whose induced permutation is even. A loop in $C_n (S,pt)$ induces an even permutation iff it lifts to a loop in $C_n^+ (S,pt)$, so this is equivalent to defining $A\beta_n^S \coloneqq \pi_1 C_n^+ (S,pt)$.

The wreath product $G\wr A\beta_n^S$ is defined to be the semi-direct product
\begin{equation}\label{eqn:wreath:product}
1\to G^n \hookrightarrow G^n \rtimes A\beta_n^S \twoheadrightarrow A\beta_n^S \to 1
\end{equation}
where $A\beta_n^S$ acts on $G^n$ by permuting the $n$ factors through its projection to $A_n \leq \Sigma_n$.
\end{defn}
\noindent The first half of Corollary \ref{coro:corollary:A} (see \S \ref{subsec:corollaries}) follows directly from the Main Theorem and the following lemma:
\begin{lem}\label{lem:classifying:space}
Pick a model for the classifying space $BG$. Then $C_n^+ (S,BG)$ is a model for the classifying space $B(G\wr A\beta_n^S)$.
\end{lem}
\begin{proof}
First we show that $C_n^+ (S,BG)$ is aspherical. In the case where \sbar\ is compact, using the classification of compact connected surfaces-with-boundary we can draw an explicit deformation retraction from \sbar\ onto a wedge of circles, so it is aspherical. In general, any map of a sphere into \sbar\ will have image contained in a \emph{compact} connected subsurface-with-boundary, so \sbar\ is also aspherical without the compactness assumption. Hence $S$ is aspherical. Moreover, $S\setminus \lbrace \text{finitely many points} \rbrace$ is again the interior of a connected surface-with-boundary, so is also aspherical by the previous argument.

Via the fibration sequences
\begin{equation*}
S_{n-1} \times BG \longrightarrow \widetilde{C}_n (S,BG) \longrightarrow \widetilde{C}_{n-1} (S,BG)
\end{equation*}
and induction on $n$, this implies that $\widetilde{C}_n (S,BG)$ is aspherical for all $n$. This is a covering space of $C_n^+ (S,BG)$, so $C_n^+ (S,BG)$ is also aspherical for all $n$.

Now we check that $\pi_1 C_n^+ (S,BG) \cong G\wr A\beta_n^S$. Forgetting the labels gives a fibration
\begin{equation*}
(BG)^n \hookrightarrow C_n^+ (S,BG) \xrightarrow{\text{\footnotesize\itshape forget}} C_n^+ (S,pt),
\end{equation*}
which admits a section. So on $\pi_1$ this induces a split short exact sequence
\begin{equation*}
1\to G^n \hookrightarrow G^n \rtimes A\beta_n^S \twoheadrightarrow A\beta_n^S \to 1.
\end{equation*}
It remains to show that this is the \emph{same} semi-direct product as $G\wr A\beta_n^S$, \eqref{eqn:wreath:product}. This can be seen most easily by just thinking about what concatenation of based loops in $C_n^+ (S,BG)$ does under this identification: it concatenates the corresponding braids, and multiplies the elements of $G$ in pairs, according to which strands have been glued together. So the multiplication in the $G^n$ component is twisted by the induced permutation coming from the $A\beta_n^S$ component.
\end{proof}
\begin{eg}
A special case of the first half of Corollary \ref{coro:corollary:A}, taking $S=\bbR^2$ and $G=*$, is homological stability for the \emph{alternating Artin braid groups}, an index-$2$ subfamily of the sequence of Artin braid groups. Another special case, taking $S=\bbR^2$ and $G=\bbZ$, is homological stability for the sequence of \emph{alternating ribbon braid groups}.
\end{eg}
\begin{rmk}
The elements of $G\wr A\beta_n^S$ can be thought of as braids embedded in $S\times I$, with an element of $G$ `attached' to each strand. In this description the `natural map' $G\wr A\beta_n^S \to G\wr A\beta_{n+1}^S$ referred to in the statement of Corollary \ref{coro:corollary:A} is given by adding a new strand (with the identity of $G$ attached) near a chosen boundary-component of \sbar.
\end{rmk}

%%%%%%%%%%%%%%%%%%%%%%%%%%%%%%%%%%%%%%%%%%%%%%%%%%%
\subsection{Wreath products with alternating groups}
%%%%%%%%%%%%%%%%%%%%%%%%%%%%%%%%%%%%%%%%%%%%%%%%%%%

We now want to take configurations in the `manifold' $M=\bbR^\infty$:
\begin{coro}\label{coro:R:infty:version}
For any path-connected space $X$, the map
\begin{equation*}
s\colon C_n^+ (\bbR^\infty ,X) \longrightarrow C_{n+1}^+ (\bbR^\infty ,X)
\end{equation*}
is an isomorphism on homology up to degree $\frac{n-5}{3}$ and surjective up to degree $\frac{n-2}{3}$.
\end{coro}
\begin{proof}
By the Main Theorem, the analogous statement is true for
\begin{equation*}
C_n^+ (\bbR^N ,X) \longrightarrow C_{n+1}^+ (\bbR^N ,X)
\end{equation*}
for all $N$. These fit into a commutative ladder of maps
\raisebox{0pt}[0pt][0pt]{%
\tikz[baseline=0.4ex,x=1mm,y=1mm,scale=1,>=stealth']{%
\draw [->,very thin] (0,0)--(0,3.7);
\draw [->,very thin] (4,0)--(4,3.7);
\draw [->,very thin] (0,1)--(4,1);
\draw [->,very thin] (0,2)--(4,2);
\node at (2,2.7) [font=\tiny] {$\cdot$};
\node at (2,3.1) [font=\tiny] {$\cdot$};
\node at (2,3.5) [font=\tiny] {$\cdot$};
\useasboundingbox (4.5,0);
}},
where the vertical maps are induced by the standard inclusions $\bbR^N \hookrightarrow \bbR^{N+1}$, and the map we are interested in is the vertical colimit of this ladder. Injectivity- and surjectivity-on-$H_*$ properties of the horizontal maps are preserved under taking this colimit, so the result follows.
\end{proof}
\begin{rmk}
This corollary depends on having an explicit range for homological stability which is \emph{independent} of the manifold $M$. If we only knew qualitatively that homological stability held for \emph{some} (unknown) range, then we would not have been able to take a direct limit and keep homological stability, as we did in the proof above. (A priori, the stability slope could $\to 0$ as the dimension of $M\to\infty$, for example.)
\end{rmk}
\begin{rmk}
We note that $\inj ([n],\bbR^\infty)$ is contractible, and the action of $A_n$ on it is free, so it is a model for $EA_n$. This means that the oriented configuration space on $\bbR^\infty$ with $X$-labels is a model for the homotopy quotient, or Borel construction:
\begin{equation*}
C_n^+ (\bbR^\infty ,X) \; = \; \inj ([n],\bbR^\infty) \times_{A_n} X^n \; \simeq \; EA_n \times_{A_n} X^n \; = \; X^n /\!\!/ A_n .
\end{equation*}
So by Corollary \ref{coro:R:infty:version} we have homological stability for the sequence
\begin{equation*}
\cdots \to X^n /\!\!/ A_n \to X^{n+1} /\!\!/ A_{n+1} \to \cdots .
\end{equation*}
\end{rmk}
\noindent In the special case $X=BG$, we have the following:
\begin{coro}[Second half of Corollary \ref{coro:corollary:A}]
For any discrete group $G$, the map ${G\wr A_n \to G\wr A_{n+1}}$ is an isomorphism on homology up to degree $\frac{n-5}{3}$ and surjective up to degree $\frac{n-2}{3}$.
\end{coro}
\noindent Here, the wreath product $G\wr A_n$ is the semi-direct product
\begin{equation}\label{eqn:wreath:product:2}
1\to G^n \hookrightarrow G^n \rtimes A_n \twoheadrightarrow A_n \to 1
\end{equation}
where $A_n$ acts by permuting the $n$ factors of $G^n$.
\begin{proof}
By Corollary \ref{coro:R:infty:version} we just need to show that $C_n^+ (\bbR^\infty ,BG)$ is a model for the classifying space $B(G\wr A_n)$. Now, $\bbR^\infty \setminus \lbrace \text{finitely many points} \rbrace$ is contractible, so by considering the fibration sequences
\begin{equation*}
(\bbR^\infty \setminus \lbrace n-1 \text{ points} \rbrace) \times BG \hookrightarrow \widetilde{C}_n (\bbR^\infty ,BG) \twoheadrightarrow \widetilde{C}_{n-1} (\bbR^\infty ,BG)
\end{equation*}
we can show inductively that $\widetilde{C}_n (\bbR^\infty ,BG)$, and hence also $C_n^+ (\bbR^\infty ,BG)$, is aspherical for all $n$.

To show that $\pi_1 C_n^+ (\bbR^\infty ,BG) \cong G\wr A_n$, we first consider $\pi_1 C_n^+ (\bbR^\infty ,pt)$. A based loop (up to $\simeq$) in $C_n^+ (\bbR^\infty ,pt)$ is an $n$-strand braid on $\bbR^\infty$. Any braid in $\bbR^\infty$ can be `untangled', so it is just a permutation of the basepoint configuration, which in this case must be \emph{even} to preserve the orientation. So $\pi_1 C_n^+ (\bbR^\infty ,pt) \cong A_n$. As in the proof of Lemma \ref{lem:classifying:space} we have a fibration
\begin{equation*}
(BG)^n \hookrightarrow C_n^+ (\bbR^\infty ,BG) \xrightarrow{\text{\footnotesize\itshape forget}} C_n^+ (\bbR^\infty ,pt),
\end{equation*}
which admits a section, so on $\pi_1$ we have a split short exact sequence
\begin{equation*}
1\to G^n \hookrightarrow G^n \rtimes A_n \twoheadrightarrow A_n \to 1.
\end{equation*}
By considering what concatenation of based loops in $C_n^+ (\bbR^\infty ,BG)$ does under this identification, we can see that the action of $A_n$ on $G^n$ in this semi-direct product is just permutation of the $n$ factors, as in \eqref{eqn:wreath:product:2}. Hence $\pi_1 C_n^+ (\bbR^\infty ,BG) \cong G\wr A_n$.
\end{proof}

%%%%%%%%%%%%%%%%%%%%%%%%%%%%%%%%%%%%%%%%%%%%%%%%%%%
\subsection{Aside: the limiting spaces for $A_n$ and $A\beta_n$}
%%%%%%%%%%%%%%%%%%%%%%%%%%%%%%%%%%%%%%%%%%%%%%%%%%%

When $S=\bbR^2$ we denote the alternating braid group $A\beta_n^{\bbR^2}$ by just $A\beta_n$.

Corollary \ref{coro:corollary:A} relates the homology of the families of groups $(A_n)$ and $(A\beta_n)$, in the stable range, to the homology of the `limiting spaces' $BA_{\infty}^+$ and $BA\beta_{\infty}^+$, where $G_\infty = \lim_n G_n$ and $(\cdot)^+$ is the Quillen plus-construction. In these two cases we can identify the limiting spaces explicitly: The `scanning' argument of Segal and McDuff implies \cite[Theorem 4.5]{McDuff1975} that
\begin{align}
B\Sigma_{\infty}^+ &\simeq \Omega_0^\infty S^\infty = \calQ_0 S^0 &\text{and}&& B\beta_{\infty}^+ &\simeq \Omega_0^2 S^2 \simeq \Omega^2 S^3 .\label{eqn:lim:space:1}\\
\intertext{%
(The first of these is the Barratt--Priddy--Quillen theorem \cite{Barratt:Priddy1972}.) Plus-constructing preserves double-covering maps (see for example \cite[Theorem 6.4]{Berrick1982}), so%
}
BA_{\infty}^+ &\simeq \widetilde{\calQ_0 S^0} &\text{and}&& BA\beta_{\infty}^+ &\simeq \widetilde{\Omega^2 S^3} ,\label{eqn:lim:space:2}\\
\intertext{%
the universal cover of $\calQ_0 S^0$ and the unique connected double cover of $\Omega^2 S^3$. Let $\Cob_n$ denote the category of $(n-1)$-dimensional manifolds and $n$-dimensional cobordisms between them (embedded in $\bbR^\infty$), as defined and studied in \cite{GMTW2009}, and let $\Cob_n (\bbR^2)$ denote the version with embeddings into $\bbR^2$. In this language \eqref{eqn:lim:space:1} can be reinterpreted (by the group-completion theorem) as%
}
\Omega B\Cob_0 &\simeq \calQ S^0 &\text{and}&& \Omega B\Cob_0 (\bbR^2) &\simeq \Omega^2 S^2 .\label{eqn:lim:space:3}\\
\intertext{%
Now if $\Cob_0^+$, $\Cob_0^+ (\bbR^2)$ denote the corresponding (embedded) cobordism categories where $0$-manifolds have an ordering-up-to-even-permutations (this is a non-tangential, i.e.\ `global', structure), then by the group-completion theorem \eqref{eqn:lim:space:2} becomes%
}
\Omega B\Cob_0^+ &\simeq \widetilde{\calQ S^0} &\text{and}&& \Omega B\Cob_0^+ (\bbR^2) &\simeq \widetilde{\Omega^2 S^2} ,\label{eqn:lim:space:4}
\end{align}
where we are taking double covers componentwise.

So in a very special case, and up to delooping once, this identifies the homotopy type of a cobordism category of manifolds with some kind of \emph{non-local} structure.
%%%%%%%%%%%%%%%%%%%%%%%%%%%%%%%%%%%%%%%%%%%%%%%%%%%%%%%%%%%%%%%%%%%%%%%%%%%%%%%%%%

%%%%%%%%%%%%%%%%%%%%%%%%%%%%%%%%%%%%%%%%%%%%%%%%%%%%%%%%%%%%%%%%%%%%%%%%%%%%%%%%%%
\section{Failure of injectivity}\label{sec:failure:of:injectivity}

In this section we elaborate on one way in which the oriented case is harder to deal with than the unordered case: the failure of the stabilisation maps to be injective on homology in general. In \S \ref{subsec:inj:of:s} we recall how injectivity-on-homology can be proved in the unordered case, and in \S \ref{subsec:failure:of:injectivity} explain why the analogous argument breaks down in the oriented case. Then in \S \ref{subsec:counterexamples:to:injectivity} we give some explicit examples demonstrating non-injectivity of $s_* \colon H_* C_n^+ (M,X) \longrightarrow H_* C_{n+1}^+ (M,X)$.

%%%%%%%%%%%%%%%%%%%%%%%%%%%%%%%%%%%%%%%%%%%%%%%%%%%
\subsection{Injectivity in the unordered case}\label{subsec:inj:of:s}
%%%%%%%%%%%%%%%%%%%%%%%%%%%%%%%%%%%%%%%%%%%%%%%%%%%

The stabilisation maps $s$ are split-injective on homology in all degrees in the case of unordered configuration spaces. This can be shown with the help of the following lemma proved by Dold (and used earlier by Nakaoka in \cite{Nakaoka1960}):
\begin{lem}[Lemma 2 of \cite{Dold1962}]\label{lem:inj:of:s}
Given a sequence of abelian groups and homomorphisms $0\to A_1 \xrightarrow{s_1} A_2 \xrightarrow{s_2} \cdots$, if there are `transfer' maps $\tau_{k,n} \colon A_n \to A_k$ $(1\leq k\leq n)$ satisfying
\begin{align*}
\tau_{n,n} &= \id &\text{and}&& \tau_{k,n} &= \tau_{k,n+1} \circ s_n \quad \mathrm{mod\; im} (s_{k-1}),
\end{align*}
then every $s_n$ is split-injective.
\end{lem}
\noindent If the abelian groups are in fact $\bbQ$-vector spaces, then it suffices to find transfer maps going just one step back:
\begin{coro}\label{coro:inj:of:s:Q}
Given a sequence of $\bbQ$-vector spaces $0\to A_1 \xrightarrow{s_1} A_2 \xrightarrow{s_2} \cdots$, if there are `transfer' maps $t_n \colon A_n \to A_{n-1}$ $(n\geq 1)$ satisfying
\begin{equation*}
t_{n+1} \circ s_n = \id + s_{n-1} \circ t_n ,
\end{equation*}
then every $s_n$ is split-injective.
\end{coro}
\begin{proof}
Define $\tau_{k,n} \coloneqq \frac{1}{(n-k)!}\, t_{k+1} \circ \cdots \circ t_n$ for $1 \leq k < n$, so that $\tau_{k,n+1} \circ s_n = \tau_{k,n} + s_{k-1} \circ \tau_{k-1,n}$, and apply Lemma \ref{lem:inj:of:s}.
\end{proof}
\noindent Lemma \ref{lem:inj:of:s} can be applied to prove injectivity of $s_*$ in the unordered case by defining
\begin{equation*}
\tau_{k,n} \colon SP^\infty C_n (M,X) \longrightarrow SP^\infty C_k (M,X)
\end{equation*}
to take an $n$-point configuration in $M$ to the sum of its $\tvect{n}{k}$ different $k$-point subsets (c.f.\ proof of Theorem 4.5 in \cite{McDuff1975}). This uses the Dold-Thom theorem: $\pi_* SP^\infty \cong H_*$ for $*\geq 1$ \cite{Dold:Thom1958}.

%%%%%%%%%%%%%%%%%%%%%%%%%%%%%%%%%%%%%%%%%%%%%%%%%%%
\subsection{Failure of injectivity in the oriented case}\label{subsec:failure:of:injectivity}
%%%%%%%%%%%%%%%%%%%%%%%%%%%%%%%%%%%%%%%%%%%%%%%%%%%

This trick doesn't work for \emph{oriented} configuration spaces, however, since there is no way for an oriented $n$-point configuration to induce an orientation on a $k$-point subset unless $k=n-1$. If we instead define $\tau_{k,n}$ to take an oriented $n$-point configuration to the sum of all its oriented $k$-point subsets --- with \emph{either} orientation --- then $\tau_{n,n} = \id + \nu$, so the first hypothesis of Lemma \ref{lem:inj:of:s} is not satisfied.

Alternatively, we could try to just prove injectivity on rational homology using Corollary \ref{coro:inj:of:s:Q}, since this only requires maps removing a \emph{single} configuration point, and in this case there is an induced orientation on the subconfiguration. However, defining
\begin{equation*}
t_n \colon SP^\infty C_n^+ (M,X) \longrightarrow SP^\infty C_{n-1}^+ (M,X)
\end{equation*}
to take an oriented $n$-point configuration to the sum of its $n$ different $(n-1)$-point subsets (with their \emph{induced} orientations) results in equations $t_{n+1} \circ s_n = \id + \nu \circ s_{n-1} \circ t_n$, so the hypothesis of Corollary \ref{coro:inj:of:s:Q} is not quite satisfied.

%%%%%%%%%%%%%%%%%%%%%%%%%%%%%%%%%
\subsection{Counterexamples}\label{subsec:counterexamples:to:injectivity}
%%%%%%%%%%%%%%%%%%%%%%%%%%%%%%%%%

As mentioned in Remark \ref{rmk:global:data} in the Introduction, the calculations in \cite{Guest:et:al1996} provide counterexamples to injectivity of the maps $s_*$ in the case of oriented configuration spaces. The same examples also serve to show that a stability slope of $\frac13$ is the best possible in the oriented case.

First, though, we mention a much simpler counterexample:
\begin{ceg}
The simplest counterexample to injectivity of $s_*$ is the map $H_1 (C_4^+ (\bbR^\infty ,pt)) \to H_1 (C_5^+ (\bbR^\infty ,pt))$, which is $H_1 A_4 \to H_1 A_5$, which is $\bbZ / 3 \to 0$. This is the colimit of the maps $s_* \colon H_1 (C_4^+ (\bbR^k ,pt)) \to H_1 (C_5^+ (\bbR^k ,pt))$ as $k\to\infty$, and injectivity is \emph{preserved} by taking such a colimit, so this provides counterexamples: $s_*$ must be non-injective for infinitely many values of $k$.
\end{ceg}

%%%%%%%%%%%%%%%%%%%%%%%%%%%%%%%%%
\subsubsection*{The \textup{\cite{Guest:et:al1996}} calculations.}
%%%%%%%%%%%%%%%%%%%%%%%%%%%%%%%%%

For an odd prime $p$, there is a splitting
\begin{equation*}
H_q (C_n^+ (M,X);\bbF_p) \cong H_q (C_n (M,X);\bbF_p) \oplus H_q (C_n (M,X);\bbF_p^{(-1)}),
\end{equation*}
where on the right summand, $\pi_1 C_n^+ (M,X) \leq \pi_1 C_n (M,X)$ acts on $\bbF_p$ by the identity, and its complement acts by multiplication by $-1$. Correspondingly, the stabilisation map $s_*$ splits into two summands: one is the stabilisation map from the unordered case, which \emph{is} split-injective by \S \ref{subsec:inj:of:s} above, and the other is the map induced by the stabilisation map (from the unordered case) on \emph{twisted} homology:
\begin{equation}\label{eqn:s:on:twisted:homology}
H_q (C_n (M,X);\bbF_p^{(-1)}) \longrightarrow H_q (C_{n+1} (M,X);\bbF_p^{(-1)}).
\end{equation}

The calculations in \cite{Guest:et:al1996} use a result of B\"{o}digheimer-Cohen-Milgram-Taylor \cite[Corollary 8.4]{BCT1989,BCM1993} to write this (under some conditions) in terms of the homology of iterated loopspaces of spheres, and then apply the Snaith splitting theorem \cite{Snaith1974} and knowledge of the structure of $H_* \Omega^2 S^3$ to analyse the result. Going through their calculations one can see that the map \eqref{eqn:s:on:twisted:homology} is the map $\bbF_p \to 0$ for $M$ any connected open surface, $X=pt$, and
\begin{equation*}
(n,q) = (\lambda p+1, \lambda (p-2)) \quad \text{for any } \lambda\geq 1
\end{equation*}
(although they state their result in slightly less generality).

This provides an infinite family of counterexamples to injectivity at each odd prime, and taking $p=3$ also provides counterexamples to demonstrate that $\frac13$ is the best possible stability slope for oriented configuration spaces, as mentioned in Remark \ref{rmk:best:possible:stability:slope} in the Introduction.
%%%%%%%%%%%%%%%%%%%%%%%%%%%%%%%%%%%%%%%%%%%%%%%%%%%%%%%%%%%%%%%%%%%%%%%%%%%%%%%%%%

%%%%%%%%%%%%%%%%%%%%%%%%%%%%%%%%%%%%%%%%%%%%%%%%%%%%%%%%%%%%%%%%%%%%%%%%%%%%%%%%%%
\setcounter{section}{0}
\renewcommand{\thesection}{{\scshape Appendix} \Alph{section}}
\section{Proof of the factorisation lemma}\label{appendix:proof:of:the:factorisation:lemma}
\renewcommand{\thesection}{\Alph{section}}

In this appendix we prove the general factorisation lemma which is used in the proof of the Main Theorem in \S \ref{sec:proof:of:the:main:theorem}.
\begin{proof}[Proof of Lemma \ref{lem:factorisation} \textup{(}page \pageref{lem:factorisation}\textup{)}]
We have a square with a given homotopy filling it:
\begin{center}
\begin{tikzpicture}
[x=1.2mm,y=1.2mm,>=stealth']
\node (tl) at (0,10) {$A$};
\node (tr) at (20,10) {$X$};
\node (bl) at (0,0) {$B$};
\node (br) at (20,0) {$Y$};
\draw [->] (tl) to node[above,font=\small]{$i$} (tr);
\draw [->] (bl) to node[below,font=\small]{$j$} (br);
\draw [->] (tl) to node[left,font=\small]{$f$} (bl);
\draw [->] (tr) to node[right,font=\small]{$g$} (br);
\node at (10,5) [font=\small] {$H$};
\end{tikzpicture}
\end{center}
and also know that there exists a \emph{split} homotopy filling the same square:
\begin{center}
\begin{tikzpicture}
[x=1.2mm,y=1.2mm,>=stealth']
\node (tl) at (0,10) {$A$};
\node (tr) at (20,10) {$X$};
\node (bl) at (0,0) {$B$};
\node (br) at (20,0) {$Y$};
\draw [->] (tl) to node[above,font=\small]{$i$} (tr);
\draw [->] (bl) to node[below,font=\small]{$j$} (br);
\draw [->] (tl) to node[left,font=\small]{$f$} (bl);
\draw [->] (tr) to node[right,font=\small]{$g$} (br);
\draw [->] (tr) to node[above,font=\footnotesize]{$d$} (bl);
\node at (5,7) [font=\small] {$F_1$};
\node at (15,3) [font=\small] {$F_2$};
\end{tikzpicture}
\end{center}
We want to find a factorisation of $CH_* \colon \widetilde{H}_* Ci \to \widetilde{H}_* Cj$, so we begin by factorising the map $CH\colon Ci \to Cj$ itself. Schematically, $CH$ looks like
\begin{center}
\begin{tikzpicture}
[scale=1,x=1mm,y=1mm,>=stealth']
\newcommand{\xoffset}{30mm}
\draw (0,0)--(10,0);
\draw (3,0)--(5,7)--(7,0);
\draw (4,3.5)--(6,3.5);
\draw[->] (15,3)--(25,3);
\begin{scope}[xshift=\xoffset]
\draw (0,0)--(10,0);
\draw (3,0)--(5,7)--(7,0);
\end{scope}
\node at (-15,3) {$Ci = X\cup_i CA =$};
\node at (55,3) {$= Y\cup_j CB = Cj$};
\end{tikzpicture}
\end{center}
where the top part of $CA$ is mapped to (all of) $CB$ by $f$, levelwise, the middle section $A\times I$ is mapped to $Y$ by the homotopy $H$, and $X$ is mapped to $Y$ by $g$. We will factorise this as follows:
\begin{equation}\label{eqn:big:diagram}
\centering
\begin{split}
\begin{tikzpicture}
[scale=1,x=1mm,y=1mm,>=stealth']
\newcommand{\xoffset}{23mm}
\node (Ci) at (-20,0) {$Ci$};
\node (CXCA) at (10,0) {$CX\cup_i CA =$};
\node (Cj) at (80,0) {$Cj$};
\node (SA) at (10,-15) {$\Sigma A$};
\node (Y) at (80,-15) {$Y$};
\node (SxACA) at (30,-25) {$(S^1 \times A)\cup CA$};
\node (SxA) at (60,-25) {$S^1 \times A$};
\incl{(Ci)}{(CXCA)}
\draw[->>] (CXCA) to node[left,font=\small]{$\mathrm{collapse}\; CX$} (SA);
\draw[->>] (SA) to (SxACA);
\draw[->,dashed] (SxACA) to node[below=-1pt,font=\small]{$\delta$} (SxA);
\draw[->] (SxA) to node[below right=-2pt and -2pt,font=\small]{$\gamma$} (Y);
\inclusion{right}{$\mathrm{inc}$}{(Y)}{(Cj)}
\draw[->] (SxACA) to node[below right=-2pt and -2pt,font=\small]{$\widetilde{\beta}$} (Cj);
\begin{scope}[xshift=\xoffset]
\draw (0,0)--(10,0)--(5,-5)--cycle;
\draw (3,0)--(5,7)--(7,0);
\draw (4,3.5)--(6,3.5);
\end{scope}
\draw[->] (35,0.5) to[out=10,in=170] node[above=-1pt,font=\small]{$\alpha$} (Cj);
\draw[->] (35,-0.5) to[out=-10,in=190] node[below=-1pt,font=\small]{$\beta$} (Cj);
\node at (56,0) {\rotatebox{-90}{$\simeq$}};
\node at (68,-15) [draw,circle,inner sep=0.5pt,font=\small] {$*$};
\end{tikzpicture}
\end{split}
\end{equation}
This requires some explanation: we will define $\alpha$ so that the map across the top is $CH$ (so $\alpha$ is an extension of $CH$). Then we will homotope $\alpha$ to a map $\beta$ which descends to $\widetilde{\beta} \colon (S^1 \times A)\cup CA \longrightarrow Cj$ when you collapse $CX$ and then glue a small cone at the top of $\Sigma A$ to a small cone at the bottom\footnote{$\Sigma A$ means the \emph{unreduced} suspension here.}. Then we will show that $\widetilde{\beta}$ factors through the square
\tikz{\node[draw,circle,inner sep=0.5pt,font=\small]{$*$};}
as indicated (a dotted arrow denotes a map which is only defined on homology).

The composition $Ci \longrightarrow \Sigma A$ is the map in the Puppe sequence inducing the boundary map in \eqref{eqn:exact:ladder}, so this will prove the first half of the lemma, with $z_H$ induced by the composition
\begin{equation*}
\Sigma A \twoheadrightarrow (S^1 \times A)\cup CA \dashrightarrow S^1 \times A \xrightarrow{\gamma} Y.
\end{equation*}
Firstly, we define $\alpha$ and $\beta$ as follows: Each region is mapped to a part of $Cj = Y\cup_j CB$ by the map or homotopy indicated; $*$ means it is sent to the tip of the cone $CB$; shaded regions have target $Y$, whereas unshaded regions are mapped (levelwise) to $CB$. By temporary abuse of notation, $Cf$ in this diagram means the map $CA \to CB$ which is levelwise $f$; similarly for $Cd$.
\begin{center}
\begin{tikzpicture}
[scale=1.5,x=1mm,y=1mm,>=stealth',font=\small]
\newcommand{\xoffseta}{0mm}
\newcommand{\xoffsetb}{25mm}
\newcommand{\xoffsetc}{50mm}
\node at (-5,1) [font=\normalsize] {$\alpha \coloneqq$};
\node at (68,1) [font=\normalsize] {$\eqqcolon \beta$};
\node at (18,1) [font=\normalsize] {$\simeq$};
\node at (45,1) [font=\normalsize] {$\simeq$};
\begin{scope}[xshift=\xoffseta]
\newcommand{\n}{(5,8)}
\newcommand{\s}{(5,-5)}
\newcommand{\w}{(0,0)}
\newcommand{\e}{(10,0)}
\newcommand{\la}{(3,0)}
\newcommand{\lb}{(4,4)}
\newcommand{\ra}{(7,0)}
\newcommand{\rb}{(6,4)}
\newcommand{\lz}{(2.5,-2.5)}
\newcommand{\rz}{(7.5,-2.5)}
\fill[black!20] \lb--\rb--\ra--\la--cycle;
\fill[black!20] \w--\e--\rz--\lz--cycle;
\draw \w--\e--\s--cycle;
\draw \la--\n--\ra;
\draw \lb--\rb;
\draw \lz--\rz;
\newcommand{\rega}{(8,6)}
\newcommand{\regb}{(8.5,2)}
\newcommand{\regc}{(11,-1.5)}
\newcommand{\regd}{(9,-4)}
\node at \rega {$Cf$};
\node at \regb {$H$};
\node at \regc {$F_2$};
\node at \regd {$Cd$};
\end{scope}
\begin{scope}[xshift=\xoffsetb]
\newcommand{\n}{(5,8)}
\newcommand{\s}{(5,-5)}
\newcommand{\w}{(0,-2)}
\newcommand{\e}{(10,-2)}
\newcommand{\la}{(3,-2)}
\newcommand{\lb}{($ (3,-2) !0.25! (5,8) $)}
\newcommand{\lc}{($ (3,-2) !0.5! (5,8) $)}
\newcommand{\ld}{($ (3,-2) !0.75! (5,8) $)}
\newcommand{\ra}{(7,-2)}
\newcommand{\rb}{($ (7,-2) !0.25! (5,8) $)}
\newcommand{\rc}{($ (7,-2) !0.5! (5,8) $)}
\newcommand{\rd}{($ (7,-2) !0.75! (5,8) $)}
\fill[black!20] \ld--\rd--\rb--\lb--cycle;
\draw \w--\e--\s--cycle;
\draw \la--\n--\ra;
\draw \lb--\rb;
\draw \lc--\rc;
\draw \ld--\rd;
\newcommand{\rega}{(8,6.5)}
\newcommand{\regb}{(8,4.25)}
\newcommand{\regc}{(10,2)}
\newcommand{\regd}{(12,-0.5)}
\newcommand{\rege}{(5,-3.4)}
\node at \rega {$Cf$};
\node at \regb {$H$};
\node at \regc {$F_2 \circ i$};
\node at \regd {$(d\circ i)\! \times\! I$};
\node at \rege {$*$};
\end{scope}
\begin{scope}[xshift=\xoffsetc]
\newcommand{\n}{(5,8)}
\newcommand{\s}{(5,-5)}
\newcommand{\w}{(0,-2)}
\newcommand{\e}{(10,-2)}
\newcommand{\la}{(3,-2)}
\newcommand{\lb}{($ (3,-2) !0.2! (5,8) $)}
\newcommand{\lc}{($ (3,-2) !0.4! (5,8) $)}
\newcommand{\ld}{($ (3,-2) !0.6! (5,8) $)}
\newcommand{\lf}{($ (3,-2) !0.8! (5,8) $)}
\newcommand{\ra}{(7,-2)}
\newcommand{\rb}{($ (7,-2) !0.2! (5,8) $)}
\newcommand{\rc}{($ (7,-2) !0.4! (5,8) $)}
\newcommand{\rd}{($ (7,-2) !0.6! (5,8) $)}
\newcommand{\rf}{($ (7,-2) !0.8! (5,8) $)}
\fill[black!20] \lf--\rf--\rb--\lb--cycle;
\draw \w--\e--\s--cycle;
\draw \la--\n--\ra;
\draw \lb--\rb;
\draw \lc--\rc;
\draw \ld--\rd;
\draw \lf--\rf;
\newcommand{\rega}{(8,7)}
\newcommand{\regb}{(8,5)}
\newcommand{\regc}{(10,3)}
\newcommand{\regd}{(11,1)}
\newcommand{\rege}{(11,-1)}
\newcommand{\regf}{(5,-3.4)}
\node at \rega {$Cf$};
\node at \regb {$H$};
\node at \regc {$F_2 \circ i$};
\node at \regd {$j \circ F_1$};
\node at \rege {$f\! \times\! I$};
\node at \regf {$*$};
\end{scope}
\end{tikzpicture}
\end{center}
Intuitively: the left homotopy `pulls $\alpha$ upwards' to obtain the map pictured in the middle, then the right homotopy gradually morphs the levelwise-$(d\!\circ\! i)$ part of this map into the homotopy $F_1$, and then `stretches' one end of it into levelwise-$f$.

It is clear from its definition that $\beta$ descends to a map $\widetilde{\beta}$ as described above; we define $\gamma$ to be the restriction of $\widetilde{\beta}$ to $S^1 \times A$:
\begin{center}
\begin{tikzpicture}
[scale=1.5,x=1mm,y=1mm,>=stealth',font=\small]
\newcommand{\xoffseta}{0mm}
\newcommand{\xoffsetb}{33mm}
\newcommand{\yoffsetb}{1.25mm}
\node at (-5,5) [font=\normalsize] {$\widetilde{\beta} =$};
\node at (19,5) [font=\normalsize] {;};
\node at (28,5) [font=\normalsize] {$\gamma \coloneqq$};
\begin{scope}[xshift=\xoffseta]
\newcommand{\n}{(5,10)}
\newcommand{\la}{(0,0)}
\newcommand{\lb}{(0,2.5)}
\newcommand{\lc}{(0,5)}
\newcommand{\ld}{(0,7.5)}
\newcommand{\ra}{(10,0)}
\newcommand{\rb}{(10,2.5)}
\newcommand{\rc}{(10,5)}
\newcommand{\rd}{(10,7.5)}
\newcommand{\md}{(5,7.5)}
\newcommand{\ma}{(5,0)}
\fill[black!20] \la--\ld--\rd--\ra--cycle;
\draw \la--\ld--\rd--\ra--cycle;
\draw[->] \la--\ma;
\draw[->] \ld--\md;
\draw \ld--\n--\rd;
\draw \lb--\rb;
\draw \lc--\rc;
\newcommand{\rega}{(12,8.75)}
\newcommand{\regb}{(12,6.25)}
\newcommand{\regc}{(14,3.75)}
\newcommand{\regd}{(14,1.25)}
\node at \rega {$Cf$};
\node at \regb {$H$};
\node at \regc {$F_2 \circ i$};
\node at \regd {$j \circ F_1$};
\end{scope}
\begin{scope}[xshift=\xoffsetb, yshift=\yoffsetb]
\newcommand{\la}{(0,0)}
\newcommand{\lb}{(0,2.5)}
\newcommand{\lc}{(0,5)}
\newcommand{\ld}{(0,7.5)}
\newcommand{\ra}{(10,0)}
\newcommand{\rb}{(10,2.5)}
\newcommand{\rc}{(10,5)}
\newcommand{\rd}{(10,7.5)}
\newcommand{\md}{(5,7.5)}
\newcommand{\ma}{(5,0)}
\fill[black!20] \la--\ld--\rd--\ra--cycle;
\draw \la--\ld--\rd--\ra--cycle;
\draw[->] \la--\ma;
\draw[->] \ld--\md;
\draw \lb--\rb;
\draw \lc--\rc;
\newcommand{\regb}{(12,6.25)}
\newcommand{\regc}{(14,3.75)}
\newcommand{\regd}{(14,1.25)}
\node at \regb {$H$};
\node at \regc {$F_2 \circ i$};
\node at \regd {$j \circ F_1$};
\end{scope}
\end{tikzpicture}
\end{center}
Now we need to construct the map $\delta$: This comes from the split cofibration sequence
\begin{center}
\begin{tikzpicture}
[scale=1,x=1mm,y=1mm,>=stealth']
\node (A) at (0,0) {$A$};
\node (SxA) at (20,0) {$S^1 \times A$};
\node (SxACA) at (52,0) {$(S^1 \times A)\cup CA$.};
\incl{(A)}{(SxA)}
\inclusion{below left}{$\varepsilon$}{(SxA)}{(SxACA)}
\draw[->>] (SxA) to[out=160,in=20] (A.north east);
\draw[->,dashed] (SxACA.north west) to[out=160,in=20] node[above=-1pt,font=\small]{$\delta$} (SxA);
\end{tikzpicture}
\end{center}
We have an \emph{actual} splitting $A \leftarrow S^1 \times A$, which induces a splitting on homology, which implies the existence of a splitting $S^1 \times A \dashleftarrow (S^1 \times A)\cup CA$ on homology.

Since we defined $\gamma$ to be the restriction of $\widetilde{\beta}$, we have $(\mathrm{inc})\circ \gamma = \widetilde{\beta}\circ \varepsilon$. But $\delta$ is a splitting on homology, so $\varepsilon_* \circ \delta = \mathrm{id}$. Hence
\begin{align*}
(\mathrm{inc})_* \circ \gamma_* \circ \delta &= \widetilde{\beta}_* \circ \varepsilon_* \circ \delta \\
&= \widetilde{\beta}_* ,
\end{align*}
so the square
\tikz{\node[draw,circle,inner sep=0.5pt,font=\small]{$*$};}
commutes on homology, as required. This completes the proof of the first half of the lemma.

Now, the map $z_H$ was constructed as the composition
\begin{equation*}
\widetilde{H}_* \Sigma A = \widetilde{H}_{*-1} A \longrightarrow \widetilde{H}_* \bigl( (S^1 \times A)\cup CA \bigr) \xrightarrow{\;\delta\;} \widetilde{H}_* (S^1 \times A) \xrightarrow{\gamma_*} \widetilde{H}_* Y
\end{equation*}
induced by the three maps along the bottom of diagram \eqref{eqn:big:diagram}. As defined above, $\gamma$ is the composition of the homotopy $H$ and the split homotopy $(j\circ F_1) * (F_2 \circ i)$. Hence to prove the second half of the lemma, it just remains to show that the composition of the first two maps is the inclusion coming from the K\"{u}nneth splitting for $\widetilde{H}_* (S^1 \times A)$.

This can be seen as follows: Using the homotopy equivalence $(S^1 \times A)\cup CA \simeq \Sigma A \vee S^1$, the K\"{u}nneth splitting, and the suspension isomorphism we identify:
\begin{align*}
\widetilde{H}_* \Sigma A &= \widetilde{H}_{*-1} A \\
\widetilde{H}_* \bigl( (S^1 \times A)\cup CA \bigr) &= \widetilde{H}_{*-1} A \oplus \widetilde{H}_* S^1 \\
\widetilde{H}_* (S^1 \times A) &= \widetilde{H}_{*-1} A \oplus H_{*-1} (pt) \oplus \widetilde{H}_* A .
\end{align*}

Analysing the map on homology induced by $\varepsilon$ carefully, we see that under this identification it sends $\widetilde{H}_{*-1} A$ to itself by the identity, $H_{*-1} (pt)$ isomorphically to $\widetilde{H}_* S^1$, and $\widetilde{H}_* A$ to $0$. Hence its right-inverse $\delta$ must send $\widetilde{H}_{*-1} A$ to itself by the identity, and $\widetilde{H}_* S^1$ isomorphically to $H_{*-1} (pt)$.

Under the identification $(S^1 \times A)\cup CA \simeq \Sigma A \vee S^1$, the map $\Sigma A \twoheadrightarrow (S^1 \times A)\cup CA$ becomes the inclusion $\Sigma A \hookrightarrow \Sigma A \vee S^1$, so on homology it induces the inclusion of the direct summand $\widetilde{H}_{*-1} A \hookrightarrow \widetilde{H}_{*-1} A \oplus \widetilde{H}_* S^1$.

Hence overall the composition $\widetilde{H}_{*-1} A \longrightarrow \widetilde{H}_* \bigl( (S^1 \times A)\cup CA \bigr) \longrightarrow \widetilde{H}_* (S^1 \times A)$ is the inclusion of the direct summand
\begin{equation*}
\widetilde{H}_{*-1} A \quad \hookrightarrow \quad \widetilde{H}_{*-1} A \oplus H_{*-1} (pt) \oplus \widetilde{H}_* A
\end{equation*}
into the K\"{u}nneth splitting for $\widetilde{H}_* (S^1 \times A)$.
\end{proof}
%%%%%%%%%%%%%%%%%%%%%%%%%%%%%%%%%%%%%%%%%%%%%%%%%%%%%%%%%%%%%%%%%%%%%%%%%%%%%%%%%%

%%%%%%%%%%%%%%%%%%%%%%%%%%%%%%%%%%%%%%%%%%%%%%%%%%%%%%%%%%%%%%%%%%%%%%%%%%%%%%%%%%
\renewcommand{\thesection}{{\scshape Appendix} \Alph{section}}
\section{Spectral sequences from $\Delta$-spaces}\label{appendix:spectral:sequences:from:dspaces}
\renewcommand{\thesection}{\Alph{section}}

The aim of this appendix is to prove Proposition \ref{prop:dss} --- the construction of a spectral sequence associated to a map of augmented $\Delta$-spaces. We will work up to this gradually, starting with the spectral sequence associated to a $\Delta$-space, and will use the general construction recalled below.

%%%%%%%%%%%%%%%%%%%%%%%%%%%%%%%%%%%%%%%%%%%%%%%%%%%
\subsection{General construction}
%%%%%%%%%%%%%%%%%%%%%%%%%%%%%%%%%%%%%%%%%%%%%%%%%%%

Recall the following construction (see for example \cite[chapter 7]{Mosher:Tangora1968}): given a filtration
\begin{equation*}
\varnothing = X_{-1} \subseteq X_0 \subseteq \cdots \quad \cdots \subseteq X_n \subseteq \cdots \quad \subseteq X
\end{equation*}
of a space $X$ such that
\begin{equation}\label{eqn:two:conditions}
\bigcup_{n\geq 0} X_n = X \qquad \text{and} \qquad H_* (X_n, X_{n-1}) = 0 \text{ for } *<n,
\end{equation}
the filtered chain complex $C_* (X)$ induces a first quadrant spectral sequence
\begin{equation*}
E_{s,t}^1 \; \cong \; H_{s+t} (X_s, X_{s-1}) \quad \Rightarrow \quad H_* (X).
\end{equation*}
The first differential in this spectral sequence is the boundary map for the pair $(X_s, X_{s-1})$ composed with the quotient map for the pair $(X_{s-1}, X_{s-2})$.

%%%%%%%%%%%%%%%%%%%%%%%%%%%%%%%%%%%%%%%%%%%%%%%%%%%
\subsection{$\Delta$-spaces}
%%%%%%%%%%%%%%%%%%%%%%%%%%%%%%%%%%%%%%%%%%%%%%%%%%%

We first describe the construction of the spectral sequence associated to a $\Delta$-space $Y_\bullet$. Filter $X =\; \geomr{Y_\bullet}$ by its skeleta,
\begin{equation*}
X_n = \geomrsk{Y_\bullet}{n}\; = \coprod_{n\geq k\geq 0} Y_k \times \Delta^k \rightquotient \sim .
\end{equation*}
The filtration quotients are $X_n / X_{n-1} \cong (Y_n)_+ \wedge S^n$, and the inclusions $X_{n-1} \hookrightarrow X_n$ are cofibrations, so
\begin{align}\label{eqn:identification:of:E2:page}
\begin{split}
H_{s+t} (X_s, X_{s-1}) \; &\cong \; \widetilde{H}_{s+t} ((Y_s)_+ \wedge S^s) \\
&\cong \; \widetilde{H}_t ((Y_s)_+) \; = \; H_t (Y_s).
\end{split}
\end{align}
This is zero for $t<0$, so \eqref{eqn:two:conditions} is satisfied and we get the spectral sequence
\begin{equation*}
E_{s,t}^1 \; \cong \; H_t (Y_s) \quad \Rightarrow \quad H_* (\geomr{Y_\bullet}).
\end{equation*}
The formula for the boundary map of the pair $(X_s, X_{s-1})$, under the identification \eqref{eqn:identification:of:E2:page}, gives the first differential as the alternating sum of $H_t$ of the face maps $Y_s \to Y_{s-1}$.

%%%%%%%%%%%%%%%%%%%%%%%%%%%%%%%%%%%%%%%%%%%%%%%%%%%
\subsection{Augmented $\Delta$-spaces}
%%%%%%%%%%%%%%%%%%%%%%%%%%%%%%%%%%%%%%%%%%%%%%%%%%%

For an augmented $\Delta$-space $Y_\bullet$, we filter the mapping cone $X = C(\geomr{Y_\bullet} \to Y_{-1})$ by
\begin{equation*}
X_n = C(\geomrsk{Y_\bullet}{n-1} \to Y_{-1})
\end{equation*}
for $n\geq 1$ and $X_0 = Y_{-1} \sqcup \lbrace \text{tip of cone} \rbrace$. The filtration quotients are now $X_n / X_{n-1} \cong (Y_{n-1})_+ \wedge S^n$ for $n\geq 1$, so similarly to before we have
\begin{equation*}
H_{s+t} (X_s, X_{s-1}) \; \cong \; H_t (Y_{s-1}),
\end{equation*}
except with an extra $\bbZ$-summand when $s=t=0$. Again this satisfies \eqref{eqn:two:conditions}, so we have a spectral sequence converging from this $E^1$ page to $H_* (C(\geomr{Y_\bullet} \to Y_{-1}))$. Removing the extra $\bbZ$-summand from the $E^1$ page turns the limit into the \emph{reduced} homology, so if we also regrade $s \mapsto s+1$ we obtain the spectral sequence
\begin{equation*}
E_{s,t}^1 \; \cong \; H_t (Y_s) \quad \Rightarrow \quad \widetilde{H}_{*+1} (C(\geomr{Y_\bullet} \to Y_{-1})),
\end{equation*}
which lives in $\lbrace s\geq -1, t\geq 0 \rbrace$. Again, $d^1$ is the alternating sum of the maps on $H_t$ induced by the face maps --- in particular, the differential $E_{0,t}^1 \to E_{-1,t}^1$ is $H_t$ of the augmentation map.

%%%%%%%%%%%%%%%%%%%%%%%%%%%%%%%%%%%%%%%%%%%%%%%%%%%
\subsection{Basepoints}\label{subsec:basepoints}
%%%%%%%%%%%%%%%%%%%%%%%%%%%%%%%%%%%%%%%%%%%%%%%%%%%

Now we will introduce basepoints: Let $Y_\bullet$ be an augmented $\Delta$-object in the category of pointed spaces. The pointed geometric realisation $\geomrp{Y_\bullet}$ is $\coprod_{n\geq 0} Y_n \times \Delta^n$ quotiented out by $\coprod_{n\geq 0} * \times \Delta^n$ and then by the face relations, and again there is an induced map $\geomrp{Y_\bullet}\, \to Y_{-1}$.

Filter $X = C(\geomrp{Y_\bullet} \to Y_{-1})$ by $X_n = C(\geomrpsk{Y_\bullet}{n-1} \to Y_{-1})$ for $n\geq 1$ and $X_0 = Y_{-1}$. The filtration quotients are $X_n / X_{n-1} \cong Y_{n-1} \wedge S^n$ for $n\geq 1$, so
\begin{equation*}
H_{s+t} (X_s, X_{s-1}) \; \cong \; \widetilde{H}_t (Y_{s-1}),
\end{equation*}
except again with an extra $\bbZ$-summand when $s=t=0$. This satisfies \eqref{eqn:two:conditions}, so removing the extra $\bbZ$-summand and regrading as before we get a spectral sequence
\begin{equation*}
E_{s,t}^1 \; \cong \; \widetilde{H}_t (Y_s) \quad \Rightarrow \quad \widetilde{H}_{*+1} (C(\geomrp{Y_\bullet} \to Y_{-1}))
\end{equation*}
in $\lbrace s\geq -1, t\geq 0 \rbrace$.

%%%%%%%%%%%%%%%%%%%%%%%%%%%%%%%%%%%%%%%%%%%%%%%%%%%
\subsection{Maps of augmented $\Delta$-spaces}\label{subsec:proof:of:prop:dss}
%%%%%%%%%%%%%%%%%%%%%%%%%%%%%%%%%%%%%%%%%%%%%%%%%%%

We can now deduce Proposition \ref{prop:dss} from the last construction above.
\begin{proof}[Proof of Proposition \ref{prop:dss}]
We are given a map of augmented $\Delta$-spaces $Y_\bullet \to Z_\bullet$. Since homotopy colimits commute,
\begin{equation*}
\hocofib (\geomr{Y_\bullet} \to \geomr{Z_\bullet}) \; \simeq \; \geomrp{\hocofib (Y_\bullet \to Z_\bullet)},
\end{equation*}
i.e.\ $C(\geomr{Y_\bullet} \to \geomr{Z_\bullet}) \; \simeq \; \geomrp{C(Y_\bullet \to Z_\bullet)}$, where the \emph{pointed} realisation appears on the right since mapping cones are naturally pointed spaces. The face and augmentation maps of $Y_\bullet$ and $Z_\bullet$ give $C(Y_\bullet \to Z_\bullet)$ the structure of an augmented $\Delta$-object in the category of pointed spaces, so we may apply the construction of \ref{subsec:basepoints} above. This yields a spectral sequence in $\lbrace s\geq -1, t\geq 0 \rbrace$ with
\begin{equation*}
E_{s,t}^1 \; \cong \; \widetilde{H}_t (C(Y_s \to Z_s)),
\end{equation*}
and converging to $\widetilde{H}_{*+1}$ of the mapping cone of
\begin{equation*}
C(\geomr{Y_\bullet} \to \geomr{Z_\bullet}) \simeq \geomrp{C(Y_\bullet \to Z_\bullet)} \longrightarrow C(Y_{-1} \to Z_{-1}),
\end{equation*}
which is the double mapping cone $C^2 (Y_\bullet \to Z_\bullet)$ of the square \eqref{eqn:map:of:aug:dspaces}. The first differential can be identified as in the other constructions above.
\end{proof}
%%%%%%%%%%%%%%%%%%%%%%%%%%%%%%%%%%%%%%%%%%%%%%%%%%%%%%%%%%%%%%%%%%%%%%%%%%%%%%%%%%

%%%%%%%%%%%%%%%%%%%%%%%%%%%%%%%%%%%%%%%%%%%%%%%%%%%%%%%%%%%%%%%%%%%%%%%%%%%%%%%%%%
\providecommand{\bysame}{\leavevmode\hbox to3em{\hrulefill}\thinspace}
\providecommand{\MR}{\relax\ifhmode\unskip\space\fi MR }
\providecommand{\MRhref}[2]{\href{http://www.ams.org/mathscinet-getitem?mr=#1}{#2}}
\providecommand{\href}[2]{#2}

%%%%%%%%%%%%%%%%%%%%%%%%%%%%%%%%%%%%%%%%%%%%%%%%%%%%%%%%%%%%%%%%%%%%%%%%%%%%%%%%%%

\end{document}